\documentclass[10pt]{amsart}

\usepackage[T1]{fontenc}
\usepackage{amsmath, amsthm, amssymb}
\usepackage{color}
\usepackage{bbm}
\usepackage{mathtools}
\usepackage{bm}
\usepackage{framed}
\usepackage[shortlabels]{enumitem}
\usepackage{mathbbol}
\usepackage{mathrsfs}
\usepackage[hidelinks]{hyperref}
\usepackage{cleveref}
\usepackage[at]{easylist} 

\newtheorem{definition}{Definition}
\newtheorem{theorem}{Theorem}
\newtheorem{lemma}{Lemma}
\newtheorem{proposition}{Proposition}
\newtheorem{remark}{Remark}

\newtheorem{corollary}{Corollary}

\DeclarePairedDelimiter\floor{\lfloor}{\rfloor}

\DeclareMathOperator\supp{supp}

\DeclareMathOperator*\argmin{argmin}

\DeclareMathOperator\cov{Cov}

\DeclareMathOperator\diag{diag}
\DeclareMathOperator\dist{dist}

\DeclareMathOperator\mes{mes}

\newcommand{\E}[1]{\mathbb{E}\left\{ #1 \right\}}
\newcommand{\pk}[1]{\mathbb{P} \left\{ #1 \right\} }
\newcommand{\Ind}[1]{\mathbb{1} \left\{ #1 \right\} }
\newcommand{\Var}[1]{\mathrm{Var} \left\{ #1 \right\} }
\newcommand{\R}{\mathbb{R}}
\newcommand{\N}{\mathbb{N}}

\newcommand{\bb}{\tilde{\bm{b}}}
\newcommand{\Jtan}{J_{\mathrm{tan}}}
\newcommand{\Jsl}{J_{\mathrm{sl}}}
\newcommand{\1}{\mathbb{1}}
\newcommand{\uu}{\check{\bm{u}}}

\allowdisplaybreaks[4]

\makeatletter
\def\namedlabel#1#2{\begingroup
  #2\def\@currentlabel{#2}\phantomsection\label{#1}\endgroup
}
\makeatother 
\begin{document}

\title[]{Asymptotic Behavior of Path Functionals for Vector-Valued Gaussian Processes at High Levels}

\author{Pavel Ievlev}
\address{Pavel Ievlev, Independent researcher}
\email{Pavel.Ievlev@unil.ch, ievlev.pn@gmail.com}

\author{Svyatoslav Novikov}
\address{Svyatoslav Novikov, Department of Actuarial Science, University of Lausanne, UNIL-Dorigny, 1015 Lausanne, Switzerland
}
\email{Svyatoslav.Novikov@unil.ch}

\author{Timofei Shashkov}
\address{Timofei Shashkov, Department of Actuarial Science, University of Lausanne, UNIL-Dorigny, 1015 Lausanne, Switzerland
}
\email{Timofei.Shashkov@unil.ch} 

\begin{abstract}
    We study precise asymptotics for high-level exceedance probabilities of
    path functionals of continuous vector-valued Gaussian processes.  The
    probabilities have the form
    \[
      \mathbb{P}\{\Gamma_{[0,T]}(\check{\bm u}(\bm X-u\bm b))>L_u\},
      \qquad u\to\infty,
    \]
    where \(\bm X\) is a centered \(\mathbb R^d\)-valued Gaussian process and
    \(\Gamma\) belongs to a broad class satisfying natural monotonicity,
    scaling, no-atom, and continuity assumptions.  The class covers classical
    sojourn times, Choquet-type sojourn integrals, area-under-the-curve
    functionals, and, through a local-footprint extension, shrinking-window
    Parisian persistence functionals.

    We obtain exact asymptotics in the stationary case and in the
    non-stationary case when the inverse generalized variance has a unique
    minimizer at the boundary point.  The non-stationary theorem covers the
    regimes \(\alpha<\beta\), \(\alpha=\beta\), and \(\alpha>\beta\), which lead
    respectively to Pickands-type, Piterbarg-type, and deterministic limiting
    constants.  We also derive conditional limit laws for the first exceedance
    time.  The main claims are stated in the body of the paper, while the
    proofs and auxiliary estimates are collected in the appendices.
\end{abstract}

\keywords{Gaussian process; sojourn time; Parisian persistence; Choquet integral;
Vector-valued Gaussian extremes; high exceedance probability;}

\subjclass[2020]{Primary 60G15; secondary 60G70}
 
\date{\today}
\maketitle

\section{Introduction}
\label{sec:introduction}

High excursions of Gaussian processes and fields are most often studied through
a supremum event.  In the scalar case, the asymptotic analysis of
\(
  \mathbb P\{\sup_{t\in[0,T]}X(t)>u\}
\)
goes back to Pickands and forms the basis of the Pickands--Piterbarg theory for
stationary and non-stationary Gaussian processes and fields; see
\cite{Pickands1967,Pickands1969,LeadbetterLindgrenRootzen1983,Piterbarg1996}.
The associated constants and their representations remain an active topic; see,
for example, \cite{DiekerYakir2014}.  A supremum records whether a high
excursion occurs, but not its duration, accumulated size, or geometry.  The
sojourn problem replaces this binary question by an occupation functional of
the excursion set.  Berman's foundational work showed that high-level
sojourns and extremes are closely linked; see
\cite{Berman1982,Berman1992}.  This distinction matters whenever an isolated
spike and a persistent excursion have different interpretations, as in storage,
queueing, reliability, and ruin models.

There are two complementary traditions for studying functionals of Gaussian
excursion sets.  For sufficiently smooth Gaussian fields, local differential
geometry provides a powerful description of high excursions.  The expected
Euler characteristic (EEC) heuristic approximates a high excursion probability
by the expected Euler characteristic of the corresponding excursion set,
whereas the Gaussian kinematic formula gives exact expressions for expected
Lipschitz--Killing curvatures of broad classes of smooth Gaussian-related
fields.  Exponential accuracy of the EEC approximation was established in
\cite{TaylorTakemuraAdler2005}, while the Gaussian kinematic formula was
developed in \cite{Taylor2006}; see also the monograph
\cite{AdlerTaylor2007}.  Recent work establishes super-exponentially accurate
EEC approximations for smooth fields with general non-constant variance
functions \cite{Cheng2023EEC} and for joint high excursions of a class of smooth
\(\mathbb R^2\)-valued Gaussian vector fields \cite{ChengXiao2023EEC}.  These
results are high-threshold results, but their principal object is the occurrence
and geometry of a smooth excursion rather than the upper tail of its occupation
volume.

A related smooth-field literature treats excursion volume and other geometric
functionals as random variables and analyzes their fluctuations, often by
Hermite or Wiener-chaos expansions, Malliavin calculus, and reduction
principles.  Central limit theorems for excursion-set volumes of weakly
dependent random fields, with Gaussian fields as a principal case, were proved
in \cite{BulinskiSpodarevTimmermann2012}.  Quantitative central limit theorems
for Gaussian-field sojourn times at fixed and moving levels were obtained in
\cite{Pham2013SojournCLT}, and limit theorems for the broader family of
Lipschitz--Killing curvatures were developed in
\cite{KratzVadlamani2017}.  More recent work studies high moving levels for
long-range dependent spatio-temporal Gaussian fields under increasing-domain
asymptotics \cite{LeonenkoRuizMedina2024} and almost-sure and quantitative
central limit theorems for integral functionals of stationary Gaussian fields,
including excursion volume, by chaos expansions and Malliavin--Stein methods
\cite{MainiRossiZheng2026}.  This literature describes global fluctuations,
reduction principles, and geometric structure.  It is complementary to the
problem considered here: an exact rare-event asymptotic for the upper tail of a
locally rescaled path functional on a fixed observation interval.

For Gaussian processes that need not be differentiable, exact high-level
sojourn tails are governed instead by local covariance scaling, conditional
weak limits, and Pickands--Berman constants.  Building on Berman's approach,
uniform double-sum methods yield exact asymptotics for broad classes of
continuous- and discrete-time Gaussian sojourns
\cite{DebickiHashorvaPengMichna2017}, for Gaussian processes with trend and the
associated passage time of the accumulated sojourn
\cite{DebickiLiuMichna2019}, and for Gaussian-related random fields
\cite{DebickiHashorvaLiuMichna2021}.  Recent developments include structural and simulation-oriented results for
the associated Berman functions \cite{DebickiHashorvaMichnaBerman2022}, as well
as locally self-similar Gaussian processes and cumulative Parisian ruin
\cite{Novikov2024Sojourns}.  This locally non-smooth line is the closest in
asymptotic scale and proof method to the present paper.

The vector-valued setting adds a second difficulty.  The exponential cost of a
simultaneous high excursion is determined by a Gaussian quadratic program, and
only a subset of coordinates may be active.  The quadratic-program active-set
reduction and a uniform double-sum method for vector-valued Gaussian processes
were developed in \cite{VVGP}.  A locally additive extension to non-homogeneous
vector-valued Gaussian fields, with an application to double-crossing
probabilities, was obtained in \cite{VVGF}.  Most closely related to the
stationary ordinary-sojourn specialization of our results is the recent paper
\cite{DebickiHashorvaMichna2025}, which derives exact asymptotics for the
Lebesgue occupation volume of a stationary \(\mathbb R^d\)-valued Gaussian
field on a multidimensional parameter box.  That work is more general in the
dimension of the parameter set for the indicator-sojourn functional.  Its main
formulation uses the all-active condition \(\Sigma^{-1}\bm b>0\), while the
paper also discusses the dimension-reduction issue outside this case.  Our
results are complementary: we work with a one-parameter process, but allow a
substantially broader class of nonlinear and non-additive path functionals,
carry out the active, tangential, and slack coordinate reduction explicitly,
treat a boundary non-stationary setting, and derive conditional
first-exceedance-time limits.

The object of this paper is therefore not one particular occupation time, but a
unified class of high-level path functionals.  We study probabilities of the
form
\begin{equation}
  \label{def probability psi(u)}
  \pk{\Gamma_{[0,T]} ( \uu (\bm{X} - u \bm{b}) ) > L_u },
  \qquad u\to\infty,
\end{equation}
where \(\bm X\) is a continuous centered \(\mathbb R^d\)-valued Gaussian
process, \(\bm b\) is a fixed high-level direction with at least one positive
coordinate, \(\uu\) is the anisotropic normalization selected by the Gaussian
quadratic program in \eqref{def b tilde}--\eqref{def hat u}, and \(L_u\) is the
natural local scaling of the functional.  The functional \(\Gamma\) may be
nonlinear and need not be continuous everywhere in the uniform topology.  The
class covered by our assumptions includes the usual exceedance event,
classical sojourn times, positive-area functionals, Choquet integrals with
respect to non-additive capacities, and shrinking-window Parisian functionals
through the local-footprint convention described below.

This functional breadth is useful because persistence can be encoded in several
inequivalent ways.  A Parisian event requires an adverse path over a time
window rather than a crossing at a single instant; Gaussian Parisian ruin was
studied, for example, in
\cite{DebickiHashorvaJi2014,DebickiHashorvaJi2016}.  Ordinary sojourn and area
functionals measure duration and accumulated excess.  Choquet functionals
replace Lebesgue measure by a capacity and can therefore weight the shape of an
excursion set non-additively: depending on the capacity, sets of equal length
but different clustering or spatial spread need not have the same value.  We
use Choquet's capacity framework \cite{choquet-ref} and the standard theory of
non-additive integration \cite{Denneberg1994} to obtain concrete corollaries
for this class.

The assumptions on \(\Gamma\) are structural rather than tied to a particular
formula.  We require non-decreasing behavior with respect to the observation
set, affine time scaling, and two regularity properties for the limiting
Gaussian model: absence of an atom at the level under consideration and
continuity at almost every shift.  Thus integral functionals
\[
  \Gamma_E(\bm f)=\int_E G(\bm f(t))\,dt
\]
and the Choquet functionals
\[
  \Gamma_E^{\nu,\psi}(\bm f)
  =
  \int_0^\infty \nu\{t\in E:\psi(\bm f(t))>r\}\,dr
\]
can be handled by the same Gaussian estimates once their local regularity is
verified.  A Parisian moving-window functional is naturally indexed by a long start set
and a short local footprint,
\[
  \Pi_{E;F}(\bm f)
  =\sup_{t\in E}\inf_{s\in F}\min_{1\le i\le d}f_i(t+s),
  \qquad \bm f\in C(E+F,\R^d).
\]
The two indices play different roles.  Only the start set \(E\) is partitioned
in the global double-sum argument; the footprint \(F\) stays bounded in local
coordinates and is rescaled jointly with \(E\).  Accordingly, our
local-footprint extension imposes union-monotonicity only in \(E\), requires no
monotonicity or decomposition property in \(F\), and replaces the one-set
scaling rule by joint affine scaling of \((E,F)\).  With a physical footprint
\(\rho_uF\), where \(\rho_u\) is the operative local time scale, the same
conditioning and double-sum proofs yield exact Parisian asymptotics.  This is
the usual non-degenerate shrinking-window convention on the correlation
(Pickands) scale in the stationary, \(\alpha<\beta\), and
\(\alpha=\beta\) regimes; see
\cite{DebickiHashorvaJi2014,DebickiHashorvaJi2016}.  Appendix~\ref{sec:auxiliary_results}
verifies the required local no-atom and continuity properties.  The
local-conditioning step is related to uniform approximation results for
homogeneous continuous functionals of scalar Gaussian fields in
\cite{DebickiHashorvaLiu2017}.  Our no-atom and almost-every-shift continuity
conditions are designed to include discontinuous occupation and Choquet
functionals, while the global argument additionally incorporates vector
thresholds, active-coordinate reduction, and stationary and boundary
non-stationary double-sum estimates.

We now summarize the main contributions.
\begin{enumerate}[(i)]
  \item \emph{Stationary vector-valued path functionals.}  For a stationary
  one-parameter vector-valued Gaussian process with local covariance behavior
  \(\|\Sigma-\mathcal R(t)-t^\alpha V\|_{\mathrm F}=o(t^\alpha)\), we obtain an exact asymptotic for
  \eqref{def probability psi(u)}.  The constant is a functional Pickands-type
  constant built from the limiting multivariate fractional Brownian field, the
  active set of the quadratic program, and \(\Gamma\).  For the classical
  indicator sojourn, this complements
  \cite{DebickiHashorvaMichna2025}; the point here is that the same theorem also
  covers positive-area, Choquet, and other admissible path
  functionals, including cases with slack coordinates.

  \item \emph{Boundary non-stationary path functionals.}  Suppose that the
  inverse generalized variance has a unique minimum at the boundary point
  \(0\), and locally
  \[
    \bigl\|\Sigma-R(t,s)-A t^\beta-A^\top s^\beta
      -V|t-s|^\alpha\bigr\|_{\mathrm F}
    =o\bigl(t^\beta+s^\beta+|t-s|^\alpha\bigr),
    \qquad t\ge s\ge0.
  \]
  The competition between the correlation scale \(u^{-2/\alpha}\) and the
  variance scale \(u^{-2/\beta}\) produces three regimes.  The leading
  constant is Pickands-type when \(\alpha<\beta\), Piterbarg-type with a
  boundary drift when \(\alpha=\beta\), and deterministic when
  \(\alpha>\beta\).

  \item \emph{Local-footprint and Parisian extensions.}  We show that the
  proofs require long-set structure only in the start set \(E\).  A bounded
  footprint \(F\) is never partitioned and is merely transported by the local
  affine rescaling.  This yields stationary and boundary non-stationary
  asymptotics for shrinking-window Parisian persistence, with no inclusion
  assumption on the window variable.

  \item \emph{Choquet-integral corollaries.}  We formulate an admissible class
  of Choquet data \((\nu,\psi)\) and derive explicit corollaries from both
  main theorems.  If \(\nu\) is affinely homogeneous with exponent
  \(\lambda_\nu\), then the threshold is
  \(L u^{-2\lambda_\nu/\alpha}\) in the stationary and
  \(\alpha<\beta\) boundary regimes, and
  \(L u^{-2\lambda_\nu/\beta}\) in the \(\alpha\ge\beta\) boundary regimes.
  This includes genuinely non-additive power capacities such as
  \(\nu(A)=\mes(A)^\theta\), \(\theta>1\), as well as interaction capacities
  that depend on the spatial arrangement of an excursion set.

  \item \emph{Conditional first-exceedance-time limits.}  We identify the
  limiting distribution of the first time at which the path functional exceeds
  its scaled level.  In the stationary case the first exceedance location is
  asymptotically uniform on \([0,T]\).  In the boundary non-stationary case it
  lives on the variance scale and its law changes with the same
  \(\alpha\)-versus-\(\beta\) trichotomy as the tail probability.
\end{enumerate}

The proofs combine conditional weak convergence on shrinking time windows,
adapted Bonferroni inequalities for functional exceedance events, and uniform
single-sum and double-sum estimates.  The formulation deliberately separates
functional verification from the Gaussian localization argument: once a
functional satisfies the structural assumptions and the local limiting model
is non-trivial, the main theorems yield both the exact tail asymptotic and, when
applicable, the conditional first-exceedance-time law.

\medskip

\noindent\textbf{Organization of the paper.}  Section~\ref{sec:preliminaries}
introduces the functional assumptions, the quadratic-programming notation, the
limiting constants, and the admissible Choquet data.  Sections~\ref{sec:stationary_case}
and~\ref{sec:non_stationary_case} state the stationary and boundary
non-stationary theorems, together with the Choquet-integral corollaries.
Section~\ref{sec:exceedance-times} gives the exceedance-time limits.
Section~\ref{sec:examples} works out representative examples.  All proofs are
collected in the appendices: Appendix~\ref{sec:proofs_main_claims} proves the
main claims, Appendix~\ref{sec:auxiliary_results} contains the auxiliary
functional results, and Appendix~\ref{sec:technical_estimates} contains the
remaining technical estimates.

\bigskip
\textbf{Notation.} Throughout the paper \( a \lesssim b \) means that there
exists \( c > 0 \) independent of the parameters \( u \), \( \Lambda \) and \( S
\), such that \( a \leq c b \). If the constant \( c \) is allowed to depend on
\( \Lambda \), it is denoted by \( \lesssim_\Lambda \). The symbol \( \sim \)
is reserved for scalar asymptotic equivalence as \( u \to \infty \), whereas
\( \sim_{u, \Lambda} \) denotes asymptotic equivalence as \( u \to \infty \)
and \( \Lambda \to \infty \) in this exact order:
\( f \sim_{u, \Lambda} g \iff \lim_{\Lambda \to \infty}
\lim_{u \to \infty} f/g=1 \). Similarly, \( f \sim_{u, \Lambda, S}g \) means
\( \lim_{\Lambda \to \infty}\lim_{S \to \infty}\lim_{u \to \infty}f/g=1 \).
For vector- and matrix-valued expressions, \(O(\cdot)\) and \(o(\cdot)\) are
understood in Euclidean and Frobenius norm, respectively; in particular, we do
not use quotient notation such as \(M_1\sim M_2\) for matrices.  For a positive
definite matrix \(\Sigma\), \(\varphi_\Sigma\) denotes the density of the
centered Gaussian vector with covariance matrix \(\Sigma\).
The symbol \(\mes\) denotes Lebesgue measure on \(\R\).  We write
\(\Gamma(a,x)=\int_x^\infty t^{a-1}e^{-t}\,dt\) for the upper incomplete Gamma
function and \(\Gamma(a)=\Gamma(a,0)\) for the Gamma function.
\smallskip

Let \( \Sigma \) be a positive definite matrix. If \( \bm{b} \in \R^d \setminus
(-\infty, 0]^d \), then the quadratic programming problem
\begin{equation*}
  \Pi_\Sigma(\bm{b}) : \quad
  \text{minimize } \bm{x}^\top \Sigma^{-1} \bm{x} \text{ subject to } \bm{x} \geq \bm{b}
\end{equation*}
has the unique solution
\begin{equation}
  \label{def b tilde}
  \bb = \argmin \{ \bm{x}^\top \Sigma^{-1} \bm{x} : \bm{x} \geq \bm{b} \}.
\end{equation}
Set \(\bm w=\Sigma^{-1}\bb\).  The Karush--Kuhn--Tucker conditions are
\begin{equation}
  \label{eq qp kkt conditions}
  \bb\ge\bm b,
  \qquad
  \bm w\ge\bm0,
  \qquad
  w_i(\tilde b_i-b_i)=0,
  \quad 1\le i\le d;
\end{equation}
see also \cite[Lemma~4.1]{VVGP}.  Define the \emph{essential index set}, its
complement, the tangential boundary set, and the slack set by
\begin{equation}
  \label{def qp index sets}
  I=\{i:w_i>0\},\qquad J=I^c,\qquad
  \Jtan=\{j\in J:\tilde b_j=b_j\},\qquad
  \Jsl=J\setminus\Jtan.
\end{equation}
Because \(\bm b\notin(-\infty,0]^d\), the set \(I\) is non-empty.  The
relations \eqref{eq qp kkt conditions} and \(\bb=\Sigma\bm w\) give
\begin{equation}
  \label{eq qp kkt decomposition}
  \bb_I=\bm b_I,\qquad
  \bm w_I=(\Sigma_{II})^{-1}\bm b_I>\bm0_I,\qquad
  \bm w_J=\bm0_J,
\end{equation}
and
\begin{equation}
  \label{eq qp inactive coordinates}
  \bb_J=\Sigma_{JI}(\Sigma_{II})^{-1}\bm b_I\ge \bm b_J,
  \qquad
  \bb_{\Jtan}=\bm b_{\Jtan},
  \qquad
  \bb_{\Jsl}>\bm b_{\Jsl}.
\end{equation}
Equivalently, \(I\) is the unique non-empty subset for which
\begin{equation*}
  (\Sigma_{II})^{-1}\bm b_I>\bm0_I,
  \qquad
  \Sigma_{JI}(\Sigma_{II})^{-1}\bm b_I\ge\bm b_J;
\end{equation*}
the optimizer is then given by \eqref{eq qp kkt decomposition}--\eqref{eq qp inactive coordinates}.
We call the coordinates in \(I\) active or
essential, those in \(\Jtan\) tangential, and those in \(\Jsl\) slack.

\section{Preliminaries}
\label{sec:preliminaries}

\subsection{Class of functionals}
\label{sec:class_of_functionals}

Let \( \mathcal{L} \subset \R_+ \) be a non-empty set, which we shall call the
set of admissible levels. Next, let \( \mathcal{K} \) denote the set of all
non-empty compact subsets of \( \R \). For a compact set \(E\) and an integer
\(m\ge1\), the space \(C(E,\R^m)\) is equipped with the uniform norm.  Every
reference below to Borel measurability or continuity of a path functional is
with respect to this topology.  For each \( E \in \mathcal{K} \), let
\[
  \Gamma_E:C(E,\R^d)\to\R
\]
be Borel measurable, but not necessarily linear or continuous.  When coordinate
deletion is used, we assume separately that, for every proper
\(K\subset\{1,\ldots,d\}\), there is a Borel map
\[
  \Gamma_E^{(-K)}:C(E,\R^{|K^c|})\to\R,
  \qquad \Gamma_E^{(-\emptyset)}=\Gamma_E.
\]
We retain the convenient notation
\(\Gamma_E(\bm f_{K^c},\bm\infty_K)\) as shorthand for
\(\Gamma_E^{(-K)}(\bm f_{K^c})\); no topology or Borel structure on arbitrary
extended-valued paths is used.  In the original probability only sets
\(E\subset[0,T]\) are applied to the process \(\bm X\), but the enlarged index
set is needed for the rescaled local windows \([0,S]\) that enter the limiting
constants. We assume that, for all \( \bm{f} \in C ( \R,\R^d ) \), the following
conditions hold whenever all displayed restrictions are well-defined:
\begin{enumerate}
  \item [(\namedlabel{B1}{\( \mathbf{B}_1 \)})] \textbf{Union-monotonicity.} For all \( A, B \in \mathcal{K} \) and \( L \in \mathcal{L} \), we have
  \[
  \Gamma_A ( \bm{f} ) > L
  \implies
  \Gamma_{A \cup B} ( \bm{f} ) > L
  \implies
  \Gamma_A ( \bm{f} ) > L
  \quad \text{or} \quad
  \sup_{t \in B} \min_{i = 1, \dots , d} f_i ( t ) > 0.
  \]
  and
  \begin{equation*}
    \Gamma_A ( \bm{f} ) > L
    \implies 
    \sup_{t \in A} \min_{i = 1, \dots , d} f_i ( t ) > 0.
  \end{equation*}
  \item [(\namedlabel{B2}{\( \mathbf{B}_2 \)})] \textbf{Affine time-scaling property.} There exists \( \lambda_\Gamma \in \R \) such that
  \[
  \Gamma_{a E + b} ( \bm{f} ) > L
  \iff
  a^{\lambda_\Gamma} \, \Gamma_E ( \bm{f} ( a \, ( \cdot ) + b ) ) > L
  \]
  for all \( a > 0 \), \( b \in \R \), \( E \in \mathcal{K} \) and \( L \in
  \mathcal{L} \).
\end{enumerate}

\begin{remark}[Scaled levels]
\label{rem:scaled-level-B1}
Although Assumption~\ref{B1} is stated only at levels in \(\mathcal L\), no
additional hypothesis is needed at the scaled levels occurring in the proofs.
Indeed, applying Assumption~\ref{B2} to the inverse affine map gives, for every
\(L\in\mathcal L\),
\begin{equation}
  \label{eq:B2-threshold-scaled-form}
  \Gamma_{aE+b}(\bm f)>a^{\lambda_\Gamma}L
  \quad\Longleftrightarrow\quad
  \Gamma_E\bigl(\bm f(a\,\cdot+b)\bigr)>L.
\end{equation}
Consequently, Assumptions~\ref{B1}--\ref{B2} imply the two implications in
Assumption~\ref{B1} at every level \(a^{\lambda_\Gamma}L\), \(a>0\): apply
Assumption~\ref{B1} at level \(L\) to the inverse affine images of the sets and
to the rescaled path, then use \eqref{eq:B2-threshold-scaled-form}.  Thus the
levels \(L_u\) in the main theorems are covered even when
\(\mathcal L=\{L\}\).  We use this scaled-level consequence without further
comment.
\end{remark}

The following two conditions are local in the observation set and are defined
with respect to a fixed continuous, but not necessarily centered, Gaussian
process \(\bm W\).  For a specified \(E\in\mathcal K\), we say that:

\begin{enumerate}
  \item [(\namedlabel{B3}{\( \mathbf{B}_3 \)})] \textbf{No \(\mathcal L\)-level
  atoms on \(E\).} For every \(L\in\mathcal L\) and Lebesgue-a.e.
  \(\bm x\in\R^d\),
  \begin{equation*}
    \pk{\Gamma_E(\bm W-\bm x)=L}=0.
  \end{equation*}
  \item [(\namedlabel{B4}{\( \mathbf{B}_4 \)})] \textbf{Continuity at almost
  every shift on \(E\).} For Lebesgue-a.e. \(\bm x\in\R^d\),
  \begin{equation*}
    \pk{\bm W-\bm x\text{ is a continuity point of }\Gamma_E}=1.
  \end{equation*}
\end{enumerate}
Thus an invocation of Assumptions~\ref{B3}--\ref{B4} must specify the compact
set or class of compact sets on which they are required.  The main theorems use
them only on the intervals \([0,S]\).

The following condition shall be used separately from the other four:

\begin{enumerate}
  \item [(\namedlabel{B5}{\( \mathbf{B}_5 \)})] \textbf{Constant reduction property.} If \( \bm{f}_K \) is a constant subvector of \( \bm{f} \) (i.e., \(t\mapsto \bm{f}_K(t)\) is constant), then 
  \[ 
  \Gamma_E ( \bm{f} ) > L 
  \iff 
  \Gamma_E ( \bm{f}_{K^c}, \bm{\infty}_K ) > L
  \quad \text{and} \quad
  \bm{f}_K > \bm{0} 
  \] 
  for all \( L \in \mathcal{L} \) and \( E \in \mathcal{K} \).
\end{enumerate} 

\begin{remark}
  \label{remark Gamma_A is monotone in A}
  Assumption~\ref{B1} implies a version of Bonferroni inequality provided in
  Lemma~\ref{lemma Bonferroni adapted for sojourns}. If \( \mathcal{L} = \R_+
  \), the first implication of assumption~\ref{B1} is equivalent to \( A \mapsto
  \Gamma_A ( \bm{f} ) \) being non-decreasing with respect to the inclusion
  order. In general, this implies that \( A \mapsto \Ind{\Gamma_A ( \bm{f} ) >
  L} \) is non-decreasing for all \( L \in \mathcal{L} \).
\end{remark}

\begin{remark}
  If \(\mathcal L=\R_+\) and the power factor in Assumption~\ref{B2} is
  replaced by a positive multiplicative function \(\ell\), then the power form
  \(\ell(a)=a^{\lambda_\Gamma}\) follows provided \(\ell\) is, for example,
  Borel measurable or locally bounded.  Without such a regularity condition,
  multiplicativity alone admits pathological non-power solutions.  We therefore
  impose the power form directly in Assumption~\ref{B2}.
\end{remark}

\begin{remark}
  The notation \(\Gamma_E(\bm f_{K^c},\bm\infty_K)\) refers to the separate
  finite-valued coordinate-deleted map \(\Gamma_E^{(-K)}\) introduced above;
  \(+\infty\) is only a mnemonic for omitting coordinates. For example,
  \begin{itemize}
    \item if \( \Gamma_E ( \bm{f} ) = \sup_{t \in E} \min_{i} f_i ( t ) \), then
  \( \Gamma_E ( \bm{f}_{K^c}, \bm{\infty}_K ) = \sup_{t \in E} \min_{i \in K^c}
  f_i(t) \);
    \item if \( \Gamma_E ( \bm{f} ) = \int_E \Ind{ \bm{f} ( t ) > \bm{0} } \, d
  t \), then \( \Gamma_E ( \bm{f}_{K^c}, \bm{\infty}_K ) = \int_E \Ind{
  \bm{f}_{K^c} ( t ) > \bm{0}_{K^c} } \, d t \);
    \item if \( \Gamma_E ( \bm{f} ) = \int_E \min_i ( f_i ( t ) )_+ \, dt \),
    then \( \Gamma_{E} ( \bm{f}_{K^c}, \bm{\infty}_{K} ) = \int_E \min_{i \in
    K^c} ( f_i ( t ) )_+ \, dt \).
  \end{itemize}
  Each displayed deleted map is Borel measurable on its own uniform path space.
\end{remark}

\begin{remark}[Basic examples and local footprints]
  The classical sojourn family
  \[
    \Gamma_E(\bm f)=\int_E\Ind{\bm f(t)>\bm0}\,dt
  \]
  satisfies~\ref{B1},~\ref{B2}, and~\ref{B5} with
  \(\lambda_\Gamma=1\).  Assumptions~\ref{B3}--\ref{B4} must be checked on
  the intervals and at the levels used in the theorem; endpoint levels can
  create atoms for discontinuous sojourn scalarizations.  The positive-area
  family
  \[
    \Gamma_E(\bm f)=\int_E\min_{1\le i\le d}(f_i(t))_+\,dt
  \]
  satisfies~\ref{B1}--\ref{B2} with \(\lambda_\Gamma=1\), and its interval-wise
  no-atom and continuity properties at positive levels are verified in
  Appendix~\ref{sec:auxiliary_results}.

  The Parisian moving-window family is treated by the local-footprint
  convention below.  The start set is the only long index, while the window is
  a bounded local footprint that is rescaled jointly with the start set.  In
  particular, no union-monotonicity assumption is imposed on the window.
\end{remark}

\subsection{Local footprints and shrinking-window functionals}
\label{sec:local-footprints}

Let
\[
  \mathcal K_0=\{F\in\mathcal K:0\in F\},
  \qquad
  \mathcal K_0^+=\{F\in\mathcal K_0:F\subset[0,\infty)\},
  \qquad
  E+F=\{t+s:t\in E,\ s\in F\}.
\]
The set \(E\) is the long observation, or start, set, whereas \(F\) is a
bounded local footprint.  The proofs partition \(E\), but never partition
\(F\).  For later use, put
\[
  [0,T]\ominus \rho F
  =\{t\in\R:t+\rho F\subset[0,T]\}
  =[-\rho\min F,\,T-\rho\max F],
\]
whenever the interval on the right is non-empty.  In particular, if
\(F\in\mathcal K_0^+\), then
\([0,T]\ominus\rho F=[0,T-\rho\max F]\).

\begin{definition}[Admissible local-footprint family]
\label{def:local-footprint-family}
For \(E\in\mathcal K\) and \(F\in\mathcal K_0\), equip
\(C(E+F,\R^d)\) with the uniform norm and let
\[
  \Gamma_{E;F}:C(E+F,\R^d)\to\R
\]
be Borel measurable.  If coordinate deletion is used, the notation
\(\Gamma_{E;F}(\bm f_{K^c},\bm\infty_K)\) denotes a separate Borel map
\(\Gamma_{E;F}^{(-K)}:C(E+F,\R^{|K^c|})\to\R\), as above.
We call \(\{\Gamma_{E;F}\}\) an admissible local-footprint family, with
admissible levels \(\mathcal L\) and exponent \(\lambda_\Gamma\), if the
following conditions hold whenever the displayed restrictions are defined.
\begin{enumerate}
  \item[(\namedlabel{LF1}{\(\mathbf{LF}_1\)})]
  \textbf{Long-set union structure.}  For every fixed \(F\in\mathcal K_0\),
  all \(A,B\in\mathcal K\), \(L\in\mathcal L\), and
  \(\bm f\in C((A\cup B)+F,\R^d)\),
  \[
    \Gamma_{A;F}(\bm f)>L
    \Longrightarrow
    \Gamma_{A\cup B;F}(\bm f)>L
    \Longrightarrow
    \Gamma_{A;F}(\bm f)>L
    \quad\hbox{or}\quad
    \sup_{t\in B}\min_i f_i(t)>0,
  \]
  and
  \[
    \Gamma_{A;F}(\bm f)>L
    \Longrightarrow
    \sup_{t\in A}\min_i f_i(t)>0.
  \]
  No inclusion, monotonicity, or decomposition property is imposed in the
  footprint variable.

  \item[(\namedlabel{LF2}{\(\mathbf{LF}_2\)})]
  \textbf{Joint affine scaling.}  For all \(a>0\), \(b\in\R\),
  \(E\in\mathcal K\), \(F\in\mathcal K_0\), and \(L\in\mathcal L\),
  \[
    \Gamma_{aE+b;aF}(\bm f)>a^{\lambda_\Gamma}L
    \quad\Longleftrightarrow\quad
    \Gamma_{E;F}\bigl(\bm f(a\,\cdot+b)\bigr)>L.
  \]

  \item[(\namedlabel{LF3}{\(\mathbf{LF}_3\)})]
  \textbf{Local regularity for the exact limiting element.}  On each pair
  \(([0,S],F)\) used in a theorem, let \(\mathcal W_{\bm x}\) denote the
  coordinate-reduced shifted limiting path appearing in the corresponding
  local constant, regarded as a random element of
  \(C([0,S]+F,\R^d)\).  Then, for every \(L\in\mathcal L\) and
  Lebesgue-a.e. shift vector \(\bm x\),
  \[
    \pk{\Gamma_{[0,S];F}(\mathcal W_{\bm x})=L}=0,
    \qquad
    \pk{\mathcal W_{\bm x}\text{ is a continuity point of }
       \Gamma_{[0,S];F}}=1.
  \]
\end{enumerate}
As in Remark~\ref{rem:scaled-level-B1}, Assumptions~\ref{LF1}--\ref{LF2}
imply the long-set union structure at the scaled level
\(a^{\lambda_\Gamma}L\) for the scaled footprint \(aF\).  This is the form used
for a physical footprint \(\rho_uF\).
The coordinate-reduction property analogous to Assumption~\ref{B5}, when it
is used, is required only for the path coordinates and with \(F\) fixed.
A one-set family canonically lifts to a local-footprint family by
\[
  \Gamma_{E;F}(\bm f)=\Gamma_E(\bm f|_E),
\]
provided its affine scaling is read in the equivalent threshold-scaled form
used in~\ref{LF2}.
\end{definition}

\begin{remark}[Parisian persistence]
\label{rem:parisian-local-footprint}
For \(F\in\mathcal K_0\), define
\begin{equation}
  \label{def:parisian-footprint-functional}
  \Pi_{E;F}(\bm f)
  =\sup_{t\in E}\inf_{s\in F}\min_{1\le i\le d}f_i(t+s),
  \qquad \bm f\in C(E+F,\R^d).
\end{equation}
The coordinate-deleted map \(\Pi_{E;F}^{(-K)}\) is obtained by omitting the
indices in \(K\) from the minimum; we keep the shorthand
\(\Pi_{E;F}(\bm f_{K^c},\bm\infty_K)\).  At the admissible level
\(\mathcal L=\{0\}\), this is an admissible
local-footprint family with \(\lambda_\Gamma=0\).  For fixed \(F\),
\[
  \Pi_{A\cup B;F}(\bm f)
  =\max\{\Pi_{A;F}(\bm f),\Pi_{B;F}(\bm f)\}.
\]
Because \(0\in F\), a successful window starting in \(E\) also gives a
pointwise exceedance at its start:
\[
  \Pi_{E;F}(\bm f)>0
  \Longrightarrow
  \sup_{t\in E}\min_i f_i(t)>0.
\]
These two identities prove~\ref{LF1}, without any condition under enlargement
of \(F\).  The identity
\[
  \Pi_{aE+b;aF}(\bm f)
  =\Pi_{E;F}\bigl(\bm f(a\,\cdot+b)\bigr)
\]
proves~\ref{LF2}.

For completeness, the local regularity survives the active--inactive
coordinate reduction.  If \(D\subset\{1,\ldots,d\}\) denotes the path
coordinates and the remaining coordinates are fixed at a vector \(\bm c\),
then
\begin{equation}
  \label{eq:parisian-constant-coordinate-reduction}
  \Pi_{E;F}(\bm f_D,\bm c_{D^c})
  =\min\left\{\Pi^D_{E;F}(\bm f_D),\min_{j\in D^c}c_j\right\},
\end{equation}
with the second term omitted when \(D^c=\emptyset\).  The active-coordinate
functional is 1-Lipschitz, so continuity is automatic.  The level-zero
no-atom assertion for the exact active--tangential--slack limiting element is
proved in Lemma~\ref{lem:parisian-partially-frozen}: strictly positive slack
constants cannot create a zero, while a frozen tangential coordinate can equal
zero only on a coordinate hyperplane in the shift space.  Hence~\ref{LF3}
holds.  The same identity gives the level-zero constant-reduction property:
fixed coordinates may be discarded exactly when all of them are strictly
positive.
\end{remark}

For a local-footprint family, let \(\bm Z\) and \(\bm q\) be defined on
\(E+F\) and set
\begin{align*}
  \mathfrak H_{\Gamma;F,L}^{\bm Z,\bm q}(E)
  ={}&\int_{\R^d}
  \exp\left(
    \bm x_I^\top\bm w_I
    -\frac12\bm x_J^\top(\Sigma^{-1})_{JJ}\bm x_J
  \right)\\[-2pt]
  &\quad\times
  \pk{\Gamma_{E;F}
    \begin{pmatrix}
      \bm Z-\bm q-\bm x_I\\
      -\bm x_{\Jtan}\\
      (\bb-\bm b)_{\Jsl}
    \end{pmatrix}>L}\,d\bm x.
\end{align*}
Write
\[
  H_{\Gamma;F,L,\alpha,V,\bm d,\Sigma,\bm b}(E)
\]
for this constant with \(\bm Z=\bm Y_{\alpha,V}\) and
\(\bm q=S_{\alpha,V}(\cdot)\bm w_I+\bm d\), both restricted to \(E+F\).
Define the corresponding infinite-horizon constants by
\begin{align*}
  \mathcal H_{\Gamma;F,L,\alpha,V,\Sigma,\bm b}
  &=\lim_{S\to\infty}\frac1S
    H_{\Gamma;F,L,\alpha,V,\bm0,\Sigma,\bm b}([0,S]),\\
  \mathcal P_{\Gamma;F,L,\alpha,V,\bm d,\Sigma,\bm b}
  &=\lim_{S\to\infty}
    H_{\Gamma;F,L,\alpha,V,\bm d,\Sigma,\bm b}([0,S]),\\
  \mathcal C_{\Gamma;F,L,\bm d,\Sigma,\bm b}
  &=\lim_{S\to\infty}
    \mathfrak H_{\Gamma;F,L}^{\bm0,\bm d_I}([0,S]),
\end{align*}
whenever the limits exist.

\begin{proposition}[Local-footprint transfer principle]
\label{prop:local-footprint-transfer}
Under the Gaussian hypotheses of the theorem being transferred, fix
\(L\in\mathcal L\) and \(F\in\mathcal K_0\); in the boundary non-stationary
setting assume \(F\in\mathcal K_0^+\).  Suppose that
Assumptions~\ref{LF1}--\ref{LF2} hold for the fixed footprint \(F\) and its
affine images as in~\ref{LF2}, that Assumption~\ref{LF3} holds on every pair
\(([0,S],F)\), \(S\ge0\), for the exact limiting element in that theorem, and
that the corresponding local model is non-trivial.  Put
\[
  \rho_u=u^{-2/\alpha}
  \quad\text{in the stationary case and when }\alpha<\beta,
  \qquad
  \rho_u=u^{-2/\beta}
  \quad\text{when }\alpha\ge\beta,
\]
and let \(E_u=[0,T]\ominus\rho_uF\).  Then the conclusions of
Theorems~\ref{thm:stationary_case} and~\ref{thm:main_theorem} remain valid after
replacing
\[
  \{\Gamma_{[0,T]}(\uu(\bm X-u\bm b))>L_u\}
  \quad\text{by}\quad
  \{\Gamma_{E_u;\rho_uF}(\uu(\bm X-u\bm b))
       >L\rho_u^{\lambda_\Gamma}\},
\]
and replacing every local \(\Gamma_E\) in a limiting constant by
\(\Gamma_{E;F}\).  In particular, the powers of \(u\) in the main formulae
are unchanged; only the limiting constants acquire the fixed footprint \(F\).

If \(F\in\mathcal K_0^+\), the same transfer applies to
Theorem~\ref{thm:exceedance_times} for the first successful start time
\[
  \mathfrak t_F(u,L)
  =\inf\left\{0<t\le T-\rho_u\max F:
    \Gamma_{[0,t];\rho_uF}(\uu(\bm X-u\bm b))
       >L\rho_u^{\lambda_\Gamma}\right\}.
\]
No inclusion property of the map \(F\mapsto\Gamma_{E;F}\) is used.
\end{proposition}

\begin{remark}[Endpoint convention]
The erosion \(E_u=[0,T]\ominus\rho_uF\) merely keeps the entire footprint
inside the original process domain.  If \(\bm X\) is defined on a fixed
neighbourhood of \([0,T]\), one may instead allow all starts in \([0,T]\).
The two conventions differ by start intervals of total length
\(O(\rho_u)\) and therefore have the same leading asymptotic.
\end{remark}

\noindent The proof is given in
Appendix~\ref{sec:proof_local_footprint_transfer}.

\begin{definition}[Capacity terminology]
\label{def:capacity-terminology}
Let \(E\subset\R\) be non-empty and compact.  A \emph{finite capacity} on
\(E\) (not necessarily normalized) is a set function \(\mu:\mathcal B(E)\to[0,\infty)\) such that
\(\mu(\emptyset)=0\) and \(\mu(A)\le\mu(B)\) whenever \(A\subset B\).  It is
\emph{regular} if, for every Borel set \(A\subset E\),
\[
  \mu(A)
  =\sup\{\mu(K):K\subset A,\ K\text{ compact}\}
  =\inf\{\mu(U):A\subset U,\ U\text{ relatively open in }E\}.
\]
It is \emph{convex} (equivalently, supermodular) if
\[
  \mu(A\cup B)+\mu(A\cap B)\ge\mu(A)+\mu(B),
  \qquad A,B\in\mathcal B(E).
\]
Its support is the closed set
\[
  \supp\mu
  =\{t\in E:\mu(U)>0\text{ for every relatively open neighbourhood }U\ni t\}.
\]
We say that \(\mu\) has the \emph{support-intersection property} if
\[
  \mu(U)>0
  \quad\Longleftrightarrow\quad
  U\cap\supp\mu\ne\emptyset
\]
for every relatively open \(U\subset E\).  This property implies
\(\mu(E\setminus\supp\mu)=0\), but for a non-additive capacity it need not imply
\(\mu(\supp\mu)=\mu(E)\); the latter equality will be called the property that
the support \emph{carries full capacity}.  We say that the capacity has
\emph{full support} if \(\supp\mu=E\).

A capacity \(\nu\) on the Borel subsets of \(\R\), possibly taking the value
\(+\infty\), is \emph{locally finite} if \(\nu(K)<\infty\) for every compact
\(K\subset\R\).  For compact \(E\), its restriction \(\nu|_E\) is the capacity
on \(E\) given by \((\nu|_E)(A)=\nu(A)\), \(A\in\mathcal B(E)\).
\end{definition}

\begin{remark}
  An interesting class of functionals is given by Choquet integrals with respect
  to a locally finite capacity \(\nu\) in the sense of
  Definition~\ref{def:capacity-terminology}. Let
  \(\psi:\R^d\to\R_+\) be a Borel scalarization function, assume that the
  resulting map on each uniform path space is Borel measurable (as is automatic
  in the continuous admissible-data setting below), and define
  \begin{equation*}
    \Gamma_E(\bm f)=\int_E \psi(\bm f)\,d\nu,
  \end{equation*}
  where the Choquet integral is
  \begin{equation*}
    \int_E u\,d\nu
    \coloneqq
    \int_0^\infty \nu\{x\in E:u(x)>t\}\,dt.
  \end{equation*}
  The following conditions are sufficient for the structural assumptions used in
  the main theorems.
  \begin{itemize}
    \item If \(\psi(\bm x)>0\) implies \(\min_i x_i>0\), then
    \(\Gamma_E\) satisfies~\ref{B1}.
    \item If \(\nu\) is affinely homogeneous, meaning
    \(\nu(aA+b)=a^{\lambda_\Gamma}\nu(A)\) for some
    \(\lambda_\Gamma\in\R\), all \(a>0\), \(b\in\R\), and all Borel sets
    \(A\) with compact closure, then \(\Gamma_E\) satisfies~\ref{B2}.
    \item On a fixed compact set \(E\) with \(0<\nu(E)<\infty\), and at a
    positive level \(L\), the no-atom condition~\ref{B3} on \(E\) follows from
    Lemma~\ref{sojorn_F2_cap} under either of the sufficient hypotheses stated
    there. In particular, this applies when \(\nu|_E\) is a regular finite
    capacity satisfying the support-intersection property, \(\psi\) is continuous
    and coordinatewise non-decreasing, and either the scalarization is strictly
    increasing above zero, or \(\nu|_E\) is convex, its support is connected and
    carries full capacity, and \(L/\nu(E)\) is not a plateau level of the
    scalarization.
    \item The continuity condition~\ref{B4} follows from
    Lemma~\ref{sojorn_F3_cap}. If \(\psi\) is continuous, then the corresponding
    Choquet functional is in fact continuous on \(C(E,\R^d)\).
    \item If, for every proper \(K\), there is a Borel coordinate-deleted
    scalarization \(\psi^{(-K)}:\R^{|K^c|}\to[0,\infty)\) such that, for every
    constant subvector \(\bm f_K\),
    \begin{equation*}
      \nu\{x\in E:\psi(\bm f(x))>t\}
      =
      \nu\{x\in E:\psi^{(-K)}(\bm f_{K^c}(x))>t\}
      \Ind{\bm f_K>\bm0_K}
    \end{equation*}
    for almost all \(t>0\), then the Choquet functional built from
    \(\psi^{(-K)}\) is \(\Gamma_E^{(-K)}\), and \(\Gamma_E\) satisfies~\ref{B5}.
    We may write \(\psi(\bm f_{K^c},\bm\infty_K)\) for
    \(\psi^{(-K)}(\bm f_{K^c})\).  This is the case for the classical sojourn
    scalarization \(\psi(\bm x)=\Ind{\min_i x_i>0}\) and any capacity \(\nu\).
  \end{itemize}
  In the generic limiting constants some slack coordinates are frozen at finite
  positive values.  The no-atom lemma must then be applied to the corresponding
  positive coordinate section of \(\psi\), rather than to \(\psi\) on all of
  \(\R^d\).  Definition~\ref{def:admissible-choquet-data} incorporates this
  section-wise requirement.
\end{remark}

\noindent The verification of these Choquet-integral claims is given in Appendix~\ref{sec:proof_choquet_claims}.

\begin{definition}[Admissible Choquet data]
\label{def:admissible-choquet-data}
Fix \(L>0\).  For a non-empty coordinate set
\(D\subset\{1,\ldots,d\}\) and a vector
\(\bm c\in(0,\infty)^{D^c}\), define the positive coordinate section
\begin{equation}
  \label{def:positive-coordinate-section}
  \psi_{D,\bm c}(\bm y_D)
  =\psi(\bm y_D,\bm c_{D^c}),
\end{equation}
with the evident convention when \(D^c=\emptyset\).  A non-negative
coordinatewise non-decreasing function \(g\) is called
\emph{strictly increasing above zero} if
\[
  g(\bm y)>0,\quad \bm z\succ\bm y
  \quad\Longrightarrow\quad
  g(\bm z)>g(\bm y),
\]
where \(\bm z\succ\bm y\) means that \(z_i\ge y_i\) for every
coordinate and that at least one of these inequalities is strict.
A pair \((\nu,\psi)\) is called \emph{\(L\)-admissible Choquet data} with
 time-scaling exponent \(\lambda_\nu\) if the following conditions hold.
\begin{enumerate}[label=\textup{(Ch\arabic*)},wide]
  \item \(\nu\) is a capacity on the Borel subsets of \(\R\), allowed to take
  the value \(+\infty\) on unbounded sets, finite on compact sets, and affinely
  homogeneous in the sense that
  \[
    \nu(aA+b)=a^{\lambda_\nu}\nu(A),
    \qquad a>0,\ b\in\R,
  \]
  for every Borel set \(A\) with compact closure.  Moreover
  \(0<\nu([0,S])<\infty\) for every \(S>0\).
  \item \(\psi:\R^d\to\R_+\) is continuous and coordinatewise
  non-decreasing.  It is
  supported in the positive orthant in the sense that
  \[
    \psi(\bm x)>0 \quad\Longrightarrow\quad \min_{1\le i\le d}x_i>0 .
  \]
  \item For every \(S>0\), the restriction \(\mu_S=\nu|_{[0,S]}\) is a
  regular finite capacity satisfying the support-intersection property
  \[
    \mu_S(U)>0
    \quad\Longleftrightarrow\quad
    U\cap\supp\mu_S\ne\emptyset
  \]
  for every relatively open \(U\subset[0,S]\).  For each such \(S\), every
  non-empty \(D\subset\{1,\ldots,d\}\), and every
  \(\bm c\in(0,\infty)^{D^c}\), at least one of the following holds:
  \begin{enumerate}[label=\textup{(\alph*)}]
    \item the section \(\psi_{D,\bm c}\) is strictly increasing above zero;
    \item \(\mu_S\) is convex, \(\supp\mu_S\) is connected,
    \(\mu_S(\supp\mu_S)=\mu_S([0,S])\), and
    \(L/\mu_S([0,S])\) is not a plateau level of \(\psi_{D,\bm c}\).
  \end{enumerate}
  Here a number \(r\) is a plateau level of a function on \(\R^{|D|}\) if its
  inverse image has positive \(|D|\)-dimensional Lebesgue measure.
\end{enumerate}
Condition~\textup{(Ch3)} is intentionally imposed only on intervals \([0,S]\):
these are exactly the sets on which Assumptions~\ref{B3}--\ref{B4} enter the
main theorems.  No interval-wise verification is being used to claim
Assumption~\ref{B3} on an arbitrary compact set \(E\); such a set would require
its own no-atom check.
For such data we write
\begin{equation}
  \label{def:choquet-functional-corollaries}
  \Gamma_E^{\nu,\psi}(\bm f)
  =
  \int_0^\infty
  \nu\{t\in E:\psi(\bm f(t))>r\}\,dr,
  \qquad E\in\mathcal K .
\end{equation}
\end{definition}

\begin{remark}[Examples of admissible Choquet data]
\label{rem:admissible-choquet-data-examples}
  The definition is a process-independent sufficient criterion.  Its
  section-wise formulation is needed only because a slack Gaussian coordinate
  converges to a finite positive constant in the generic limiting model.  The
  following examples give useful non-additive choices.
  \begin{enumerate}[(i)]
    \item \emph{Distorted length capacities.}  For \(c>0\) and
    \(\theta\ge1\),
    \[
      \nu_{\theta,c}(A)=c\,\mes(A)^\theta
    \]
    is affinely homogeneous with \(\lambda_\nu=\theta\).  For every
    \(S>0\), its restriction to \([0,S]\) is regular, its support is the whole
    interval (and therefore carries full capacity), and it satisfies the
    support-intersection property.  It is convex: if
    \(a=\mes(A)\ge b=\mes(B)\) and \(v=\mes(A\cap B)\), then
    \(\mes(A\cup B)=a+b-v\), and the increasing-increments property of
    \(x\mapsto x^\theta\) gives
    \[
      (a+b-v)^\theta+v^\theta\ge a^\theta+b^\theta.
    \]
    For \(\theta>1\) the capacity is non-additive.

    \item \emph{Spread-sensitive interaction capacities.}  Let \(m\ge2\),
    \(\rho\ge0\), and \(\theta\ge1\).  With the convention
    \(M_{m,0}(A)=\mes(A)^m\), define for \(\rho>0\)
    \[
      M_{m,\rho}(A)
      =
      \int_{A^m}
        \prod_{1\le p<q\le m}|t_p-t_q|^\rho
      \,dt_1\cdots dt_m,
      \qquad
      \nu_{m,\rho,\theta}(A)=M_{m,\rho}(A)^\theta .
    \]
    This capacity is affinely homogeneous with
    \[
      \lambda_\nu
      =
      \theta\left(m+\rho\binom{m}{2}\right).
    \]
    On every non-degenerate compact interval the kernel is bounded and
    continuous.  Approximation of a Borel set from inside by compact sets and
    from outside by open sets, together with dominated convergence on the
    corresponding product sets, proves regularity.  Every non-empty relatively
    open set has positive \(M_{m,\rho}\)-capacity, so the support is the whole
    interval (and therefore carries full capacity), and the
    support-intersection property holds.

    The capacity is convex.  Indeed, the pointwise indicator inequality
    \[
      \Ind{(A\cup B)^m}+\Ind{(A\cap B)^m}
      \ge \Ind{A^m}+\Ind{B^m}
    \]
    gives supermodularity of \(M_{m,\rho}\).  Writing \(M=M_{m,\rho}\), if
    \(u=M(A\cup B)\), \(a=M(A)\), \(b=M(B)\), and
    \(v=M(A\cap B)\), then
    \(u\ge a,b\), \(v\le a,b\), and \(u+v\ge a+b\).  Relabeling if necessary,
    suppose \(a\ge b\).  Then \(u\ge a+b-v\), and the
    increasing-increments property of \(x\mapsto x^\theta\) yields
    \[
      u^\theta-a^\theta
      \ge (a+b-v)^\theta-a^\theta
      \ge b^\theta-v^\theta.
    \]
    Hence \(u^\theta+v^\theta\ge a^\theta+b^\theta\), proving convexity of
    \(\nu_{m,\rho,\theta}\).  For \(m=2\) one obtains the
    distance-energy capacity
    \[
      \nu_{\rho,\theta}(A)
      =
      \left(
        \int_A\int_A |s-t|^\rho\,ds\,dt
      \right)^\theta,
      \qquad
      \lambda_\nu=\theta(\rho+2).
    \]
    When \(\rho>0\), separated components create additional positive cross
    terms, so this capacity records more than the total length of the set.

    \item \emph{Section-stable positive-orthant scalarizations.}  For positive
    parameters \(p,q,p_i,\eta_i\), the functions
    \[
      \psi_{\mathrm{prod},\bm p}(\bm x)=\prod_{i=1}^d (x_i)_+^{p_i}
    \]
    and
    \[
      \psi_{\mathrm{soft},p,q,\bm\eta}(\bm x)
      =
      \begin{cases}
        \left(\sum_{i=1}^d \eta_i x_i^{-p}\right)^{-q/p},
        & \min_i x_i>0,\\[3pt]
        0, & \text{otherwise},
      \end{cases}
    \]
    are continuous, coordinatewise non-decreasing, and supported in the
    positive orthant.  Every
    positive coordinate section is strictly increasing above zero: a product
    section is a positive constant times a product of the remaining coordinates,
    while a soft-min section has the form
    \((\kappa+\sum_{i\in D}\eta_i x_i^{-p})^{-q/p}\) on the positive orthant.
    Hence either scalarization can be combined with either capacity in
    \textup{(i)}--\textup{(ii)} to produce \(L\)-admissible Choquet data for
    every \(L>0\).
  \end{enumerate}

  The familiar scalarization
  \(\psi_{\min,p}(\bm x)=(\min_i x_i)_+^p\) has no positive plateau levels in
  full dimension, but it is not section-stable: after a coordinate is fixed at
  \(c>0\), the resulting section has a plateau at \(c^p\).  It is therefore
  covered directly when no slack coordinate is frozen, but not by the universal
  section-wise definition above.  This failure can indeed occur.  For instance,
  take \(E=[0,1]\), \(\nu=\mes\),
  \(\psi(x_1,x_2)=(\min\{x_1,x_2\})_+\), freeze the second coordinate at
  \(c>0\), and let the first path be identically \(2c\).  The resulting
  one-dimensional section integral equals \(c\) for every scalar shift
  \(x\in[0,c]\), so its level-\(c\) shift set has positive Lebesgue measure.
  Thus the section-wise condition prevents an actual no-atom failure in the
  presence of inactive slack coordinates.
\end{remark}

\subsection{Pickands constants}
\label{sec:pickands_constants}

Consider a number \( \alpha \in (0, 2] \) and a real matrix \( V \in \R^{d
\times d} \). Let \( \bm{Y}_{\alpha, V} ( t ), \ t \in \R \) be a multivariate
fractional Brownian motion defined by its covariance matrix
function\footnote{For a generic matrix \(V\), validity requires the conditions
in~\cite[(2.3)]{VVGP} or~\cite[Proposition~9]{Amblard_and_Coeurjolly}.  For the
matrices arising from Assumption~\ref{R2}, and from Assumption~\ref{D2} when
\(\alpha\le\beta\), validity follows directly from
Remarks~\ref{rem:stationary-kernel-validity} and
\ref{rem:nonstationary-kernel-validity}.}
\begin{equation*}
  R_{\alpha, V} ( t, s )
  = S_{\alpha, V} ( t ) + S_{\alpha, V}(-s) - S_{\alpha, V} ( t-s ),
\end{equation*}
where 
\begin{equation*}
    S_{\alpha, V}(t) = |t|^\alpha \, ( V \, \Ind{t \geq 0} + V^\top \, \Ind{t < 0} ).
\end{equation*}
Thus, for example, if \(t\ge s\ge0\), then
\[
  R_{\alpha,V}(t,s)=Vt^\alpha+V^\top s^\alpha-V(t-s)^\alpha.
\]
For each tuple of the following objects
\begin{itemize}
    \item a family of functionals
    \( 
    \Gamma 
    = ( \Gamma_E : C(E, \R^d) \to \R )_{E \in \mathcal{K}} 
    \),
    \item a number \( L \in \R \),
\item a number \( \alpha \in (0, 2] \),
    \item a matrix \( V \in \R^{|I| \times |I|} \),
    \item a function \( \bm{d} : \R \to \R^{|I|} \),
    \item a positive definite matrix \( \Sigma \in \R^{d \times d} \),
    \item a vector \( \bm{b} \in \R^d \setminus (-\infty, 0]^d \),
    \item a compact set \( E \in \mathcal{K} \),
\end{itemize}
define
\begin{equation}
  \label{def pre-Pickands constants}
  H_{\Gamma, L, \alpha, V, \bm{d}, \Sigma, \bm{b}} (E)
  \coloneqq
  \int_{\R^d}
  \exp \left(
  \bm{x}_I^\top \bm{w}_I 
  - \frac{1}{2} \, \bm{x}_J^\top ( \Sigma^{-1} )_{J J} \, \bm{x}_J
  \right)
  \pk{
    \Gamma_E
    \begin{pmatrix}
      \bm{Y}_{\alpha, V} - S_{\alpha,V}(\cdot)\bm{w}_I - \bm{d} - \bm{x}_I \\
      -\bm{x}_{\Jtan} \\
      ( \bb - \bm{b} )_{\Jsl}
    \end{pmatrix}
    > L
  }
  \, d \bm{x},
\end{equation}
where \(\bb\), \(\bm w\), \(I\), \(J\), \(\Jtan\), and \(\Jsl\) are
defined by the quadratic-programming decomposition associated with
\(\Pi_\Sigma(\bm b)\). Here \(S_{\alpha,V}(\cdot)\bm{w}_I\) denotes the
continuous drift \(t \mapsto S_{\alpha,V}(t)\bm{w}_I\).
Let us further define
\begin{align}
  \label{def Pickands constants}
  \mathcal{H}_{\Gamma, L, \alpha, V, \Sigma, \bm{b}}
  & =
  \lim_{S \to \infty} 
  S^{-1} 
  H_{\Gamma, L, \alpha, V, \bm{0}, \Sigma, \bm{b}} ([0, S])
  \\[7pt]
  \label{def Piterbarg constants}
  \mathcal{P}_{\Gamma, L, \alpha, V, \bm{d}, \Sigma, \bm{b}}
  & =
  \lim_{S \to \infty} 
  H_{\Gamma, L, \alpha, V, \bm{d}, \Sigma, \bm{b}} ([0, S])
\end{align}
whenever the limits exist.

For later reference define the generic preconstant associated with a limiting
model. Let \(E\in\mathcal K\), let \(\bm Z\) be an \(\R^{|I|}\)-valued
continuous centered Gaussian field on \(E\), and let \(\bm q:E\to\R^{|I|}\) be
continuous. Define
\begin{equation}
  \label{def generic limiting constant}
  \mathfrak H_{\Gamma,L}^{\bm Z,\bm q}(E)
  =
  \int_{\R^d}
  \exp\left(
    \bm x_I^\top\bm w_I
    -\frac12\bm x_J^\top(\Sigma^{-1})_{JJ}\bm x_J
  \right)
  \pk{
    \Gamma_E
    \begin{pmatrix}
      \bm Z-\bm q-\bm x_I\\
      -\bm x_{\Jtan}\\
      (\bb-\bm b)_{\Jsl}
    \end{pmatrix}
    >L
  }\,d\bm x .
\end{equation}
Below, when Assumptions~\ref{B3}--\ref{B4} are invoked for a limiting model
\((\bm Z,\bm q)\) on \(E\), they mean Assumption~\ref{B3} at the fixed level
\(L\) and Assumption~\ref{B4} applied to the shifted limiting paths that appear
inside \eqref{def generic limiting constant}. We say that the limiting model is
\emph{non-trivial} if \(\mathfrak H_{\Gamma,L}^{\bm Z,\bm q}([0,S])>0\) for at
least one \(S>0\). With \(\bm Z=\bm Y_{\alpha,V}\) and
\(\bm q=S_{\alpha,V}(\cdot)\bm w_I+\bm d\), the quantity in
\eqref{def generic limiting constant} is precisely
\(H_{\Gamma,L,\alpha,V,\bm d,\Sigma,\bm b}(E)\). The deterministic
\(\alpha>\beta\) constant corresponds to \(\bm Z\equiv\bm0\) and
\(\bm q=\bm d_I\).

\section{Stationary case}
\label{sec:stationary_case}

Let \(T>0\), and let \( \bm{X} ( t ), \ t \in [0, T] \) be a continuous
centered \( \R^d\)-valued Gaussian process with covariance matrix function
\begin{equation*}
  R : [0, T]^2 \to \R^{d \times d} :
  R ( t, s ) = \E{\bm{X} ( t ) \, \bm{X}^\top (s)}.
\end{equation*}
In this section we assume that \( \bm{X} \) is stationary, which in terms of \(
R \) means that \( R ( t + s, s ) = \mathcal{R} ( t ) \) for some \( \mathcal{R}
: [0, T] \to \R^{d \times d} \) and all \( t, s, t + s \in [0, T] \). Set
\(\Sigma=\mathcal R(0)=R(0,0)\), and assume explicitly that \(\Sigma\) is
positive definite.  All quadratic-programming objects in this section are
formed with this \(\Sigma\) and the fixed threshold vector \(\bm b\).  We denote
by \( \uu \in (0, \infty)^d \) the normalization vector
\begin{equation}
  \label{def hat u}
  (\uu)_I=u\bm1_I,
  \qquad
  (\uu)_{\Jtan}=\bm1_{\Jtan},
  \qquad
  (\uu)_{\Jsl}=u^{-1}\bm1_{\Jsl}.
\end{equation}
Then \(\Sigma(t)=R(t,t)=\Sigma\) for all \(t\in[0,T]\). We use the standard
extension \(\mathcal R(-t)=\mathcal R(t)^\top\) whenever negative lags appear.
We shall furthermore impose the following assumptions on \(R\):
\begin{enumerate}
  \item [(\namedlabel{R1}{\( \mathbf{R}_1 \)})] For every \(t\in(0,T]\),
  the symmetric part of \(\Sigma_{II}-\mathcal R_{II}(t)\) is positive
  definite; equivalently,
  \[
    \bm z^\top(\Sigma_{II}-\mathcal R_{II}(t))\bm z>0
    \qquad\text{for every }\bm z\ne\bm0.
  \]
  \item [(\namedlabel{R2}{\( \mathbf{R} _2 \)})] There exists a \( d \times d \)
  real matrix \( V \) and a number \(\alpha\in(0,2]\) such that
  \( \bm{w}^\top V\bm{w} > 0 \) and
  \begin{equation*}
    \bigl\|\Sigma-\mathcal R(t)-t^\alpha V\bigr\|_{\mathrm F}
    =o(t^\alpha),
    \qquad t\downarrow0.
  \end{equation*}
\end{enumerate}

\begin{remark}[Validity of the stationary limiting kernel]
\label{rem:stationary-kernel-validity}
Assumption~\ref{R2} automatically supplies the covariance-kernel condition
needed for \(\bm Y_{\alpha,V}\).  Fix an interior point \(t_0\in(0,T)\).  For
any finite collection of real times and all sufficiently small \(q>0\), the
kernel
\[
  C_q(t,s)=q^{-\alpha}\Bigl[
    \mathcal R\bigl(q(t-s)\bigr)
    -\mathcal R(qt)\Sigma^{-1}\mathcal R(-qs)
  \Bigr]
\]
is the covariance kernel of
\(q^{-\alpha/2}\bm X(t_0+q\,\cdot)\) conditional on \(\bm X(t_0)\).  Assumption
\ref{R2}, together with covariance transposition at negative lags, gives
\[
  C_q(t,s)\longrightarrow
  S_{\alpha,V}(t)+S_{\alpha,V}(-s)-S_{\alpha,V}(t-s)
  =R_{\alpha,V}(t,s).
\]
Every finite block matrix generated by the limit is therefore positive
semidefinite.  The same is true for the principal kernel obtained from
\(V_{II}\); its increment bound is of order \(|t-s|^\alpha\), so it has a
continuous Gaussian version.
\end{remark}

\begin{theorem}[Stationary case]
  \label{thm:stationary_case}
  Let \( \bm{X} \) be a continuous stationary Gaussian process satisfying
  assumptions~\ref{R1} and~\ref{R2}. Let \( \Gamma \) satisfy
  assumptions~\ref{B1}--\ref{B2}, let \(L\in\mathcal L\), and assume that
  Assumptions~\ref{B3}--\ref{B4} hold on every interval \([0,S]\), \(S\ge0\),
  for the limiting model
  \[
    \bm Z=\bm Y_{\alpha,V_{I,I}},
    \qquad
    \bm q(t)=S_{\alpha,V_{I,I}}(t)\bm w_I,
  \]
  and that this limiting model is non-trivial.
  With \( L_u = L \cdot u^{-2 \lambda_\Gamma / \alpha} \), we have
  \begin{equation*}
    \pk{\Gamma_{[0, T]} ( \uu ( \bm{X} - u \bm{b} ) ) > L_u}
    \sim
    T \, \mathcal{H}_{\Gamma, L, \alpha, V_{I, I}, \Sigma, \bm{b}} \, 
    u^{2 / \alpha - |I|} \,
    \varphi_{\Sigma} ( u \bb ),
\end{equation*}
  as \( u \to \infty \).
\end{theorem}

\noindent The proof is given in Appendix~\ref{sec:proof_stationary_case}.

\begin{corollary}[Stationary Choquet-integral asymptotics]
\label{cor:stationary-choquet-integral}
Let \(\bm X\) be a continuous stationary Gaussian process satisfying
Assumptions~\ref{R1}--\ref{R2} on \([0,T]\), and fix \(L>0\).  Let
\((\nu,\psi)\) be \(L\)-admissible Choquet data in the sense of
Definition~\ref{def:admissible-choquet-data}, and assume that the stationary
limiting model in Theorem~\ref{thm:stationary_case} is non-trivial for
\(\Gamma^{\nu,\psi}\).  Set
\[
  \mathcal H_{\nu,\psi}^{\mathrm{stat}}(L)
  =
  \mathcal H_{\Gamma^{\nu,\psi},L,\alpha,V_{I,I},\Sigma,\bm b}.
\]
Then, as \(u\to\infty\),
\begin{equation}
  \label{eq:stationary-choquet-corollary}
  \pk{
    \Gamma_{[0,T]}^{\nu,\psi}\bigl(\uu(\bm X-u\bm b)\bigr)
    > L u^{-2\lambda_\nu/\alpha}
  }
  \sim
  T\,\mathcal H_{\nu,\psi}^{\mathrm{stat}}(L)\,
  u^{2/\alpha-|I|}\,
  \varphi_\Sigma(u\bb).
\end{equation}
The constant \(\mathcal H_{\nu,\psi}^{\mathrm{stat}}(L)\) is finite and
positive.
\end{corollary}

\begin{proof}
Take the admissible-level set to be \(\mathcal L=\{L\}\).  The positive-orthant support and coordinatewise non-decreasing requirements in
\textup{(Ch2)}, together with inclusion monotonicity of the capacity, give
Assumption~\ref{B1}; affine homogeneity in \textup{(Ch1)} gives
Assumption~\ref{B2} with \(\lambda_\Gamma=\lambda_\nu\).

It remains to identify correctly the scalarization in the limiting constant.
Put \(D_*=I\cup\Jtan\) and
\(\bm c=(\bb-\bm b)_{\Jsl}\in(0,\infty)^{\Jsl}\).  The path inside
\eqref{def generic limiting constant} is obtained by applying the section
\(\psi_{D_*,\bm c}\) to the \(D_*\)-dimensional continuous path
\[
  \bigl(\bm Y_{\alpha,V_{I,I}}-
  S_{\alpha,V_{I,I}}(\cdot)\bm w_I,\,\bm0_{\Jtan}\bigr)
  -(\bm x_I,\bm x_{\Jtan}).
\]
For every \(S>0\), condition~\textup{(Ch3)} and
Corollary~\ref{cor:choquet-partial-shifts} therefore give
Assumption~\ref{B3} in the full \(\bm x\)-space.  Continuity of \(\psi\) and
Lemma~\ref{sojorn_F3_cap} give Assumption~\ref{B4}; in fact the Choquet
functional is continuous at every continuous path.

For completeness, the degenerate interval \([0,0]\) also causes no problem.
If its capacity is zero, the functional cannot equal the positive level \(L\).
If the relevant section satisfies~\textup{(Ch3)(a)}, the diagonal argument in
Lemma~\ref{sojorn_F2_cap} applies verbatim on a singleton.  Otherwise
\textup{(Ch3)(b)} makes the capacity convex on every non-degenerate interval;
a finite convex translation-invariant capacity assigns zero capacity to a
singleton, since superadditivity over arbitrarily many distinct singleton sets
would otherwise contradict finiteness on a compact interval.  Thus all
hypotheses of Theorem~\ref{thm:stationary_case} are satisfied, and its
conclusion is exactly \eqref{eq:stationary-choquet-corollary}.
\end{proof}

\begin{corollary}[Stationary Parisian persistence]
\label{cor:stationary-parisian}
Let \(F\in\mathcal K_0\), set
\[
  E_u=[0,T]\ominus u^{-2/\alpha}F,
\]
and let \(\bm X\) satisfy the assumptions of
Theorem~\ref{thm:stationary_case}.  Then, as \(u\to\infty\),
\begin{equation}
  \label{eq:stationary-parisian-asymptotic}
  \pk{
    \exists t\in E_u\ \text{such that}\
    \bm X(t+u^{-2/\alpha}s)>u\bm b
    \ \text{for every }s\in F
  }
  \sim
  T\,\mathcal H_F^{\mathrm{Par}}\,
  u^{2/\alpha-|I|}\varphi_\Sigma(u\bb),
\end{equation}
where inequalities between vectors are coordinatewise and
\[
  \mathcal H_F^{\mathrm{Par}}
  =\mathcal H_{\Pi;F,0,\alpha,V_{I,I},\Sigma,\bm b}
  \in(0,\infty).
\]
For \(F=[0,D]\), this is the exact asymptotic for persistence above
\(u\bm b\) throughout a Parisian delay of length \(D u^{-2/\alpha}\), with
the whole delay required to remain inside \([0,T]\).
\end{corollary}

\begin{proof}
The event in \eqref{eq:stationary-parisian-asymptotic} is exactly
\[
  \{\Pi_{E_u;u^{-2/\alpha}F}
       (\uu(\bm X-u\bm b))>0\}.
\]
Remark~\ref{rem:parisian-local-footprint} verifies
Assumptions~\ref{LF1}--\ref{LF3}.  The limiting model is non-trivial: on the
compact set \([0,S]+F\), its continuous active coordinates and deterministic
drift are bounded, so a sufficiently negative active shift and strictly
negative tangential shifts make one complete window positive with positive
probability; the slack coordinates are already strictly positive.  Hence the
local preconstant is positive for some \(S>0\).  Proposition~\ref{prop:local-footprint-transfer}
and Theorem~\ref{thm:stationary_case} give the result.
\end{proof}

\section{Non-stationary case}
\label{sec:non_stationary_case}

In this section let \(T>0\) and assume that \( \bm{X} ( t ),
\ t \in [0, T] \), is non-stationary. Write \(\Sigma(t)=R(t,t)\) and assume
that \(\Sigma(t)\) is positive definite for every \(t\in[0,T]\).  Put
\(\Sigma=\Sigma(0)\); all quadratic-programming objects \(\bb\), \(I\), \(J\),
\(\Jtan\), \(\Jsl\), and \(\bm w\) in this section are formed with this
boundary covariance matrix \(\Sigma\). For \( t \in [0, T] \) we denote the
so-called \emph{inverse generalized variance} of \( \bm{X} \) by
\begin{equation*}
  g ( t ) = \min_{\bm{x} \geq \bm{b}} \bm{x}^\top \Sigma^{-1} (t) \, \bm{x}.
\end{equation*}
We assume that
\begin{enumerate}
    \item [(\namedlabel{D1}{\( \mathbf{D}_1 \)})] \( g \) is continuous and attains its unique minimum at \( 0 \).
    \item [(\namedlabel{D2}{\( \mathbf{D}_2 \)})] There exist \( d \times d \)
    matrices \( A \) and \( V \), and numbers \(\beta>0\),
    \(\alpha\in(0,2]\), such that
    \begin{equation}
      \label{eq:D2-joint-remainder}
      \bigl\|\Sigma-R(t,s)-At^\beta-A^\top s^\beta
        -V|t-s|^\alpha\bigr\|_{\mathrm F}
      =o\bigl(t^\beta+s^\beta+|t-s|^\alpha\bigr)
    \end{equation}
    as \((t,s)\to(0,0)\) in the wedge \(t\ge s\ge0\), with
    \( \bm{w}^\top A \bm{w} > 0 \) and \( \bm{w}^\top V \bm{w} > 0 \).  The
    little-oh in \eqref{eq:D2-joint-remainder} is a joint limit; hence the
    corresponding ratio is uniform on the intersections of the wedge with
    shrinking rectangles.  For \(s>t\), the expansion is obtained from
    \(R(t,s)=R(s,t)^\top\), and the correlation term is then
    \(V^\top|t-s|^\alpha\).
    \item [(\namedlabel{D3}{\( \mathbf{D}_3 \)})] There exists \( \gamma \in [\min\{\alpha,\beta\},2] \) such that for all \(s,t\) in some
    open neighbourhood of \(0\)
    \begin{equation*}
      \E{|\bm{X}(t)-\bm{X}(s)|^2} \lesssim |t-s|^\gamma.
    \end{equation*}
\end{enumerate}

\begin{remark}[Validity of the non-stationary limiting kernel]
\label{rem:nonstationary-kernel-validity}
When \(\alpha\le\beta\), Assumption~\ref{D2} also guarantees that the kernel
used in the stochastic limiting regimes is valid.  For fixed \(s,t\ge0\),
expansion of the conditional covariance at the boundary gives
\[
  q^{-\alpha}\Bigl[
    R(qt,qs)-R(qt,0)\Sigma^{-1}R(0,qs)
  \Bigr]
  \longrightarrow R_{\alpha,V}(t,s),
  \qquad q\downarrow0.
\]
The terms involving \(A\) and \(A^\top\) cancel; the joint remainder in
Assumption~\ref{D2} is \(o(q^\alpha)\) when \(\alpha\le\beta\).  Since every
kernel on the left is a conditional covariance kernel, the limit and its
principal subkernel corresponding to \(V_{II}\) are positive semidefinite and
admit continuous Gaussian versions.  No such stochastic kernel is used in the
deterministic regime \(\alpha>\beta\).
\end{remark}

\begin{theorem}[Non-stationary case]
  \label{thm:main_theorem}
  Let \( \bm{X} \) be a continuous centered Gaussian process satisfying
  assumptions~\ref{D1}--\,\ref{D3}, let \( \Gamma \) satisfy
  assumptions~\ref{B1}--\ref{B2}, and let \( L \in \mathcal{L} \).
  \begin{enumerate}[(a)]
    \item If \( \alpha < \beta \) and Assumptions~\ref{B3}--\ref{B4} hold on
    every interval \([0,S]\), \(S\ge0\), for the limiting model
    \[
      \bm Z=\bm Y_{\alpha,V_{I,I}},
      \qquad
      \bm q(t)=S_{\alpha,V_{I,I}}(t)\bm w_I,
    \]
    and this limiting model is non-trivial,
    then with \( L_u = L \cdot u^{-2 \lambda_\Gamma / \alpha} \) holds
    \begin{equation*}
      \pk{
        \Gamma_{[0, T]} ( \uu ( \bm{X} - u \bm{b} ) ) > L_u
      }
      \sim
      \mathcal{H}_{\Gamma, L, \alpha, V_{I, I}, \Sigma, \bm{b}} \, 
      \Gamma ( 1 / \beta + 1 ) \,
      \tau_{\bm{w}}^{-1 / \beta} \,
      u^{2 / \alpha - 2 / \beta - |I|} \, 
      \varphi_{\Sigma} ( u \bb ),
\end{equation*}
    where \( \tau_{\bm{w}} = \bm{w}^\top A \bm{w} \).
    \item If \( \alpha = \beta \), \(\bm d(t)=t^\beta A\bm w\), and
    Assumptions~\ref{B3}--\ref{B4} hold on every interval \([0,S]\), \(S\ge0\),
    for the limiting model
    \[
      \bm Z=\bm Y_{\alpha,V_{I,I}},
      \qquad
      \bm q(t)=S_{\alpha,V_{I,I}}(t)\bm w_I+\bm d_I(t),
    \]
    and this limiting model is non-trivial,
    then with \( L_u = L \cdot u^{-2 \lambda_\Gamma / \beta} \) we have
    \begin{equation*}
      \pk{
        \Gamma_{[0, T]} ( \uu ( \bm{X} - u \bm{b} ) ) > L_u
      }
      \sim
      \mathcal{P}_{\Gamma, L, \alpha, V_{I, I}, \bm{d}_I, \Sigma, \bm{b}} \, 
      u^{-|I|} \,
      \varphi_{\Sigma} ( u \bb ).
\end{equation*}
    \item If \( \alpha > \beta \), \(\bm d(t)=t^\beta A\bm w\), and
    Assumptions~\ref{B3}--\ref{B4} hold on every interval \([0,S]\), \(S\ge0\),
    for the deterministic limiting model
    \[
      \bm Z\equiv\bm0,
      \qquad
      \bm q(t)=\bm d_I(t),
    \]
    and this limiting model is non-trivial,
    then with \( L_u = L \cdot u^{-2 \lambda_\Gamma / \beta} \) we have
    \begin{equation*}
      \pk{
        \Gamma_{[0, T]} ( \uu ( \bm{X} - u \bm{b} ) ) > L_u
      }
      \sim
      \mathcal{C}_{\Gamma, L, \bm d, \Sigma, \bm{b}} \,
      u^{-|I|} \,
      \varphi_{\Sigma} ( u \bb ),
\end{equation*}
    where
    \begin{equation}
      \label{def constant C main theorem}
      \mathcal{C}_{\Gamma, L, \bm d, \Sigma, \bm{b}}
      = 
      \lim_{S \to \infty} 
      \int_{\R^d}
      \exp \left( 
        \bm{x}_I^\top \bm{w}_I
        - \frac{1}{2} \, \bm{x}_J^\top ( \Sigma^{-1} )_{J J} \, \bm{x}_J
      \right) 
      \Ind{
        \Gamma_{[0, S]} \left( 
          -\bm{d}_I - \bm{x}_I,
          -\bm{x}_{\Jtan},
          ( \bb - \bm{b} )_{\Jsl}
        \right) 
        > L
      }
      \, d \bm{x}.
    \end{equation}
  \end{enumerate}
  The displayed constants are finite and, by the non-triviality assumptions in
  the corresponding cases, positive.
\end{theorem}

\noindent The proof is given in Appendix~\ref{sec:proof_main_theorem}.

\begin{corollary}[Boundary Parisian persistence]
\label{cor:nonstationary-parisian}
Let \(F\in\mathcal K_0^+\), let \(\bm X\) satisfy
Assumptions~\ref{D1}--\ref{D3} on \([0,T]\), and define
\[
  \rho_u=
  \begin{cases}
    u^{-2/\alpha},&\alpha<\beta,\\
    u^{-2/\beta},&\alpha\ge\beta,
  \end{cases}
  \qquad
  E_u=[0,T-\rho_u\max F].
\]
Then, as \(u\to\infty\),
\begin{enumerate}[(a)]
  \item If \(\alpha<\beta\), then
  \begin{equation}
    \label{eq:nonstationary-parisian-alpha-less}
    \pk{
      \exists t\in E_u:\
      \bm X(t+u^{-2/\alpha}s)>u\bm b
      \text{ for every }s\in F
    }
    \sim
    \mathcal H_F^{\mathrm{Par,bd}}
    \Gamma(1/\beta+1)\tau_{\bm w}^{-1/\beta}
    u^{2/\alpha-2/\beta-|I|}\varphi_\Sigma(u\bb).
  \end{equation}

  \item If \(\alpha=\beta\), then
  \begin{equation}
    \label{eq:nonstationary-parisian-alpha-equal}
    \pk{
      \exists t\in E_u:\
      \bm X(t+u^{-2/\beta}s)>u\bm b
      \text{ for every }s\in F
    }
    \sim
    \mathcal P_F^{\mathrm{Par,bd}}
    u^{-|I|}\varphi_\Sigma(u\bb).
  \end{equation}

  \item If \(\alpha>\beta\), then
  \begin{equation}
    \label{eq:nonstationary-parisian-alpha-greater}
    \pk{
      \exists t\in E_u:\
      \bm X(t+u^{-2/\beta}s)>u\bm b
      \text{ for every }s\in F
    }
    \sim
    \mathcal C_F^{\mathrm{Par,bd}}
    u^{-|I|}\varphi_\Sigma(u\bb).
  \end{equation}
\end{enumerate}
Here
\begin{align*}
  \mathcal H_F^{\mathrm{Par,bd}}
  &=\mathcal H_{\Pi;F,0,\alpha,V_{I,I},\Sigma,\bm b},\\
  \mathcal P_F^{\mathrm{Par,bd}}
  &=\mathcal P_{\Pi;F,0,\alpha,V_{I,I},\bm d_I,\Sigma,\bm b},\\
  \mathcal C_F^{\mathrm{Par,bd}}
  &=\mathcal C_{\Pi;F,0,\bm d,\Sigma,\bm b},
  \qquad \bm d(t)=t^\beta A\bm w,
\end{align*}
and the constants are finite and positive in their respective regimes.  The
first successful start time has the corresponding conditional limits from
Theorem~\ref{thm:exceedance_times}, with the truncated constants carrying the
same fixed footprint \(F\).
\end{corollary}

\begin{proof}
For the relevant \(\rho_u\), each event is
\[
  \{\Pi_{E_u;\rho_uF}(\uu(\bm X-u\bm b))>0\}.
\]
Remark~\ref{rem:parisian-local-footprint} verifies all structural and exact
local regularity conditions.  Non-triviality follows by the same compact-path
and negative-shift argument as in the proof of
Corollary~\ref{cor:stationary-parisian}; the argument also covers the
deterministic limiting model in the regime \(\alpha>\beta\).  Proposition~\ref{prop:local-footprint-transfer}
therefore reduces the three claims to parts~(a)--(c) of
Theorem~\ref{thm:main_theorem}.
\end{proof}

\begin{remark}[Which Parisian scale is covered?]
\label{rem:parisian-scales}
In the stationary case and in the boundary regimes \(\alpha<\beta\) and
\(\alpha=\beta\), the footprint has the correlation, or Pickands, scale
\(u^{-2/\alpha}\).  This is the standard non-degenerate shrinking-delay scale
in finite-horizon Gaussian Parisian asymptotics; see, for example,
\cite{DebickiHashorvaJi2016}.  When \(\alpha<\beta\), that footprint is shorter
than the boundary localization scale \(u^{-2/\beta}\), so it is local inside
each long-scale neighborhood.  When \(\alpha=\beta\), the two
scales coincide.

When \(\alpha>\beta\), the present one-scale transfer uses the smaller boundary
scale \(u^{-2/\beta}\).  A footprint of correlation length
\(u^{-2/\alpha}\) is then larger than the entire boundary-localization window
and is not a fixed mark after rescaling by \(u^{-2/\beta}\).  The elementary
variance-cost calculation
\[
  u^2\bigl(u^{-2/\alpha}\bigr)^\beta
  =u^{2-2\beta/\alpha}\longrightarrow\infty
\]
shows why such a delay is not governed by the deterministic constant in
\eqref{eq:nonstationary-parisian-alpha-greater}.  A Pickands-scale delay in
this regime requires a separate two-scale analysis and is not claimed here.
\end{remark}

\begin{corollary}[Boundary non-stationary Choquet-integral asymptotics]
\label{cor:nonstationary-choquet-integral}
Let \(\bm X\) satisfy Assumptions~\ref{D1}--\ref{D3} on \([0,T]\), and fix
\(L>0\).  Let \((\nu,\psi)\) be \(L\)-admissible Choquet data in the sense of
Definition~\ref{def:admissible-choquet-data}, and put
\(\Gamma=\Gamma^{\nu,\psi}\).  In each of the following regimes assume that the
corresponding limiting model in Theorem~\ref{thm:main_theorem} is non-trivial
for this Choquet functional.  With \(\bm d(t)=t^\beta A\bm w\), define
\begin{align*}
  \mathcal H_{\nu,\psi}^{\mathrm{bd}}(L)
  & =
  \mathcal H_{\Gamma,L,\alpha,V_{I,I},\Sigma,\bm b},\\
  \mathcal P_{\nu,\psi}^{\mathrm{bd}}(L)
  & =
  \mathcal P_{\Gamma,L,\alpha,V_{I,I},\bm d_I,\Sigma,\bm b},\\
  \mathcal C_{\nu,\psi}^{\mathrm{bd}}(L)
  & =
  \mathcal C_{\Gamma,L,\bm d,\Sigma,\bm b}.
\end{align*}
Then the following asymptotics hold as \(u\to\infty\).
\begin{enumerate}[(a)]
  \item If \(\alpha<\beta\), then
  \begin{equation}
    \label{eq:nonstationary-choquet-alpha-less-corollary}
    \pk{
      \Gamma_{[0,T]}^{\nu,\psi}\bigl(\uu(\bm X-u\bm b)\bigr)
      > L u^{-2\lambda_\nu/\alpha}
    }
    \sim
    \mathcal H_{\nu,\psi}^{\mathrm{bd}}(L)\,
    \Gamma(1/\beta+1)\,
    \tau_{\bm w}^{-1/\beta}\,
    u^{2/\alpha-2/\beta-|I|}\,
    \varphi_\Sigma(u\bb).
  \end{equation}
  \item If \(\alpha=\beta\), then
  \begin{equation}
    \label{eq:nonstationary-choquet-alpha-equal-corollary}
    \pk{
      \Gamma_{[0,T]}^{\nu,\psi}\bigl(\uu(\bm X-u\bm b)\bigr)
      > L u^{-2\lambda_\nu/\beta}
    }
    \sim
    \mathcal P_{\nu,\psi}^{\mathrm{bd}}(L)\,
    u^{-|I|}\,
    \varphi_\Sigma(u\bb).
  \end{equation}
  \item If \(\alpha>\beta\), then
  \begin{equation}
    \label{eq:nonstationary-choquet-alpha-greater-corollary}
    \pk{
      \Gamma_{[0,T]}^{\nu,\psi}\bigl(\uu(\bm X-u\bm b)\bigr)
      > L u^{-2\lambda_\nu/\beta}
    }
    \sim
    \mathcal C_{\nu,\psi}^{\mathrm{bd}}(L)\,
    u^{-|I|}\,
    \varphi_\Sigma(u\bb).
  \end{equation}
\end{enumerate}
The constants displayed above are finite and positive in the regimes in which
they are used.
\end{corollary}

\begin{proof}
Take again \(\mathcal L=\{L\}\).  As in the stationary corollary,
\textup{(Ch1)}--\textup{(Ch2)} and the verification in
Appendix~\ref{sec:proof_choquet_claims} give Assumptions~\ref{B1}--\ref{B2}
with \(\lambda_\Gamma=\lambda_\nu\).  As in the stationary proof, put
\(D_*=I\cup\Jtan\) and \(\bm c=(\bb-\bm b)_{\Jsl}\).  In all three regimes the
functional appearing in the limiting constant reduces to the section
\(\psi_{D_*,\bm c}\) applied to a continuous \(D_*\)-dimensional path and
shifted by \((\bm x_I,\bm x_{\Jtan})\).  Hence, on every interval
\([0,S]\), \(S>0\), condition~\textup{(Ch3)} and
Corollary~\ref{cor:choquet-partial-shifts} give Assumption~\ref{B3} in the full
shift space.  Continuity of \(\psi\) and
Lemma~\ref{sojorn_F3_cap} give Assumption~\ref{B4} for the Gaussian and
deterministic limiting paths alike.  The preceding singleton argument verifies
the degenerate interval as well.  Applying parts~(a)--(c) of
Theorem~\ref{thm:main_theorem} yields the three formulae, with the threshold
powers obtained by substituting \(\lambda_\Gamma=\lambda_\nu\).
\end{proof}

\begin{remark}[Power capacities]
\label{rem:power-capacity-choquet-corollaries}
For \(\theta\ge1\), the capacity
\(\nu_\theta(A)=\mes(A)^\theta\) is affinely homogeneous with exponent
\(\lambda_\nu=\theta\).  For every \(S>0\), its restriction to \([0,S]\) is a
regular finite convex capacity.  Moreover, every non-empty relatively open
set \(U\subset[0,S]\) has positive Lebesgue measure and therefore
\(\nu_\theta(U)>0\).  Consequently,
\(\supp(\nu_\theta|_{[0,S]})=[0,S]\); the support is connected, carries full
capacity, and the support-intersection property holds.  Therefore the
preceding two corollaries apply to
\[
  \Gamma_E^{\theta,\psi}(\bm f)
  =
  \int_0^\infty
  \mes\{t\in E:\psi(\bm f(t))>r\}^\theta\,dr
\]
whenever \(\psi\) is continuous, coordinatewise non-decreasing, supported in
the positive orthant,
and every positive coordinate section satisfies one of the alternatives in
\textup{(Ch3)}.  The product and soft-min scalarizations in
Remark~\ref{rem:admissible-choquet-data-examples} satisfy the first alternative.
In particular, the scalar choice \(\psi(x)=x_+\) gives a non-additive
Choquet integral for \(\theta>1\), with thresholds
\(L u^{-2\theta/\alpha}\) in the stationary and \(\alpha<\beta\) boundary
regimes, and thresholds \(L u^{-2\theta/\beta}\) in the
\(\alpha\ge\beta\) boundary regimes.
\end{remark}

\begin{remark}[No-slack simplification]
\label{rem:choquet-no-slack}
If \(\Jsl=\emptyset\), no coordinate is frozen at a finite positive
value in the limiting constant.  In both Choquet corollaries, the
\(L\)-admissibility hypothesis may then be weakened as follows: retain
\textup{(Ch1)}--\textup{(Ch2)}, but require \textup{(Ch3)} only for the full
scalarization \(\psi\), rather than for all positive coordinate sections.
Indeed, the proofs use only the section indexed by
\(D_*=I\cup\Jtan=\{1,\ldots,d\}\).  Consequently, with a convex full-support
capacity, scalarizations such as \((\min_i x_i)_+^p\) are covered in the
no-slack case; their full-dimensional positive level sets have Lebesgue measure
zero.
\end{remark}

\section{Exceedance times}
\label{sec:exceedance-times}

Let the exceedance time \( \mathfrak{t} ( u, L ) \) on the observation window
\([0,T]\) be defined by
\begin{equation*}
  \mathfrak{t} ( u, L ) = \inf \left\{ 
    0<t\le T : 
    \Gamma_{[0, t]} ( \uu ( \bm{X} - u \bm{b} ) ) 
    > L_u
  \right\},
\end{equation*}
with the convention \(\inf\emptyset=\infty\).

\begin{theorem}[Exceedance times]
    \label{thm:exceedance_times}
    Let \( \mathfrak t(u,L) \) be defined as above, with the threshold sequence
    \(L_u\) chosen as in the corresponding asymptotic theorem.
    \begin{enumerate}[(a),wide]
        \item \textbf{Stationary case.} Under the assumptions of
        Theorem~\ref{thm:stationary_case}, for every \(s\in[0,T]\),
        \begin{equation*}
          \pk{
            \mathfrak t(u,L)\le s
            \,\middle|\,
            \mathfrak t(u,L)\le T
          }
          \to \frac{s}{T}.
        \end{equation*}

        \item \textbf{Non-stationary case \(\alpha<\beta\).} Under the
        assumptions of Theorem~\ref{thm:main_theorem}(a), for every
        \(S\geq0\),
        \begin{equation*}
            \pk{
              u^{2/\beta}\mathfrak{t}(u,L)\leq S
              \,\middle|\,
              \mathfrak{t}(u,L)\leq T
            }
            \to
            \frac{
                \Gamma (1 / \beta) - \Gamma (1 / \beta, \tau_{\bm{w}} S^\beta )
            }{
                \Gamma (1 / \beta)
            }.
        \end{equation*}

        \item \textbf{Non-stationary case \(\alpha=\beta\).} Under the
        assumptions of Theorem~\ref{thm:main_theorem}(b), put
        \begin{equation*}
          G_{\mathfrak P}(S)
          =
          \frac{
          H_{\Gamma,L,\alpha,V_{I,I},\bm d_I,\Sigma,\bm b}([0,S])
          }{
          \mathcal{P}_{\Gamma,L,\alpha,V_{I,I},\bm d_I,\Sigma,\bm b}
          },
          \qquad S\ge0,
        \end{equation*}
        and define its right-continuous modification, extended to \(\R\), by
        \begin{equation*}
          F_{\mathfrak P}(S)
          =
          \begin{cases}
            0, & S<0,\\[2pt]
            \displaystyle\lim_{r\downarrow S}G_{\mathfrak P}(r), & S\ge0.
          \end{cases}
        \end{equation*}
        Assumption~\ref{B1} and the definition of the Piterbarg constant imply
        that \(G_{\mathfrak P}\) is non-decreasing and tends to \(1\); hence
        \(F_{\mathfrak P}\) is a distribution function.  Then, at every
        continuity point \(S\ge0\) of \(F_{\mathfrak P}\),
        \begin{equation*}
            \pk{
              u^{2/\beta}\mathfrak{t}(u,L)\leq S
              \,\middle|\,
              \mathfrak{t}(u,L)\leq T
            }
            \to F_{\mathfrak P}(S).
        \end{equation*}
        Consequently, if \(0\le a<r\) are continuity points of
        \(F_{\mathfrak P}\), then
        \begin{equation*}
            \pk{
              u^{2/\beta}\mathfrak t(u,L)\in(a,r]
              \,\middle|\,
              \mathfrak t(u,L)\le T
            }
            \to F_{\mathfrak P}(r)-F_{\mathfrak P}(a).
        \end{equation*}

        \item \textbf{Non-stationary case \(\alpha>\beta\).} Under the
        assumptions of Theorem~\ref{thm:main_theorem}(c), define
        \begin{equation*}
          \mathfrak C(S)
          =
          \int_{\R^d}
          \exp \left(
            \bm{x}_I^\top \bm{w}_I
            - \frac{1}{2} \, \bm{x}_J^\top ( \Sigma^{-1} )_{J J} \, \bm{x}_J
          \right)
          \Ind{
            \Gamma_{[0,S]} \left(
              -\bm{d}_I - \bm{x}_I,
              -\bm{x}_{\Jtan},
              ( \bb - \bm{b} )_{\Jsl}
            \right)>L
          }
          \,d\bm x,
          \qquad S\ge0,
        \end{equation*}
        put
        \begin{equation*}
          G_{\mathfrak C}(S)
          =
          \frac{\mathfrak C(S)}{\mathcal C_{\Gamma,L,\bm d,\Sigma,\bm b}},
          \qquad S\ge0,
        \end{equation*}
        and define its right-continuous modification, extended to \(\R\), by
        \begin{equation*}
          F_{\mathfrak C}(S)
          =
          \begin{cases}
            0, & S<0,\\[2pt]
            \displaystyle\lim_{r\downarrow S}G_{\mathfrak C}(r), & S\ge0.
          \end{cases}
        \end{equation*}
        Assumption~\ref{B1} and the definition of the deterministic constant
        imply that \(G_{\mathfrak C}\) is non-decreasing and tends to \(1\);
        hence \(F_{\mathfrak C}\) is a distribution function.  Then, at every
        continuity point \(S\ge0\) of \(F_{\mathfrak C}\),
        \begin{equation*}
            \pk{
              u^{2/\beta}\mathfrak{t}(u,L)\leq S
              \,\middle|\,
              \mathfrak{t}(u,L)\leq T
            }
            \to F_{\mathfrak C}(S).
        \end{equation*}
        Consequently, if \(0\le a<r\) are continuity points of
        \(F_{\mathfrak C}\), then
        \begin{equation*}
            \pk{
              u^{2/\beta}\mathfrak t(u,L)\in(a,r]
              \,\middle|\,
              \mathfrak t(u,L)\le T
            }
            \to F_{\mathfrak C}(r)-F_{\mathfrak C}(a).
        \end{equation*}
    \end{enumerate}
\end{theorem}

\begin{remark}
    Note that
    \begin{equation*}
    \frac{
        \Gamma (1 / \beta) - \Gamma (1 / \beta, \tau_{\bm{w}} S^\beta )
    }{
        \Gamma (1 / \beta)
    }
    = \pk{
      Y \leq \tau_{\bm{w}} S^\beta
    },
    \end{equation*}
    where \( Y \sim \operatorname{Gamma} ( 1 / \beta, 1 ) \).
\end{remark}

\noindent The proof is given in Appendix~\ref{sec:proof_exceedance_times}.

\section{Examples}
\label{sec:examples}

This section gives two direct applications of the main theorems.  The first
one is a vector-valued stationary example for an integrated positive
minimum.  The second one is a scalar non-stationary boundary example for a
non-additive Choquet integral; it displays the three regimes
\(\alpha<\beta\), \(\alpha=\beta\), and \(\alpha>\beta\), and in the
deterministic regime gives the leading constant in closed form.  In the first
example the covariance matrix at the origin is \(\Sigma=I_d\) and
\(\bm b\in(0,\infty)^d\), while the second example is scalar with \(b>0\).
Hence the quadratic programme is coordinatewise separable in the first example
and one-dimensional in the second: in both cases \(\bb=\bm b\), the active set
is the full coordinate set, and there are no inactive coordinates.  Strict
positivity of \(\bm b\) alone would not imply this for a general covariance
matrix.

\subsection{Integral and Choquet functionals used below}

The first example uses the integral functional
\begin{equation}
  \label{eq:area-functional-example}
  \Gamma_E^{\mathrm{area}}(\bm f)
  =
  \int_E \min_{1\le i\le d} (f_i(t))_+\,dt,
  \qquad (x)_+=\max(x,0).
\end{equation}
It satisfies Assumptions~\ref{B1}--\ref{B2} with
\(\mathcal L=(0,\infty)\) and \(\lambda_\Gamma=1\).  For the continuous
limiting Gaussian fields appearing below, Assumptions~\ref{B3}--\ref{B4} at any
fixed level \(L>0\) follow from the integral-functional lemmas in
Appendix~\ref{sec:auxiliary_results}: the scalarization
\(\bm x\mapsto \min_i(x_i)_+\) is coordinatewise non-decreasing,
bounded on compact sets, and has no positive plateau levels.  In dimension
one this reduces to
\(\Gamma_E^{\mathrm{area}}(f)=\int_E f(t)_+\,dt\).

We shall also use a Choquet-type functional.  Fix
\(\theta\ge1\).  For Borel \(A\subset\R\), set
\(\nu_\theta(A)=\mes(A)^\theta\), allowing the value \(+\infty\) on
unbounded sets.  Its restriction to every compact interval is a finite regular
capacity, and it is non-additive when \(\theta>1\).  For scalar continuous
paths define
\begin{equation}
  \label{eq:choquet-power-functional-example}
  \Gamma_E^{\mathrm{Ch},\theta}(f)
  =
  \int_0^\infty
  \nu_\theta\{t\in E:f(t)>r\}\,dr
  =
  \int_0^\infty
  \mes\{t\in E:f(t)>r\}^\theta\,dr .
\end{equation}
Equivalently, this is the Choquet integral of \(f_+\) with respect to
\(\nu_\theta\).  Since
\(\nu_\theta(aA+b)=a^\theta\nu_\theta(A)\), the affine scaling property holds
with \(\lambda_\Gamma=\theta\).  The capacity is convex for
\(\theta\ge1\), because \(x\mapsto x^\theta\) has increasing increments.  On every
non-degenerate compact interval \(E\), each non-empty relatively
open subset of \(E\) has positive Lebesgue measure, so
\(\supp(\nu_\theta|_E)=E\).  Thus the support is connected, carries full
capacity, and the support-intersection property holds.  The scalarization
\(x_+\) is continuous, non-decreasing, supported on \((0,\infty)\), and strictly
increasing above zero.
Thus \((\nu_\theta,x_+)\) is \(L\)-admissible Choquet data in the sense of
Definition~\ref{def:admissible-choquet-data} for every \(L>0\), and
Corollaries~\ref{cor:stationary-choquet-integral}--\ref{cor:nonstationary-choquet-integral}
apply with \(\lambda_\Gamma=\theta\).  The case \(\theta=1\) is exactly the
ordinary area functional, while \(\theta>1\) gives a non-additive Choquet
integral that emphasizes longer level sets.  In this scalar case there are no proper positive coordinate sections,
so the section-wise condition introduces no additional restriction.

For comparison, the classical vector sojourn functional
\[
  \Gamma_E^{\mathrm{soj}}(\bm f)
  = \int_E \Ind{\bm f(t)>\bm0}\,dt
\]
has the same scaling exponent as \eqref{eq:area-functional-example}, namely
\(\lambda_\Gamma=1\).  The formulae below may be read with this functional as
well whenever the corresponding no-atom assumptions are verified for the
limiting model.

\subsection{Stationary vector example: multivariate area above a high level}
\label{sec:stationary-vector-example}

Fix \(d\ge1\), \(T>0\), \(L>0\), \(\alpha\in(0,2]\), and vectors
\(\bm b=(b_1,\ldots,b_d)^\top\), \(\bm c=(c_1,\ldots,c_d)^\top\) with
\(b_i>0\) and \(c_i>0\).  Let \(Z_1,\ldots,Z_d\) be independent centered
stationary Gaussian processes with unit variance and correlation functions
\[
  r_i(t)=\exp(-c_i |t|^\alpha),
  \qquad t\in\mathbb R,
\]
and take their continuous versions, which exist by Kolmogorov's criterion.  Set
\(\bm X(t)=(Z_1(t),\ldots,Z_d(t))^\top\).  Define the accumulated positive
multivariate excess
\begin{equation*}
  A_T(u)
  =
  \int_0^T
  \min_{1\le i\le d} \bigl(u(X_i(t)-u b_i)\bigr)_+\,dt .
\end{equation*}
Let \(B_{\alpha/2,1},\ldots,B_{\alpha/2,d}\) be independent standard
fractional Brownian motions with Hurst index \(\alpha/2\), and put
\[
  Y_i(t)=\sqrt{2c_i}\,B_{\alpha/2,i}(t),
  \qquad t\ge0.
\]
Equivalently,
\(\cov(Y_i(t),Y_i(s))=c_i(t^\alpha+s^\alpha-|t-s|^\alpha)\) for
\(s,t\ge0\).  Define
\begin{equation}
  \label{def:H-area-stationary-example}
  \mathcal H_{\alpha,\bm c,\bm b}^{\mathrm{area}}(L)
  =
  \lim_{S\to\infty} S^{-1}
  \int_{\R^d} e^{\bm x^\top\bm b}
  \pk{
    \int_0^S
    \min_{1\le i\le d}
    \bigl(Y_i(t)-c_i b_i t^\alpha-x_i\bigr)_+\,dt
    >L
  }
  \,d\bm x .
\end{equation}

\begin{proposition}[Stationary vector area asymptotic]
\label{prop:stationary-vector-area-example}
For the process above,
\begin{equation}
  \label{eq:stationary-vector-area-result}
  \pk{ A_T(u)>L u^{-2/\alpha} }
  \sim
  T\,\mathcal H_{\alpha,\bm c,\bm b}^{\mathrm{area}}(L)\,
  u^{2/\alpha-d}\,(2\pi)^{-d/2}
  \exp\left(-\frac{u^2}{2}\sum_{i=1}^d b_i^2\right).
\end{equation}
The constant in \eqref{def:H-area-stationary-example} is finite and positive.
Moreover, if
\[
  \mathfrak t_A(u,L)
  =
  \inf\left\{0<t\le T:
  \int_0^t
  \min_{1\le i\le d}\bigl(u(X_i(s)-u b_i)\bigr)_+\,ds
  >L u^{-2/\alpha}
  \right\},
\]
with \(\inf\emptyset=\infty\), then for every \(s\in[0,T]\),
\begin{equation}
  \label{eq:stationary-vector-area-time-result}
  \pk{
    \mathfrak t_A(u,L)\le s
    \,\middle|\,
    \mathfrak t_A(u,L)\le T
  }
  \to \frac{s}{T}.
\end{equation}
\end{proposition}

\begin{proof}
Here \(\Sigma=I_d\), and since \(\bm b>\bm0\), the quadratic programme has
\(\bb=\bm b\), \(I=\{1,\ldots,d\}\), \(J=\emptyset\), and
\(\bm w=\bm b\).  The covariance matrix is
\[
  \mathcal R(t)=\diag(e^{-c_1|t|^\alpha},\ldots,e^{-c_d|t|^\alpha}),
\]
so
\[
  \bigl\|I_d-\mathcal R(t)
    -|t|^\alpha\diag(c_1,\ldots,c_d)\bigr\|_{\mathrm F}
  =o(|t|^\alpha),
  \qquad t\downarrow0.
\]
Thus Assumption~\ref{R2} holds with
\(V=\diag(c_1,\ldots,c_d)\), and Assumption~\ref{R1} follows because
\(1-e^{-c_i|t|^\alpha}>0\) for \(t>0\).  The local field
\(\bm Y_{\alpha,V}\) in Section~\ref{sec:pickands_constants} is exactly
\((Y_1,\ldots,Y_d)^\top\), and
\(S_{\alpha,V}(t)\bm w=(c_1b_1t^\alpha,\ldots,c_db_dt^\alpha)^\top\) for
\(t\ge0\).  The limiting model is non-trivial: choosing all coordinates of
\(\bm x\) sufficiently negative and using continuity of \(\bm Y_{\alpha,V}\)
gives a positive probability that the integral in
\eqref{def:H-area-stationary-example} exceeds \(L\) on a short interval.
Applying Theorem~\ref{thm:stationary_case} to \(\Gamma^{\mathrm{area}}\) gives
\eqref{eq:stationary-vector-area-result}, and
Theorem~\ref{thm:exceedance_times}(a) gives
\eqref{eq:stationary-vector-area-time-result}.
\end{proof}

\subsection{Non-stationary boundary example: a Choquet integral}
\label{sec:nonstationary-choquet-boundary-example}

We now keep the model one-dimensional in order to display the constants more
explicitly.  Fix \(T>0\), \(L>0\), \(a>0\), \(b>0\), \(c>0\),
\(\alpha\in(0,2]\), \(\beta>0\), and \(\theta\ge1\).  Let \(Z\) be a centered
stationary Gaussian process with unit variance and correlation
\(r(t)=\exp(-c|t|^\alpha)\), again taken with a continuous version.  Put
\begin{equation}
  \label{eq:nonstationary-model-example}
  X(t)=\frac{Z(t)}{1+a t^\beta},
  \qquad 0\le t\le T,
\end{equation}
and define the Choquet excess functional
\begin{equation}
  \label{eq:choquet-excess-example}
  C_T^{(\theta)}(u)
  =
  \Gamma_{[0,T]}^{\mathrm{Ch},\theta}\bigl(u(X-u b)\bigr)
  =
  \int_0^\infty
  \mes\{t\in[0,T]:u(X(t)-u b)>r\}^\theta\,dr .
\end{equation}
Let \(B_{\alpha/2}\) be a standard fractional Brownian motion with Hurst index
\(\alpha/2\), and write
\[
  Y(t)=\sqrt{2c}\,B_{\alpha/2}(t),
  \qquad
  \tau=a b^2 .
\]
Define the Pickands-type and Piterbarg-type Choquet constants
\begin{align}
  \label{def:H-choquet-nonstationary-example}
  \mathcal H_{\alpha,c,b}^{\mathrm{Ch},\theta}(L)
  & =
  \lim_{S\to\infty} S^{-1}
  \int_\R e^{bx}
  \pk{
    \Gamma_{[0,S]}^{\mathrm{Ch},\theta}
    \bigl(Y(t)-c b t^\alpha-x\bigr)>L
  }\,dx, \\
  \label{def:P-choquet-nonstationary-example}
  \mathcal P_{\alpha,a,c,b}^{\mathrm{Ch},\theta}(L)
  & =
  \lim_{S\to\infty}
  \int_\R e^{bx}
  \pk{
    \Gamma_{[0,S]}^{\mathrm{Ch},\theta}
    \bigl(Y(t)-(a+c)b t^\alpha-x\bigr)>L
  }\,dx .
\end{align}
The second constant is used only in the case \(\alpha=\beta\).  Finally, set
\begin{equation}
  \label{def:ell-choquet-example}
  \ell_{L,\theta}
  =
  \left(
    \frac{\beta+\theta}{\beta}\,L\,(ab)^{\theta/\beta}
  \right)^{\beta/(\beta+\theta)} .
\end{equation}

\begin{proposition}[Boundary variance trend for a powered Choquet integral]
\label{prop:nonstationary-choquet-example}
For the process \eqref{eq:nonstationary-model-example} and the functional
\eqref{eq:choquet-excess-example}, the following asymptotics hold as
\(u\to\infty\).  The constants in
\eqref{def:H-choquet-nonstationary-example} and
\eqref{def:P-choquet-nonstationary-example} exist, are finite and positive in
the regimes in which they are used.
\begin{enumerate}[(i)]
  \item If \(\alpha<\beta\), then
  \begin{equation}
    \label{eq:choquet-alpha-less-beta-example}
    \pk{ C_T^{(\theta)}(u)>L u^{-2\theta/\alpha} }
    \sim
    \mathcal H_{\alpha,c,b}^{\mathrm{Ch},\theta}(L)\,
    \Gamma(1/\beta+1)\,(ab^2)^{-1/\beta}\,
    u^{2/\alpha-2/\beta-1}\,
    \frac{1}{\sqrt{2\pi}}e^{-b^2u^2/2}.
  \end{equation}
  \item If \(\alpha=\beta\), then
  \begin{equation}
    \label{eq:choquet-alpha-equal-beta-example}
    \pk{ C_T^{(\theta)}(u)>L u^{-2\theta/\beta} }
    \sim
    \mathcal P_{\alpha,a,c,b}^{\mathrm{Ch},\theta}(L)\,
    u^{-1}\,
    \frac{1}{\sqrt{2\pi}}e^{-b^2u^2/2}.
  \end{equation}
  \item If \(\alpha>\beta\), then the Gaussian local fluctuation no longer
  appears in the leading constant and
  \begin{equation}
    \label{eq:choquet-alpha-greater-beta-example}
    \pk{ C_T^{(\theta)}(u)>L u^{-2\theta/\beta} }
    \sim
    \frac{1}{b}e^{-b\ell_{L,\theta}}\,
    u^{-1}\,
    \frac{1}{\sqrt{2\pi}}e^{-b^2u^2/2}.
  \end{equation}
\end{enumerate}
Let
\[
  \mathfrak t_{\theta}(u,L)
  =
  \inf\left\{0<t\le T:
  \Gamma_{[0,t]}^{\mathrm{Ch},\theta}\bigl(u(X-u b)\bigr)>L_u
  \right\},
\]
where \(L_u=L u^{-2\theta/\alpha}\) in the case \(\alpha<\beta\) and
\(L_u=L u^{-2\theta/\beta}\) in the cases \(\alpha\ge\beta\).  If
\(\alpha<\beta\), then for every \(S\ge0\),
\begin{equation}
  \label{eq:choquet-time-alpha-less-beta-example}
  \pk{
    u^{2/\beta}\mathfrak t_{\theta}(u,L)\le S
    \,\middle|\,
    \mathfrak t_{\theta}(u,L)\le T
  }
  \to
  \frac{\Gamma(1/\beta)-\Gamma(1/\beta,ab^2S^\beta)}{\Gamma(1/\beta)}.
\end{equation}
If \(\alpha>\beta\), define, for \(S>0\), the unique number
\(y_{S,\theta}>0\) by
\begin{equation}
  \label{eq:yS-choquet-example}
  \int_0^\infty
  \mes\{t\in[0,S]:y_{S,\theta}-ab t^\beta>r\}^\theta\,dr=L .
\end{equation}
Then, for every \(S>0\),
\begin{equation}
  \label{eq:choquet-time-alpha-greater-beta-example}
  \pk{
    u^{2/\beta}\mathfrak t_{\theta}(u,L)\le S
    \,\middle|\,
    \mathfrak t_{\theta}(u,L)\le T
  }
  \to
  \exp\{-b(y_{S,\theta}-\ell_{L,\theta})\}.
\end{equation}
\end{proposition}

\begin{proof}
The covariance of \(X\) is
\[
  R(t,s)=\frac{e^{-c|t-s|^\alpha}}{(1+at^\beta)(1+as^\beta)} .
\]
Thus \(\Sigma=R(0,0)=1\) and, as \(t,s\to0\) with \(t\ge s\ge0\),
\begin{equation}
  \label{eq:nonstationary-cov-exp-example}
  1-R(t,s)
  = at^\beta+as^\beta+c|t-s|^\alpha
  +o\bigl(t^\beta+s^\beta+|t-s|^\alpha\bigr).
\end{equation}
Consequently Assumption~\ref{D2} holds with \(A=a\), \(V=c\), and
\(\tau_{\bm w}=ab^2\).  Since \(b>0\), the inverse generalized variance is
\[
  g(t)=b^2(1+at^\beta)^2,
\]
which has its unique minimum at \(0\), proving Assumption~\ref{D1}.  For
Assumption~\ref{D3}, write
\[
  X(t)-X(s)
  =
  \frac{Z(t)-Z(s)}{1+at^\beta}
  +Z(s)\left(\frac{1}{1+at^\beta}-\frac{1}{1+as^\beta}\right).
\]
Near the origin,
\(\E{(Z(t)-Z(s))^2}\lesssim |t-s|^\alpha\) and
\(|t^\beta-s^\beta|\lesssim |t-s|^{\min\{\beta,1\}}\).  Hence
\[
  \E{(X(t)-X(s))^2}
  \lesssim |t-s|^\alpha+|t-s|^{2\min\{\beta,1\}}
  \lesssim |t-s|^{\min\{\alpha,\beta\}},
\]
where the last step uses \(\alpha\le2\).
Thus Assumption~\ref{D3} holds with \(\gamma=\min\{\alpha,\beta\}\).  The
functional assumptions required by Theorem~\ref{thm:main_theorem} were verified
above for \(\Gamma^{\mathrm{Ch},\theta}\).

The Pickands and Piterbarg limiting models are non-trivial by the same
negative-shift and continuity argument used in the stationary example: after
choosing \(x<0\) sufficiently negative, the continuous limiting path is positive
on a short initial interval with positive probability.  The cases
\(\alpha<\beta\) and \(\alpha=\beta\) now follow immediately from
Theorem~\ref{thm:main_theorem}(a) and Theorem~\ref{thm:main_theorem}(b), with
\(\lambda_\Gamma=\theta\) and
\(\varphi_1(ub)=(2\pi)^{-1/2}e^{-b^2u^2/2}\).  In the case
\(\alpha>\beta\), Theorem~\ref{thm:main_theorem}(c) gives the deterministic
constant as
\[
  \mathcal C_\theta
  =
  \lim_{S\to\infty}
  \int_\R e^{bx}
  \Ind{\Gamma_{[0,S]}^{\mathrm{Ch},\theta}(-abt^\beta-x)>L}\,dx .
\]
For \(y\ge0\), monotone convergence of the level-set lengths gives
\begin{align*}
  J_\infty(y)
  &\coloneqq
  \lim_{S\to\infty}
  \Gamma_{[0,S]}^{\mathrm{Ch},\theta}\bigl((y-abt^\beta)_+\bigr) \\
  & =
  \int_0^y
  \mes\{t\ge0:y-abt^\beta>r\}^\theta\,dr \\
  & =
  \int_0^y
  \left(\frac{y-r}{ab}\right)^{\theta/\beta}\,dr
  =
  \frac{\beta}{\beta+\theta}(ab)^{-\theta/\beta}y^{1+\theta/\beta}.
\end{align*}
For \(y<0\), put \(J_\infty(y)=0\).  Hence the limiting indicator is
\(\Ind{J_\infty(-x)>L}\), which is equivalent to
\(-x>\ell_{L,\theta}\), where \(\ell_{L,\theta}\) is defined in
\eqref{def:ell-choquet-example}.  A second application of monotone convergence
therefore yields
\[
  \mathcal C_\theta
  =
  \int_{-\infty}^{-\ell_{L,\theta}}e^{bx}\,dx
  =
  \frac1b e^{-b\ell_{L,\theta}}.
\]
This proves \eqref{eq:choquet-alpha-greater-beta-example}.

The exceedance-time limit \eqref{eq:choquet-time-alpha-less-beta-example} is
exactly Theorem~\ref{thm:exceedance_times}(b).  For \(\alpha>\beta\), put
\[
  J_S(y)=
  \int_0^\infty
  \mes\{t\in[0,S]:y-abt^\beta>r\}^\theta\,dr .
\]
For \(S>0\) and \(y\ge0\), the level-set length gives the explicit
representation
\[
  J_S(y)
  =
  \int_0^y
  \min\left\{S,\left(\frac{y-r}{ab}\right)^{1/\beta}\right\}^{\theta}\,dr.
\]
It follows directly that \(y\mapsto J_S(y)\) is continuous and strictly
increasing from \(0\) to \(\infty\), so \eqref{eq:yS-choquet-example} defines a
unique positive \(y_{S,\theta}\).  The same representation shows that
\((S,y)\mapsto J_S(y)\) is continuous on \((0,\infty)^2\), and the strict
monotonicity in \(y\) implies that \(S\mapsto y_{S,\theta}\) is continuous.
Moreover, \(J_S(y)\uparrow J_\infty(y)\) as \(S\to\infty\), whence
\(y_{S,\theta}\downarrow\ell_{L,\theta}\).  Thus every \(S>0\) is a
continuity point of the deterministic limiting distribution in
Theorem~\ref{thm:exceedance_times}(d), and the limiting distribution is proper.  The truncated deterministic constant equals
\[
  \mathfrak C(S)
  =
  \int_\R e^{bx}
  \Ind{J_S(-x)>L}\,dx
  =
  \frac1b e^{-b y_{S,\theta}}.
\]
Dividing by \(\mathcal C_\theta=b^{-1}e^{-b\ell_{L,\theta}}\) and applying
Theorem~\ref{thm:exceedance_times}(d) gives
\eqref{eq:choquet-time-alpha-greater-beta-example}.
\end{proof}

\newpage
\appendix

\section{Proofs of the main claims}
\label{sec:proofs_main_claims}

\subsection{Choquet-integral functionals}
\label{sec:proof_choquet_claims}
\begin{proof}[Verification of the Choquet-integral claims]
  The first implication in~\ref{B1} follows from monotonicity of the capacity:
  for every \(t\ge0\),
  \[
    \{x\in A:\psi(\bm f(x))>t\}
    \subset
    \{x\in A\cup B:\psi(\bm f(x))>t\}.
  \]
  Hence \(\Gamma_A(\bm f)>L\) implies \(\Gamma_{A\cup B}(\bm f)>L\). Since
  \(L\ge0\), the inequality \(\Gamma_A(\bm f)>L\) also implies that
  \(\nu\{x\in A:\psi(\bm f(x))>t\}>0\) for some \(t>0\). Thus this set is
  non-empty, and the assumption \(\psi(\bm x)>0\Rightarrow\min_i x_i>0\) gives
  \(\sup_{x\in A}\min_i f_i(x)>0\).

  Suppose next that \(\Gamma_{A\cup B}(\bm f)>L\) but
  \(\Gamma_A(\bm f)\le L\). Then
  \[
    \int_0^\infty
    \Big(
      \nu\{x\in A\cup B:\psi(\bm f(x))>t\}
      -\nu\{x\in A:\psi(\bm f(x))>t\}
    \Big)\,dt>0.
  \]
  Hence the integrand is strictly positive for some \(t>0\). For this \(t\), the level set in \(A\cup B\) has strictly larger
  capacity than its intersection with \(A\).  It therefore contains a point of
  \(B\setminus A\) (and in particular meets \(B\)).  Thus
  \(\psi(\bm f(x))>t>0\) for some \(x\in B\), and again \(\sup_{x\in B}\min_i f_i(x)>0\). This proves~\ref{B1}.

  The affine scaling property~\ref{B2} is immediate from the covariance of
  \(\nu\):
  \begin{align*}
    \Gamma_{aE+b}(\bm f)
    &=\int_0^\infty
      \nu\{x\in aE+b:\psi(\bm f(x))>t\}\,dt \\
    &=\int_0^\infty
      \nu\bigl(b+a\{x\in E:\psi(\bm f(ax+b))>t\}\bigr)\,dt \\
    &=a^{\lambda_\nu}
      \int_0^\infty
      \nu\{x\in E:\psi(\bm f(ax+b))>t\}\,dt
      =a^{\lambda_\nu}\Gamma_E(\bm f(a\cdot+b)).
  \end{align*}
  Thus Assumption~\ref{B2} holds with
  \(\lambda_\Gamma=\lambda_\nu\).

  The no-atom and continuity assertions are precisely the applications of
  Lemmas~\ref{sojorn_F2_cap} and~\ref{sojorn_F3_cap} to
  \(G=\psi\) and to the shifted limiting process appearing in~\ref{B3} and
  \ref{B4}.  If some coordinates of that path are frozen at finite positive
  values, Corollary~\ref{cor:choquet-partial-shifts} applies in the remaining
  coordinates to the section \(\psi_{D,\bm c}\) from
  \eqref{def:positive-coordinate-section}; this is the role of the section-wise
  condition in Definition~\ref{def:admissible-choquet-data}.

  Finally, assume the displayed level-set identity in the remark. Integrating it
  over \(t>0\) gives
  \[
    \Gamma_E(\bm f)
    =
    \Gamma_E(\bm f_{K^c},\bm\infty_K)
    \Ind{\bm f_K>\bm0_K}.
  \]
  Since admissible levels are non-negative and the event in~\ref{B5} is strict,
  this identity is equivalent to
  \[
    \Gamma_E(\bm f)>L
    \quad\Longleftrightarrow\quad
    \Gamma_E(\bm f_{K^c},\bm\infty_K)>L
    \text{ and }
    \bm f_K>\bm0_K,
  \]
  which proves~\ref{B5}.
\end{proof}

\subsection{Local-footprint transfer}
\label{sec:proof_local_footprint_transfer}
\begin{proof}[Proof of Proposition~\ref{prop:local-footprint-transfer}]
Fix \(u\) and keep the physical footprint \(\rho_uF\) fixed.  On a common
ambient continuous path, define the family indexed only by the start set by
\[
  \widehat\Gamma_E^{(u)}(\bm f)
  =\Gamma_{E;\rho_uF}(\bm f|_{E+\rho_uF}).
\]
Thus, when several start blocks are compared, every restriction is taken from
the common path domain \((\bigcup_k E_k)+\rho_uF\).  Assumption~\ref{LF1} is
exactly the part of Assumption~\ref{B1} used by the
adapted Bonferroni lemma.  Consequently all upper and lower decompositions are
performed on start blocks only.  Moreover, a functional exceedance on a start
block implies a pointwise vector exceedance at a point of that same block.
Therefore the one-block bounds, the separated-block estimates, and the
adjacent-block estimates in the original double-sum proofs are unchanged.
Possible overlap of the enlarged path domains \(B_{k,u}+\rho_uF\) is irrelevant:
the Gaussian two-block remainder is bounded by pointwise exceedances in the
disjoint start blocks.  This is the precise reason no assumption in the
footprint variable is needed.

Let \(a_u=\inf E_u\) and consider a complete start block
\[
  B_{k,u}=a_u+\rho_u(Sk+[0,S]).
\]
By~\ref{LF2},
\begin{align*}
 &\Gamma_{B_{k,u};\rho_uF}
    (\uu(\bm X-u\bm b))>L\rho_u^{\lambda_\Gamma}
 \\
 &\qquad\Longleftrightarrow
 \Gamma_{[0,S];F}
 \left(\uu\bigl(\bm X(a_u+\rho_u(Sk+\cdot))-u\bm b\bigr)\right)>L.
\end{align*}
Thus, after rescaling, the footprint is the fixed compact set \(F\), not a
second growing direction.

For the local conditioning argument, set \(D_{S,F}=[0,S]+F\) and regard
\[
  \bm f\longmapsto\Gamma_{[0,S];F}(\bm f)-L
\]
as a functional on \(C(D_{S,F},\R^d)\).  Assumption~\ref{LF1} gives
condition~\ref{F1}, while Assumption~\ref{LF3} gives conditions~\ref{F2}--\ref{F3}
for the exact coordinate-reduced limiting element.  The generic limiting
constant retains the finite slack coordinates, so no constant-reduction
assumption is needed for the transfer.  Since \(F\) is fixed and
compact, \(D_{S,F}\) is contained in a fixed compact interval for every fixed
\(S\).  The conditional mean, covariance, and increment estimates used in
Lemma~\ref{lemma: limiting process} and in the stationary local argument are
uniform on such a compact set.  Lemma~\ref{vectorPickands_functionals}
therefore yields the same single-block asymptotic, with
\(\Gamma_{[0,S]}\) replaced by \(\Gamma_{[0,S];F}\).

The arguments proving finiteness, positivity, and existence of the
Pickands-type constant partition only the start interval and use only the
single-block limit, the start-set Bonferroni implication, and the same
pointwise two-block Gaussian estimates.  They therefore give
\[
  \mathcal H_{\Gamma;F,L,\alpha,V,\Sigma,\bm b}
  =\lim_{S\to\infty}\frac1S
   H_{\Gamma;F,L,\alpha,V,\bm0,\Sigma,\bm b}([0,S]).
\]
The monotone limits defining the Piterbarg-type and deterministic constants
are unchanged for the same reason.

It remains only to compare the admissible start set
\(E_u=[0,T]\ominus\rho_uF\) with \([0,T]\).  In the stationary case the
erosion removes intervals of total length \(O(\rho_u)\), hence at most a
bounded number of local blocks; their contribution is
\(O(u^{-|I|}\varphi_\Sigma(u\bb))\), negligible relative to the main factor
\(u^{2/\alpha-|I|}\varphi_\Sigma(u\bb)\).  In the boundary non-stationary
case \(F\subset[0,\infty)\), so the left endpoint \(0\), where the generalized
variance is minimized, is retained.  Only an \(O(\rho_u)\) interval at the
right endpoint is removed, and that region is already negligible in the
localization step because it is separated from the unique minimizer.  Hence
the Laplace sums and all three leading constants are unchanged.

Finally, with \(F\) fixed, the scaled-level consequence of
Assumptions~\ref{LF1}--\ref{LF2} makes
\[
  \{t:\Gamma_{[0,t];\rho_uF}(\uu(\bm X-u\bm b))
       >L\rho_u^{\lambda_\Gamma}\}
\]
upward closed.  The sandwich proof of
Theorem~\ref{thm:exceedance_times} therefore applies verbatim to the first
successful start time.  This proves the transfer principle.
\end{proof}

\subsection{Stationary case}
\label{sec:proof_stationary_case}
\begin{proof}[Proof of Theorem~\ref{thm:stationary_case}]
Fix a block length \(S>0\) and put
    \(N_u=\floor{T S^{-1}u^{2/\alpha}}\).  Let
    \(B_{k,u}=S u^{-2/\alpha}[k,k+1]\), \(0\le k\le N_u-1\), and let
    \(R_u=[N_uS u^{-2/\alpha},T]\) be the possible terminal remainder.  This
    remainder has length at most \(S u^{-2/\alpha}\).  Applying
    Lemma~\ref{lemma Bonferroni adapted for sojourns} to the complete blocks,
    and using the scaled-level consequence in
    Remark~\ref{rem:scaled-level-B1} with
    \(A=\cup_{k<N_u}B_{k,u}\), \(B=R_u\), gives
    \begin{equation*}
      \mathbb{\Sigma}_1 - \mathbb{\Sigma}_2
      \leq
      \pk{\Gamma_{[0, T]} ( \uu ( \bm{X} - u \bm{b} ) ) > L_u}
      \leq
      \mathbb{\Sigma}_1 + \mathbb{\Sigma}_2
      + \pk{\exists t\in R_u:\bm X(t)>u\bm b},
    \end{equation*}
    where
    \begin{gather*}
        \mathbb{\Sigma}_1
        =
        \sum_{k = 0}^{N_u-1}
        \pk{
        \Gamma_{B_{k,u}}
        ( \uu ( \bm{X} - u \bm{b} ) )
        > L_u
        },
        \\[7pt]
        \mathbb{\Sigma}_2
        =
        \sum_{0\le k \neq l \le N_u-1}
        \pk{
        \begin{aligned}
        & \exists\, t \in B_{k,u}:
        \bm{X} ( t ) > u \bm{b}
        \\
        & \exists\, s \in B_{l,u}:
        \bm{X} ( s ) > u \bm{b}.
        \end{aligned}
        }
    \end{gather*}
    The terminal-remainder probability is bounded by one local high-excursion
    probability on an interval of length \(S u^{-2/\alpha}\), hence it is
    \(O(u^{-|I|}\varphi_\Sigma(u\bb))=o(u^{2/\alpha-|I|}\varphi_\Sigma(u\bb))\).
    We next control the double sum. Split the pairs according to whether
    \(S|k-l|<\delta u^{2/\alpha}\), where \(\delta\) is the constant from
    Lemma~\ref{lemma:double_sum_bound_stationary}. The complementary pairs are
    separated by a fixed positive time distance; by~\ref{R1} and the standard
    Gaussian quadratic-programming bound for two separated copies, their total
    contribution is exponentially smaller than
    \(u^{2/\alpha-|I|}\varphi_\Sigma(u\bb)\).  For the near pairs, write
    \(h=|k-l|\). Lemma~\ref{lemma:double_sum_bound_stationary} gives the
    adjacent contribution \(h=1\) of order
    \[
      N_u\,(S^{1/2}+S e^{-cS^{\alpha/2}})
      u^{-|I|}\varphi_\Sigma(u\bb),
    \]
    and the separated contribution \(h\ge2\) of order
    \[
      N_u\,S\sum_{m\ge1}e^{-c(Sm)^\alpha}
      u^{-|I|}\varphi_\Sigma(u\bb).
    \]
    Since \(N_u\sim T u^{2/\alpha}/S\), both bounds are
    \(o(u^{2/\alpha-|I|}\varphi_\Sigma(u\bb))\) after first letting
    \(u\to\infty\) and then \(S\to\infty\). Hence
    \begin{equation*}
      \lim_{S \to \infty} \limsup_{u \to \infty}
      \frac{\mathbb{\Sigma}_2}{u^{2 / \alpha - |I|} \, \varphi_\Sigma ( u \bb )} 
      = 0.
    \end{equation*}
    For each term of the single sum \( \mathbb{\Sigma}_1 \) we have 
    \begin{align*}
      \pk{\Gamma_{S u^{-2 / \alpha}[k, k+1]} ( \uu (\bm{X} - u \bm{b}) ) > L_u}
      & = \pk{
        u^{-2 \lambda_\Gamma / \alpha} \Gamma_{[0,S]} ( \uu (\bm{X}_{u,k} - u \bm{b})) 
        > L \cdot u^{-2 \lambda_\Gamma / \alpha} 
        }
      \\
      & = \pk{\Gamma_{[0,S]} ( \uu (\bm{X}_{u, 0} - u \bm{b})) > L},
    \end{align*}
    where
    \( \bm{X}_{u,k}(t)=\bm{X}(u^{-2/\alpha}(Sk+t))\), \(t\in[0,S]\).  Here
    \ref{B2} is used with the affine map \(t\mapsto u^{-2/\alpha}(Sk+t)\), and
    stationarity gives \( \bm{X}_{u,k}\overset{d}{=}\bm{X}_{u,0}\).  The local
    covariance expansion, together with the corresponding increment bound, gives
    assumptions~\ref{A1}--\ref{A3} with limiting field
    \(\bm Y_{\alpha,V_{I,I}}\) and local drift
    \(S_{\alpha,V_{I,I}}(\cdot)\bm w_I\). Assumption~\ref{B1} supplies
    \ref{F1} for \(\tilde{\Gamma}=\Gamma_{[0,S]}(\cdot)-L\), and the
    Assumptions~\ref{B3}--\ref{B4} for the limiting model are exactly
    conditions~\ref{F2}--\ref{F3} for this shifted limiting element. Applying Lemma~\ref{vectorPickands_functionals} to
    \(\bm{X}_{u,0}\) and \(\tilde{\Gamma}\), we obtain
    \begin{equation*}
      \mathbb{\Sigma}_1
      \sim 
      T \, u^{2 / \alpha} \, \frac{H ( S )}{S} \, 
      u^{-|I|} \,
      \varphi_{\Sigma} ( u \bb )
\end{equation*}
    where \( H(S) \coloneqq H_{\Gamma, L, \alpha, V_{I, I}, \bm{0}_I, \Sigma,
    \bm{b}} ( [0,S] ) \). By
    Lemma~\ref{lemma: lim H(S) / S in (0, infty) if alpha < beta}, the limit
    \( \mathcal{H} = \lim_{S \to \infty} S^{-1} H(S) \) exists and is finite.
    Combining the upper and lower Bonferroni bounds and then letting
    \(S\to\infty\) gives the claimed stationary asymptotic. The constant is
    positive by the non-triviality assumption.
\end{proof}

\subsection{Non-stationary case}
\label{sec:proof_main_theorem}
\begin{proof}[Proof of Theorem~\ref{thm:main_theorem}]
We split the proof into the three cases \( \alpha < \beta \), \( \alpha = \beta \), and \( \alpha > \beta \).
  \bigskip

  \textbf{Case \( \alpha < \beta \).}
  Let \( \Lambda > 0 \). By the scaled-level consequence in
  Remark~\ref{rem:scaled-level-B1}, we have that
  \begin{align*}
    \pk{
      \Gamma_{[0, \Lambda u^{-2 / \beta} ]} 
      ( \uu ( \bm{X} - u \bm{b} ) )
      > L_u
      }
    & 
    \leq
    \pk{
      \Gamma_{[0,T]} 
      ( \uu ( \bm{X} - u \bm{b} ) ) 
      > L_u
      }
    \\[7pt]
    &
    \leq
    \pk{
      \Gamma_{[0, \Lambda u^{-2 / \beta}]} 
      ( \uu ( \bm{X} - u \bm{b} ) )
      > L_u
    }
    + \pk{ 
    \exists\, t \in [\Lambda u^{-2 / \beta} , T] : 
    \bm{X} ( t ) > u \bm{b}
    }.
  \end{align*}
  Using \Cref{log layer bound lemma} the second term on the right may be bounded
  as
  \begin{equation}
    \label{main theorem: log layer bound application}
    \pk{\exists\, t \in [\Lambda u^{-2 / \beta}, T] : \bm{X} (t) > u \bm{b}}
    \lesssim 
    \exp\left( -c \Lambda^\beta \right) 
    u^{2/\alpha-2/\beta-|I|} \, \varphi_\Sigma ( u \bb )
  \end{equation}
  with some \( c > 0 \) for all \( u \) large enough. Let \(S>0\) and put
  \(M_u=\floor{\Lambda S^{-1}u^{2/\alpha-2/\beta}}\).  Set
  \(B_{k,u}=S u^{-2/\alpha}[k,k+1]\), \(0\le k\le M_u-1\), and let
  \(R_{u,\Lambda}=[M_uS u^{-2/\alpha},\Lambda u^{-2/\beta}]\) be the possible
  terminal remainder.  Lemma~\ref{lemma Bonferroni adapted for sojourns}, applied
  to the complete blocks, and the scaled-level implication from
  Remark~\ref{rem:scaled-level-B1} give
  \begin{equation}
    \label{main theorem: sojourn-bonferroni application}
    \mathbb{\Sigma}_1 - \mathbb{\Sigma}_2
    \leq
    \pk{
      \Gamma_{[0, \Lambda u^{-2 / \beta}]}
      ( \uu ( \bm{X} - u \bm{b} ) )
      > L_u
    }
    \leq
    \mathbb{\Sigma}_1 + \mathbb{\Sigma}_2
    + \pk{\exists t\in R_{u,\Lambda}:\bm X(t)>u\bm b},
  \end{equation}
  where
  \begin{gather*}
    \mathbb{\Sigma}_1
    =
    \sum_{k=0}^{M_u-1}
    \pk{
      \Gamma_{B_{k,u}}
      ( \uu ( \bm{X} - u \bm{b} ) )
      > L_u
    },
    \\[7pt]
    \mathbb{\Sigma}_2
    =
    \sum_{0\le k \neq l \le M_u-1}
    \pk{
    \begin{aligned}
      & \exists\, t \in B_{k,u}:
      \bm{X} ( t ) > u \bm{b}
      \\
      & \exists\, s \in B_{l,u}:
      \bm{X} ( s ) > u \bm{b}.
    \end{aligned}
    }
  \end{gather*}
  The terminal remainder has length at most \(S u^{-2/\alpha}\).  By the same
  local bound used for one complete block, its high-excursion probability is
  \(O(u^{-|I|}\varphi_\Sigma(u\bb))\), which is negligible on the scale
  \(u^{2/\alpha-2/\beta-|I|}\varphi_\Sigma(u\bb)\) because \(\alpha<\beta\).
  Lemma~\ref{lemma:double_sum_bound} gives the required double-sum estimate.
  Indeed, adjacent pairs contribute at most
  \(M_u(S^{1/2}+S e^{-cS^{\alpha/2}})u^{-|I|}\varphi_\Sigma(u\bb)\), and pairs
  separated by at least one full block contribute at most
  \(M_uS\sum_{m\ge1}e^{-c(Sm)^\alpha}u^{-|I|}\varphi_\Sigma(u\bb)\). Since
  \(M_u\sim \Lambda u^{2/\alpha-2/\beta}/S\),
  \begin{equation}
    \label{main theorem: double sum bounded by S^{-1/2}}
    \limsup_{u \to \infty} 
    \frac{
      \mathbb{\Sigma}_2
    }{ 
      u^{2 / \alpha - 2 / \beta - |I|} \,
      \varphi_{\Sigma} ( u \bb )
}
    \lesssim_\Lambda
    S^{-1 / 2}.
  \end{equation}
  By Assumption~\ref{B2}, we have
  \begin{align*}
    \pk{\Gamma_{S u^{-2 / \alpha} [k, k+1]} ( \uu (\bm{X} - u \bm{b} ) ) > L_u}
    & 
    = \pk{
      u^{-2 \lambda_\Gamma / \alpha}
      \Gamma_{[0, S]} 
      ( \uu ( \bm{X}_{u, k} - u \bm{b} ) )
      > 
      u^{-2 \lambda_\Gamma / \alpha} L
    }
    \\[7pt]
    & = \pk{
      \Gamma_{[0, S]} 
      ( \uu ( \bm{X}_{u, k} - u \bm{b} ) )
      > 
      L
    }
  \end{align*}
  where we set
  \[
    \tau_k=Sk,
    \qquad
    \bm X_{u,k}(t)=\bm X(u^{-2/\alpha}(\tau_k+t)),
    \quad t\in[0,S],
  \]
  and used that \(L_u=L\cdot u^{-2\lambda_\Gamma/\alpha}\). For
  \(k\le \Lambda S^{-1}u^{2/\alpha-2/\beta}\), the parameters
  \(\tau_k\) lie in \(Q_u=[0,\Lambda u^{2/\alpha-2/\beta}]\). Thus
  Lemma~\ref{lemma: limiting process} verifies assumptions
  \ref{A1}--\ref{A3} uniformly in \(k\). The local drift entering
  Lemma~\ref{vectorPickands_functionals} is
  \(S_{\alpha,V_{I,I}}(\cdot)\bm w_I\), hence the resulting preconstant is
  \(H_{\Gamma,L,\alpha,V_{I,I},\bm0_I,\Sigma,\bm b}([0,S])\). Let
  \(\tilde{\Gamma}(\cdot)=\Gamma_{[0,S]}(\cdot)-L\). Assumption~\ref{B1}
  gives~\ref{F1}, while Assumptions~\ref{B3}--\ref{B4} for the limiting model
  give~\ref{F2}--\ref{F3}. Hence we can
  apply Lemma~\ref{vectorPickands_functionals} uniformly in \(k\) to obtain
  \begin{equation*}
    \label{main theorem: Pickands lemma application}
    \pk{\Gamma_{[0, S]} ( \uu ( \bm{X}_{u, k} - u \bm{b} ) ) > L }
    \sim
    H ( S ) \,
    u^{-|I|} \, 
    \varphi_{\Sigma_{u,\tau_k}} ( u \bb ),
  \end{equation*}
  where 
  \( 
  H(S) 
  \coloneqq 
  H_{\Gamma, L, \alpha, V_{I, I}, \bm{0}_I, \Sigma, \bm{b}} ([0, S])
  \).
  Moreover, uniformly over the same range of \(k\),
  \begin{equation*}
    \frac{\varphi_{\Sigma_{u,\tau_k}}(u\bb)}{\varphi_\Sigma(u\bb)}
    =
    \exp\left\{-\tau_{\bm w}
      (S k u^{2/\beta-2/\alpha})^\beta(1+o(1))\right\}.
  \end{equation*}
  This follows by expanding \(\bb^\top\Sigma(t)^{-1}\bb\) at
  \(t=0\), with \(t=u^{-2/\alpha}Sk\), and using
  \(\Sigma-\Sigma(t)=(A+A^\top)t^\beta+o(t^\beta)\). Therefore,
  \begin{align*}
    \mathbb{\Sigma}_1
    & \sim 
    H ( S ) \, 
    u^{-|I|} \,
    \varphi_\Sigma ( u \bb )
    \sum_{k=0}^{M_u-1}
    e^{ -\tau_{ \bm{w} }( S k u^{ 2/\beta - 2/\alpha } )^\beta ( 1 + o( 1 ) ) } 
    \\[7pt]
    & \sim
    \frac{H ( S )}{S} \,
    u^{2 / \alpha - 2 / \beta} \,
    u^{-|I|} \,
    \varphi_\Sigma ( u \bb )
    \int_0^\Lambda e^{ -\tau_{ \bm{w} } x^{ \beta } } \, dx.
  \end{align*}
  Letting \( u \to \infty \) and then \( S \to \infty \), and using that 
  \( 
  \mathcal{H} = \lim_{S \to \infty} S^{-1} H(S) \in (0,\infty) 
  \) 
  by Lemma~\ref{lemma: lim H(S) / S in (0, infty) if alpha < beta}, we obtain
  \begin{align*}
    \int_0^\Lambda e^{-\tau_{\bm{w}} x^\beta} \, dx
    & \leq
\liminf_{u \to \infty} 
    \frac{
    \pk{\Gamma_{[0,T]} ( \uu (\bm{X} - u \bm{b}) ) > L_u}
    }{
      \mathcal{H} \, u^{2 / \alpha - 2 / \beta - |I|} \,
      \varphi_{\Sigma} ( u \bb )
}
    \\[7pt]
    & \leq
    \limsup_{u \to \infty} 
    \frac{
    \pk{\Gamma_{[0,T]} ( \uu (\bm{X} - u \bm{b}) ) > L_u}
    }{
      \mathcal{H} \, u^{2 / \alpha - 2 / \beta - |I|} \,
      \varphi_{\Sigma} ( u \bb )
}
    \leq
    \int_0^\Lambda e^{-\tau_{\bm{w}} x^\beta} \, dx
    + C e^{-c \Lambda^\beta}.
  \end{align*}
  Letting \( \Lambda \to \infty \), we further obtain
  \begin{equation*}
\lim_{u \to \infty}
    \frac{
      \pk{\Gamma_{[0,T]} ( \uu ( \bm{X} - u \bm{b} ) ) > L_u }
    }{
      \mathcal{H} \, 
      u^{2 / \alpha - 2 / \beta - |I|} \, 
      \varphi_{\Sigma} ( u \bb )
}
    = 
    \Gamma \left( \frac{1}{\beta} + 1 \right) \, 
    \tau_{\bm{w}}^{-1 / \beta}.
  \end{equation*}

  \textbf{Case \( \alpha = \beta \).}
  For fixed \(\Lambda>0\), union-monotonicity and the log-layer bound give the
  deterministic truncation estimate
  \begin{equation*}
    0\le
    \pk{\Gamma_{[0,T]}(\uu(\bm X-u\bm b))>L_u}
    -
    \pk{\Gamma_{[0,\Lambda u^{-2/\beta}]}(\uu(\bm X-u\bm b))>L_u}
    \lesssim e^{-c\Lambda^\beta}u^{-|I|}\varphi_\Sigma(u\bb).
  \end{equation*}
  After the affine rescaling of~\ref{B2}, the local probability on
  \([0,\Lambda u^{-2/\beta}]\) is written on the fixed interval
  \([0,\Lambda]\).  Lemma~\ref{lemma: limiting process} and
  Lemma~\ref{vectorPickands_functionals}, applied with the total limiting drift
  \(S_{\alpha,V_{I,I}}(\cdot)\bm{w}_I+\bm{d}_I\), then give
  \begin{equation*}
    \pk{
      \Gamma_{[0, \Lambda u^{-2 / \beta}]}
      ( \uu ( \bm{X} - u \bm{b} ) )
      > L_u
    }
    \sim
    H_{\Gamma, L, \alpha, V_{I,I}, \bm{d}_I, \Sigma, \bm{b}}([0,\Lambda])
    u^{-|I|}\varphi_\Sigma(u\bb).
  \end{equation*}
  Therefore
  \begin{align*}
    H_{\Gamma, L, \alpha, V_{I,I}, \bm{d}_I, \Sigma, \bm{b}}([0,\Lambda])
    &\le
    \liminf_{u\to\infty}
    \frac{\pk{\Gamma_{[0,T]}(\uu(\bm X-u\bm b))>L_u}}
    {u^{-|I|}\varphi_\Sigma(u\bb)}
    \\
    &\le
    \limsup_{u\to\infty}
    \frac{\pk{\Gamma_{[0,T]}(\uu(\bm X-u\bm b))>L_u}}
    {u^{-|I|}\varphi_\Sigma(u\bb)}
    \\
    &\le
    H_{\Gamma, L, \alpha, V_{I,I}, \bm{d}_I, \Sigma, \bm{b}}([0,\Lambda])
    +Ce^{-c\Lambda^\beta}.
  \end{align*}
  Lemma~\ref{lemma: lim H(S) in (0, infty) if alpha > beta} shows that the
  truncated constants increase to the finite limit
  \(\mathcal{P}_{\Gamma, L, \alpha, V_{I,I}, \bm{d}_I, \Sigma, \bm{b}}\).
  Letting \(\Lambda\to\infty\) proves part~(b).

  \textbf{Case \( \alpha > \beta \).}
  The same truncation estimate holds:
  \begin{equation*}
    0\le
    \pk{\Gamma_{[0,T]}(\uu(\bm X-u\bm b))>L_u}
    -
    \pk{\Gamma_{[0,\Lambda u^{-2/\beta}]}(\uu(\bm X-u\bm b))>L_u}
    \lesssim e^{-c\Lambda^\beta}u^{-|I|}\varphi_\Sigma(u\bb).
  \end{equation*}
  After the affine rescaling of~\ref{B2} to the fixed interval
  \([0,\Lambda]\), Lemma~\ref{lemma: limiting process} gives, on the scale
  \(u^{-2/\beta}\), a degenerate Gaussian part and the deterministic drift
  \(\bm{d}(t)=t^\beta A\bm{w}\).  Hence Lemma~\ref{vectorPickands_functionals}
  and Assumptions~\ref{B3}--\ref{B4} for the deterministic limiting model yield
  \begin{equation*}
    \pk{
      \Gamma_{[0, \Lambda u^{-2 / \beta}]}
      ( \uu ( \bm{X} - u \bm{b} ) )
      > L_u
    }
    \sim
    C(\Lambda)\,u^{-|I|}\varphi_\Sigma(u\bb),
  \end{equation*}
  where
  \begin{equation*}
    C(\Lambda)=
      \int_{\R^d}
      \exp \left(
        \bm{x}_I^\top \bm{w}_I
        - \frac{1}{2} \, \bm{x}_J^\top ( \Sigma^{-1} )_{J J} \, \bm{x}_J
      \right)
      \Ind{
        \Gamma_{[0, \Lambda]} \left(
          -\bm{d}_I - \bm{x}_I,
          -\bm{x}_{\Jtan},
          ( \bb - \bm{b} )_{\Jsl}
        \right)
        > L
      }
      \, d \bm{x}.
  \end{equation*}
  Consequently,
  \begin{align*}
    C(\Lambda)
    &\le
    \liminf_{u\to\infty}
    \frac{\pk{\Gamma_{[0,T]}(\uu(\bm X-u\bm b))>L_u}}
    {u^{-|I|}\varphi_\Sigma(u\bb)}
    \\
    &\le
    \limsup_{u\to\infty}
    \frac{\pk{\Gamma_{[0,T]}(\uu(\bm X-u\bm b))>L_u}}
    {u^{-|I|}\varphi_\Sigma(u\bb)}
    \\
    &\le C(\Lambda)+Ce^{-c\Lambda^\beta}.
  \end{align*}
  Lemma~\ref{lemma: lim H(S) in (0, infty) if alpha > beta} shows that
  \(C(\Lambda)\) increases to the finite limit
  \(\mathcal{C}_{\Gamma,L,\bm d,\Sigma,\bm{b}}\). Letting \(\Lambda\to\infty\)
  proves part~(c).
\end{proof}

\subsection{Exceedance times}
\label{sec:proof_exceedance_times}
\begin{proof}[Proof of Theorem~\ref{thm:exceedance_times}]
  We first record a deterministic consequence of the scaled-level
  union-monotonicity supplied by Assumptions~\ref{B1}--\ref{B2} and
  Remark~\ref{rem:scaled-level-B1}. On \((0,T]\), the set
  \[
    A_u=\{t:\Gamma_{[0,t]}(\uu(\bm X-u\bm b))>L_u\}
  \]
  is upward closed. Hence, for every \(0<s_-<s<s_+\le T\),
  \begin{equation}
    \label{eq:hitting_sandwich}
    \{\Gamma_{[0,s_-]}(\uu(\bm X-u\bm b))>L_u\}
    \subset
    \{\mathfrak t(u,L)\le s\}
    \subset
    \{\Gamma_{[0,s_+]}(\uu(\bm X-u\bm b))>L_u\}.
  \end{equation}
  Indeed, if \(\mathfrak t(u,L)\le s<s_+\), then some point of \(A_u\) lies below
  \(s_+\), and upward closedness gives \(s_+\in A_u\). The same argument at
  the endpoint gives
  \(
    \{\mathfrak t(u,L)\le T\}
    =\{\Gamma_{[0,T]}(\uu(\bm X-u\bm b))>L_u\}
  \).

  In the stationary case, applying Theorem~\ref{thm:stationary_case} to the
  left and right sides of \eqref{eq:hitting_sandwich}, and then letting
  \(s_-\uparrow s\) and \(s_+\downarrow s\), gives
  \[
    \frac{
      \pk{\mathfrak t(u,L)\le s}
    }{
      \pk{\mathfrak t(u,L)\le T}
    }
    \to \frac{s}{T},
  \]
  for every \(s\in(0,T)\). The endpoints \(s=0\) and \(s=T\) follow from the
  same sandwich and from the denominator identity.

  In the non-stationary case \(\alpha<\beta\), the rescaled version of
  \eqref{eq:hitting_sandwich} gives, for every \(0<S_-<S<S_+\),
  lower and upper bounds with the intervals \([0,S_-u^{-2/\beta}]\) and
  \([0,S_+u^{-2/\beta}]\). Repeating the single-sum part of the proof of
  Theorem~\ref{thm:main_theorem}, with the outer horizon \(\Lambda\) replaced by
  \(S_\pm\) and with an independent block length \(R\to\infty\) used in the
  Bonferroni decomposition, gives
  \begin{equation*}
    \pk{\Gamma_{[0,S_\pm u^{-2/\beta}]}(\uu(\bm X-u\bm b))>L_u}
    \sim
    \mathcal{H}_{\Gamma,L,\alpha,V_{I,I},\Sigma,\bm b}
    \left(\int_0^{S_\pm} e^{-\tau_{\bm{w}}x^\beta}\,dx\right)
    u^{2/\alpha-2/\beta-|I|}\varphi_\Sigma(u\bb).
  \end{equation*}
  Theorem~\ref{thm:main_theorem}(a) gives the denominator with
  \(\int_0^\infty e^{-\tau_{\bm{w}}x^\beta}\,dx
  =\Gamma(1/\beta+1)\tau_{\bm{w}}^{-1/\beta}\). Letting
  \(S_-\uparrow S\) and \(S_+\downarrow S\) yields the stated incomplete-gamma
  expression. The case \(S=0\) follows by taking \(S_+\downarrow0\).

  In the case \(\alpha=\beta\), the same sandwich gives bounds with
  \([0,S_-u^{-2/\beta}]\) and \([0,S_+u^{-2/\beta}]
  \). After rescaling to fixed intervals, Lemma~\ref{lemma: limiting process}
  and Lemma~\ref{vectorPickands_functionals} yield
  \begin{equation*}
    \pk{\Gamma_{[0,S_\pm u^{-2/\beta}]}(\uu(\bm X-u\bm b))>L_u}
    \sim
    H_{\Gamma,L,\alpha,V_{I,I},\bm d_I,\Sigma,\bm b}([0,S_\pm])
    u^{-|I|}\varphi_\Sigma(u\bb).
  \end{equation*}
  Dividing by the asymptotic in Theorem~\ref{thm:main_theorem}(b) gives
  lower and upper limits \(G_{\mathfrak P}(S-)\) and
  \(G_{\mathfrak P}(S+)\).  Since \(F_{\mathfrak P}\) is the
  right-continuous modification of this non-decreasing function, both limits
  equal \(F_{\mathfrak P}(S)\) at every positive continuity point of
  \(F_{\mathfrak P}\).  If \(S=0\) is a continuity point of the distribution
  function, then \(F_{\mathfrak P}(0)=0\); the upper bound obtained from
  \([0,S_+u^{-2/\beta}]\), followed by \(S_+\downarrow0\), proves the same
  conclusion at zero.

  Finally, if \(\alpha>\beta\), the rescaled local Gaussian component
  degenerates and the same sandwich gives
  \begin{equation*}
    \pk{\Gamma_{[0,S_\pm u^{-2/\beta}]}(\uu(\bm X-u\bm b))>L_u}
    \sim
    \mathfrak C(S_\pm)u^{-|I|}\varphi_\Sigma(u\bb).
  \end{equation*}
  Dividing by Theorem~\ref{thm:main_theorem}(c) gives lower and upper limits
  \(G_{\mathfrak C}(S-)\) and \(G_{\mathfrak C}(S+)\), which both equal
  \(F_{\mathfrak C}(S)\) at every positive continuity point of the
  right-continuous modification.  At \(S=0\), continuity of the distribution
  function implies \(F_{\mathfrak C}(0)=0\), and the upper bound with
  \(S_+\downarrow0\) completes the argument.  The interval statements in
  parts~(c) and~(d) are differences of the corresponding distribution
  functions at continuity points.
\end{proof}

\newpage
\section{Auxiliary results}
\label{sec:auxiliary_results}

Let \( 0 \in E \in \mathcal{K} \) and let \( \bm{X}_{u, \tau}(t) \), \( t \in E
\), \(u>0\), \( \tau \in Q_u \) be a family of continuous centered \( \R^d
\)-valued Gaussian fields.
Denote its cmf by $R_{u, \tau}(t, s)=\E{\bm{X}_{u, \tau}(t) \bm{X}_{u,
\tau}(s)^{\top}}$. \smallskip  

We shall impose the following assumptions:
\begin{enumerate}[itemsep=0.3cm]
  \item[(\namedlabel{A1}{\( \mathbf{A}_1 \)})] There exists a positive
  definite matrix \(\Sigma\) such that, with
  \[
    \bb=
    \argmin\{\bm x^\top\Sigma^{-1}\bm x:\bm x\ge\bm b\},
  \]
  the matrix \(\Sigma_{u,\tau}=R_{u,\tau}(0,0)\) is positive definite for all
  sufficiently large \(u\) and all \(\tau\in Q_u\), and the following
  conditions hold:
  \begin{enumerate}[itemsep=0.3cm, topsep=0.3cm]
    \item [(\namedlabel{A1.1}{\( \mathbf{A}_{1.1} \)})] \( 
    \lim_{u \to \infty} 
    \sup_{\tau \in Q_u}
    \| \Sigma_{u, \tau} - \Sigma \|_{\mathrm{F}}
    = 0
    \)
    \item [(\namedlabel{A1.2}{\( \mathbf{A}_{1.2} \)})] \( 
    \lim_{u \to \infty} 
    \sup_{\tau \in Q_u}
    u \cdot
    \left|
    \Big(
      \Sigma^{-1}_{u, \tau} \bb 
\Big)_J
    \right|
    = 0.
    \)
  \end{enumerate}
  \item[(\namedlabel{A2}{\( \mathbf{A}_2 \)})] There exist a continuous $\mathbb{R}^d$-valued function $\bm{d}(t), t \in E$
  and a continuous matrix-valued function $K(t, s),(t, s) \in E \times E$ such
  that
  \begin{enumerate}[itemsep=0.3cm, topsep=0.3cm]
    \item[(\namedlabel{A2.1}{\( \mathbf{A}_{2.1} \)})] \(
      \displaystyle
      \lim _{u \rightarrow \infty}
      \sup _{\tau \in Q_u, t \in E}
      u \cdot
      \left\|
      \Sigma_{u, \tau}
      -R_{u, \tau}(t, 0)
      \right\|_{\mathrm{F}}
      = 0,
    \)
\item[(\namedlabel{A2.2}{\( \mathbf{A}_{2.2} \)})] \(
      \displaystyle
      \lim _{u \rightarrow \infty}
      \sup _{\tau \in Q_u, t \in E}
      \left| 
      u^2 \left[ 
      \Sigma_{u, \tau} - R_{u, \tau}(t, 0) 
      \right]
      \Sigma_{u,\tau}^{-1} \bb
-\bm{d}(t)
      \right|
      = 0,
    \)
\item[(\namedlabel{A2.3}{\( \mathbf{A}_{2.3} \)})] \(
      \displaystyle
      \lim_{u \rightarrow \infty}
      \sup_{\tau \in Q_u, t, s \in E}
      \left\|
      u^2 \left[
      R_{u, \tau}(t, s)
      -R_{u, \tau}(t, 0) 
      \Sigma_{u, \tau}^{-1} 
      R_{u, \tau}(0, s)
      \right]
      -K(t, s)
      \right\|_{\mathrm{F}}
      = 0.
    \)
\end{enumerate}
\item[(\namedlabel{A3}{\( \mathbf{A}_3 \)})] There exist positive constants $C$ and $\gamma \in(0,2]$ such that for any $s,
  t \in E$
  \begin{equation*}
    \sup_{\tau \in Q_u} u^2 \,
    \E{
      \left|
      \bm{X}_{u, \tau}(t)
      -\bm{X}_{u, \tau}(s)
      \right|^2
    }
    \leq C|t-s|^\gamma.
  \end{equation*}
\end{enumerate}

All quadratic-programming notation in this section, including
\(I,J,\Jtan,\Jsl,\bb\), and \(\bm w\), is associated with the pair
\((\Sigma,\bm b)\) appearing in Assumption~\ref{A1}.

\begin{remark}
  Assumption~\ref{A1} is implied by the following stronger assumption:
  \begin{enumerate}[]
    \item [(\namedlabel{A1'}{\( \mathbf{A}_1' \)})] For all large \( u \) and all \( \tau \in Q_u \) the matrix \( \Sigma_{u,
    \tau} = R_{u, \tau} ( 0, 0 ) \) is positive definite and
    \begin{equation*}
      \lim_{u \to \infty} 
      \sup_{\tau \in Q_u}
      u \cdot 
      \|
        \Sigma_{u, \tau} - \Sigma
      \|
      = 0
    \end{equation*}
    holds for some positive definite matrix \( \Sigma \).
  \end{enumerate}
  This version of Assumption~\ref{A1} is used in~\cite[Lemma 4.7]{VVGP}.
\end{remark} 

Equip \(C(E,\R^d)\) with the uniform norm, and let
\( \Gamma : C ( E, \R^d ) \to \R \) be Borel measurable, but not necessarily
linear or continuous. We shall use the following conditions.  Conditions~\ref{F2} and
\ref{F3} are always read for the particular shifted limiting random element
\(\mathcal W_{\bm x}\) that appears in the lemma or corollary where they are
invoked.
\begin{enumerate}[itemsep=0.3cm, topsep=0.3cm]
  \item [(\namedlabel{F1}{\( \mathbf{F}_1 \)})] \( \Gamma ( \bm{f} ) > 0 \implies \sup_{t \in E} \min_{i = 1, \dots, d} f_i (
          t ) > 0 \).
  \item [(\namedlabel{F2}{\( \mathbf{F}_2 \)})] \( \pk{ \Gamma ( \mathcal W_{\bm x} ) = 0 } = 0 \) for almost all \(
        \bm{x} \in \R^d \).
  \item [(\namedlabel{F3}{\( \mathbf{F}_3 \)})] \( \pk{ \mathcal W_{\bm x} \text{ is a continuity point of } \Gamma} = 1
            \) for almost all \( \bm{x} \in \R^d \).
\end{enumerate} 
In the basic integral examples one has \(\mathcal W_{\bm x}=\bm W-\bm d-\bm x\).

\begin{lemma}
    \label{vectorPickands_functionals}
    Suppose that \( \bm{X}_{u, \tau}(t) \), \( t \in E \), \( u > 0 \), \( \tau \in Q_u \) satisfy~\ref{A1}--\,\ref{A3} and that \( \Gamma \) satisfies~\ref{F1}.
    Let \(I\dot\cup\Jtan\dot\cup\Jsl=\{1,\ldots,d\}\) be the
    quadratic-programming decomposition associated with \((\Sigma,\bm b)\),
    and use the normalization in \eqref{def hat u}, namely
    \[
      (\uu)_I=u\bm1_I,
      \qquad
      (\uu)_{\Jtan}=\bm1_{\Jtan},
      \qquad
      (\uu)_{\Jsl}=u^{-1}\bm1_{\Jsl}.
    \]
    Let \( \bm{W}(t) \), \(t \in E\), be a centered Gaussian random field with covariance
    matrix function \( K (t, s) \). For \(\bm x\in\R^d\), put
    \[
      \mathcal W_{\bm x}
      =
      \begin{pmatrix}
        \bm W_I-\bm d_I-\bm x_I\\
        -\bm x_{\Jtan}\\
        (\bb-\bm b)_{\Jsl}
      \end{pmatrix}.
    \]
    Assume moreover that \(\Gamma\) satisfies conditions~\ref{F2}--\ref{F3} for
    this shifted limiting random element \(\mathcal W_{\bm x}\). Then, as
    $u \to \infty$, we have
    \begin{equation*}
      \pk{\Gamma ( \uu ( \bm{X}_{u, \tau} - u \bm{b} ) ) > 0}
      \sim
      H_{\Gamma, \bm{W}, \bm{d}, \Sigma, \bm{b}} \,
      u^{-|I|} \, 
      \varphi_{\Sigma_{u, \tau}} ( u \bb ),
    \end{equation*}
    where
    \begin{equation*}
      H_{\Gamma, \bm{W}, \bm{d}, \Sigma, \bm{b}}
      \coloneqq
      \int_{\R^d}
      \exp \left(
      \bm{x}_I^\top \bm{w}_I 
      - \frac{1}{2} \, \bm{x}_J^\top ( \Sigma^{-1} )_{J J} \, \bm{x}_J
      \right)
      \pk{
        \Gamma 
        \begin{pmatrix}
          \bm{W}_I - \bm{d}_I - \bm{x}_I \\
          -\bm{x}_{\Jtan}                                  \\
          ( \bb - \bm{b} )_{\Jsl}
        \end{pmatrix}
        > 0
      }
      \, d \bm{x}
    \end{equation*}
    uniformly in $\tau \in Q_u$.
  \end{lemma}

\begin{proof}
    Put \( \overline{\bm u}=(u\bm 1_I,\bm 1_J) \). Conditioning on
    \(\bm X_{u,\tau}(0)\), and making the change of variables
    \(\bm y=u\bb-\bm x/\overline{\bm u}\), gives
    \begin{align*}
      &\pk{\Gamma(\uu(\bm X_{u,\tau}-u\bm b))>0} \\
      &\quad =u^{-|I|}\int_{\R^d}
      a_{u,\tau}(\bm x)
      \varphi_{\Sigma_{u,\tau}}(u\bb-\bm x/\overline{\bm u})\,d\bm x,
    \end{align*}
    where
    \begin{equation*}
      a_{u,\tau}(\bm x)
      =\pk{\Gamma(\bm\chi_{u,\tau,\bm x}+u\uu(\bb-\bm b))>0}
    \end{equation*}
    and
    \begin{equation*}
      \bm\chi_{u,\tau,\bm x}(t)
      =\uu\Big(\bm X_{u,\tau}(t)-u\bb\ \Big|\
      \bm X_{u,\tau}(0)=u\bb-\bm x/\overline{\bm u}\Big).
    \end{equation*}
    Its conditional mean is
    \begin{equation*}
      \bm m_{u,\tau,\bm x}(t)
      =\diag(\uu)\left[
        R_{u,\tau}(t,0)\Sigma_{u,\tau}^{-1}
        (u\bb-\bm x/\overline{\bm u})-u\bb
      \right].
    \end{equation*}
    By Assumptions~\ref{A2.1}--\ref{A2.2}, uniformly in
    \(t\in E\), \(\tau\in Q_u\), and for each fixed \(\bm x\),
    \begin{equation*}
      \bm m_{u,\tau,\bm x}(t)
      \longrightarrow
      \begin{pmatrix}
        -\bm d_I(t)-\bm x_I\\
        -\bm x_{\Jtan}\\
        \bm 0_{\Jsl}
      \end{pmatrix}.
    \end{equation*}
    Indeed, the first term in the active coordinates is exactly
    \(-u^2[\Sigma_{u,\tau}-R_{u,\tau}(t,0)]\Sigma_{u,\tau}^{-1}\bb\). Since
    \(R_{u,\tau}(t,0)\Sigma_{u,\tau}^{-1}=I+o(u^{-1})\) uniformly by
    \ref{A2.1}, the term containing \(\bm x/\overline{\bm u}\) contributes
    \(-\bm x_I+o(1)\) in the active coordinates and \(-\bm x_{\Jtan}+o(1)\) in the
    coordinates \(\Jtan\subset J\). In \(\Jsl\), both contributions vanish
    after multiplication by \(u^{-1}\).

    The conditional covariance is
    \begin{align*}
      &\E{\big[\bm\chi_{u,\tau,\bm x}(t)-\bm m_{u,\tau,\bm x}(t)\big]
        \big[\bm\chi_{u,\tau,\bm x}(s)-\bm m_{u,\tau,\bm x}(s)\big]^\top} \\
      &\quad =
      \diag(\uu)\Big[
        R_{u,\tau}(t,s)
        -R_{u,\tau}(t,0)\Sigma_{u,\tau}^{-1}R_{u,\tau}(0,s)
      \Big]\diag(\uu) \\
      &\quad \longrightarrow
      \begin{pmatrix}
        K_{II}(t,s)&0\\
        0&0
      \end{pmatrix},
    \end{align*}
    uniformly in \(t,s\in E\) and \(\tau\in Q_u\). Assumption~\ref{A3}
    gives the required tightness in \(C(E,\R^d)\) after the same scaling, and
    hence
    \begin{equation*}
      \bm\chi_{u,\tau,\bm x}+u\uu(\bb-\bm b)
      \Rightarrow
      \begin{pmatrix}
        \bm W_I-\bm d_I-\bm x_I\\
        -\bm x_{\Jtan}\\
        (\bb-\bm b)_{\Jsl}
      \end{pmatrix}
    \end{equation*}
    in \(C(E,\R^d)\), uniformly in \(\tau\in Q_u\) at the level of finite-dimensional
    distributions and tightness. Equivalently, every sequence \(\tau_u\in Q_u\)
    has the same weak limit displayed above. Therefore, by
    conditions~\ref{F2}--\ref{F3}, the portmanteau theorem gives, for
    Lebesgue-a.e. \(\bm x\),
    \begin{equation*}
      a_{u,\tau}(\bm x)\longrightarrow
      a(\bm x)
      \coloneqq
      \pk{
        \Gamma
        \begin{pmatrix}
          \bm W_I-\bm d_I-\bm x_I\\
          -\bm x_{\Jtan}\\
          (\bb-\bm b)_{\Jsl}
        \end{pmatrix}
        >0
      }
    \end{equation*}
    uniformly in \(\tau\in Q_u\).

    Next, with \(\bm z_u=\bm x/\overline{\bm u}\),
    \begin{equation*}
      \varphi_{\Sigma_{u,\tau}}(u\bb-\bm z_u)
      =\varphi_{\Sigma_{u,\tau}}(u\bb)
      \exp\left(
        \bm x_I^\top\bm w_I
        -\frac12\bm x_J^\top(\Sigma^{-1})_{JJ}\bm x_J
        +\theta_{u,\tau}(\bm x)
      \right),
    \end{equation*}
    where the identity is exact with
    \begin{align*}
      \theta_{u,\tau}(\bm x)
      &=u\bb^\top(\Sigma_{u,\tau}^{-1}-\Sigma^{-1})\bm z_u
        -\frac12 \bm z_u^\top\Sigma_{u,\tau}^{-1}\bm z_u
        +\frac12\bm x_J^\top(\Sigma^{-1})_{JJ}\bm x_J .
    \end{align*}
    For each fixed \(\bm x\), the active part of the first term tends to zero by
    \ref{A1.1}, the inactive part tends to zero by~\ref{A1.2}, and the quadratic
    term tends to zero because \(\bm z_u\to(0_I,\bm x_J)\) and
    \(\Sigma_{u,\tau}^{-1}\to\Sigma^{-1}\). Hence
    \(\theta_{u,\tau}(\bm x)\to0\), uniformly in \(\tau\).

    It remains to justify passage of the limit under the integral.  This is the
    uniform domination step in the proof of~\cite[Lemma~4.7]{VVGP}; we record
    why its hypotheses match the present formulation.  By~\ref{F1},
    \(a_{u,\tau}(\bm x)\) is bounded by a simultaneous-exceedance probability
    for the conditioned field.  Assumptions~\ref{A2.1} and~\ref{A3}, followed
    by Gaussian entropy and Borell--TIS bounds, give a Gaussian tail whenever
    an active coordinate of \(\bm x_I\) is large and positive.  On the negative
    active orthant the factor \(e^{\bm w_I^\top\bm x_I}\) is integrable because
    \(\bm w_I>\bm0_I\).  Uniform positive definiteness of
    \(\Sigma_{u,\tau}\) gives a Gaussian factor in \(\bm x_J\).  The only
    potentially linear inactive term is
    \(u(\Sigma_{u,\tau}^{-1}\bb)_J^\top\bm x_J\), and it is
    \(o(|\bm x_J|)\) uniformly by~\ref{A1.2}; completing the square therefore
    preserves an integrable Gaussian envelope.  Assumption~\ref{A1.1}
    controls the remaining density coefficients.  Consequently the normalized
    integrand is bounded, uniformly for all large \(u\) and \(\tau\in Q_u\),
    by an \(L^1(\R^d)\) function.

    Dominated convergence yields
    \begin{align*}
      \pk{\Gamma(\uu(\bm X_{u,\tau}-u\bm b))>0}
      &\sim
      u^{-|I|}\varphi_{\Sigma_{u,\tau}}(u\bb)
      \int_{\R^d}
      \exp\left(
        \bm x_I^\top\bm w_I
        -\frac12\bm x_J^\top(\Sigma^{-1})_{JJ}\bm x_J
      \right)a(\bm x)\,d\bm x,
    \end{align*}
    uniformly in \(\tau\in Q_u\), which is the asserted formula.
  \end{proof} 

The following corollary specifies Lemma~\ref{vectorPickands_functionals} for the
case when the functional \( \Gamma \) satisfies a constant reduction
property~\ref{B5}.

\begin{corollary}
    \label{corollary1}
    Under the conditions of Lemma~\ref{vectorPickands_functionals}, suppose
    additionally that for every proper \(K\) there is a Borel coordinate-deleted
    map \(\Gamma^{(-K)}:C(E,\R^{|K^c|})\to\R\), written as
    \(\Gamma(\bm f_{K^c},\bm\infty_K)\), and that
    \begin{enumerate}
      \item [(\namedlabel{F4}{\( \, \mathbf{F}_4 \)})] If \( \bm{f}_K \) is a
      constant subvector of \( \bm{f} \), then \( \Gamma( \bm{f} ) > 0 \iff
      \Gamma ( \bm{f}_{K^c}, \bm{\infty}_K ) > 0 \) and
      \( \bm{f}_K > \bm{0} \),
    \end{enumerate}
    then as \( u \to \infty \) we have
    \begin{equation*}
      \pk{\Gamma ( \uu ( \bm{X}_{u, \tau} - u \bm{b} ) ) > 0}
      \sim
      H_{\Gamma, \bm{W}, \bm{d}, \bm{w}} \, \pk{\bm{X}_{u, \tau} ( 0 ) > u \, \bm{b}},
    \end{equation*}
    where
    \begin{equation*}
      H_{\Gamma, \bm{W}, \bm{d}, \bm{w}}
      = \int_{\R^{|I|}} 
      e^{\bm{1}_I^\top \bm{x}_I} \,
      \pk{
        \Gamma ( \bm{W}_I - \bm{d}_I - \bm{x}_I / \bm{w}_I, \bm{\infty}_{J} ) > 0
      }
      \, d \bm{x}_I.
    \end{equation*}
  \end{corollary}
  
  \begin{proof}
    By~\ref{F4} and the fact that \(\tilde b_i>b_i\) for \(i\in\Jsl\), we have
    \begin{equation}
      \label{DG_condition}
      \pk{
      \Gamma 
      \begin{pmatrix}
        \bm{W}_I - \bm{d}_I - \bm{x}_I \\
        -\bm{x}_{\Jtan}                                  \\
        ( \bb - \bm{b} )_{\Jsl}
      \end{pmatrix}
      > 0
      }
      =
      \Ind{\bm{x}_{\Jtan} < \bm{0}_{\Jtan}} \,
      \pk{
      \Gamma 
      \begin{pmatrix}
        \bm{W}_I - \bm{d}_I - \bm{x}_I \\
        \bm{\infty}_J
      \end{pmatrix}
      > 0
      }
    \end{equation}
    It remains to note that
    \begin{equation*}
      \pk{\bm{X}_{u, \tau} ( 0 ) > u \bm{b}}
      \sim
      u^{-|I|} \,
      \varphi_{\Sigma_{u, \tau}} ( u \bb )
      \prod_{i \in I} w_i^{-1}
      \int_{\R^{|J|}}
      \exp \left( 
        -\frac{1}{2} \, \bm{x}_J^\top ( \Sigma^{-1} )_{J J} \, \bm{x}_J
      \right) 
      \, \Ind{\bm{x}_{\Jtan} < \bm{0}_{\Jtan}}
      \, d \bm{x}_J
    \end{equation*}
    and compare it with the expression for \( \pk{\Gamma ( \uu ( \bm{X}_{u,
    \tau} - u \bm{b} ) ) > 0} \) obtained in
    Lemma~\ref{vectorPickands_functionals}. The remaining active-coordinate
    integral is transformed by \(y_i=w_i x_i\), which gives precisely the
    displayed constant \(H_{\Gamma,\bm W,\bm d,\bm w}\).
  \end{proof}

Let \( D(g) \) denote the set of all continuity points of a function \( g : F_1
\to F_2 \), where \( F_1 \) and \( F_2 \) are two topological spaces. Next, we introduce the following partial order on \( \R^d \): \( \bm{x}' \succ \bm{x} \) if \( x_i' \geq x_i \) for every \(i\) and there exists \( i \in \{
1, \dots , d \} \) such that \( x_i' > x_i \). In the following, we shall say that \( G : \R^d \to \R \) is monotone increasing
if \( \bm{x}' \succ \bm{x} \) implies \( G ( \bm{x}' ) \geq G ( \bm{x} ) \). We shall also say that \( L \) is a plateau level of a monotone function \( G \)
if \( \lambda ( G^{-1} ( L ) ) > 0 \), where \( \lambda \) is the standard
Lebesgue measure on \( \R^d \).

\begin{lemma}\label{sojorn_F3} Let \( \bm{W}(t), \ t \in E \) be a continuous stochastic process, \( G : \R^d
  \to \R \) a Borel function bounded on compacta and \( \mu \) a finite measure on
  \( E \). If
  \begin{equation}
    \label{F4_G_condition}
    \pk{ \mu ( F^c ) = 0 } = 1
    \quad \text{for almost all } \bm{x} \in \R^d,
  \end{equation}
  where
  \begin{equation}
    \label{sojourn F3 def F}
    F \coloneqq \{ t \in E : \bm{W}(t) - \bm{x} \in D(G) \},
  \end{equation}
  then \( \Gamma : C ( E, \R^d ) \to \R \) defined by
  \begin{equation}
    \label{def:sojourn}
    \Gamma ( \bm{f} )
    = \int_E G ( \bm{f} ( t ) ) \, d \mu ( t ) - L
  \end{equation}
  satisfies condition~\ref{F3}. If \( G \) is monotone, then
  condition~\eqref{F4_G_condition} is satisfied.
\end{lemma}

\begin{proof}
  Fix \(\bm{x}\) for which \eqref{F4_G_condition} holds and put
  \(\bm f=\bm W-\bm x\). On the event \(\mu(F^c)=0\), let
  \(\bm f_n\to\bm f\) in \(C(E,\R^d)\). For \(t\in F\), \(\bm f(t)\in D(G)\),
  so \(G(\bm f_n(t))\to G(\bm f(t))\). Moreover, for all large \(n\),
  \(\bm f_n(E)\) is contained in a fixed compact neighbourhood of the compact
  set \(\bm f(E)\). Since \(G\) is bounded on compacta, there is a finite
  bound \(M<\infty\) such that
  \[
    |G(\bm f_n(t))|+|G(\bm f(t))|\le M,
    \qquad t\in E,
  \]
  for all large \(n\). The dominated convergence theorem on \(F\), together
  with \(\mu(F^c)=0\), gives
  \[
    \int_E G(\bm f_n(t))\,d\mu(t)
    =\int_F G(\bm f_n(t))\,d\mu(t)
    \longrightarrow
    \int_F G(\bm f(t))\,d\mu(t)
    =\int_E G(\bm f(t))\,d\mu(t).
  \]
  Hence \(\bm W-\bm x\) is a continuity point of \(\Gamma\) almost surely for
  Lebesgue-a.e. \(\bm x\), which is condition~\ref{F3}.

  It remains to verify \eqref{F4_G_condition} under monotonicity. A
  finite-valued coordinatewise monotone function on \(\R^d\) has a
  Lebesgue-null discontinuity set; put \(N=D(G)^c\), so \(\lambda(N)=0\).
  Tonelli's theorem gives
  \begin{align*}
    \int_{\R^d}
    \E{\mu\{t\in E:\bm W(t)-\bm x\in N\}}\,d\bm x
    &=
    \int_E
    \E{\int_{\R^d}\Ind{\bm W(t)-\bm x\in N}\,d\bm x}\,d\mu(t) \\
    &=\int_E\lambda(N)\,d\mu(t)=0.
  \end{align*}
  Hence \(\E{\mu(F^c)}=0\), and therefore \(\mu(F^c)=0\) almost surely, for
  Lebesgue-a.e. \(\bm x\).
\end{proof}

Recall that a \emph{capacity} on \( E \) is a set function
\( \mu : \mathcal B(E) \to [0, \infty] \) such that \( \mu(\emptyset) = 0 \) and
\( \mu(A) \le \mu(B) \) whenever \( A \subset B \); see
Definition~\ref{def:capacity-terminology}.  A capacity \( \mu \) on \( E \) is
said to be
\begin{itemize}
  \item \emph{finite} if \( \mu(E) < \infty \);
  \item \emph{null-additive} if \( \mu (A) = 0 \implies \mu(A \cup B) = \mu(B)\)
  for every Borel set \(B\subset E\);
  \item \emph{upper regular} if \( \mu(A) = \inf \{ \mu(G) : A \subset G, \ G
  \text{ is open} \} \);
  \item \emph{lower regular} if \( \mu(A) = \sup \{ \mu(K) : K \subset A, \ K
  \text{ is compact} \} \);
  \item \emph{regular} if it is both upper and lower regular;
  \item \emph{convex} (or \emph{supermodular}) if \( \mu ( A \cup B ) + \mu ( A
  \cap B ) \geq \mu ( A ) + \mu ( B ) \);
  \item \emph{dominated by a measure \( \mu_0 \)} if \( \mu_0(A) = 0 \implies
  \mu(A) = 0 \).
\end{itemize}
For this terminology and the corresponding Choquet-integral theory,
see~\cite{Denneberg1994}; the foundational theory of capacities is due
to~\cite{choquet-ref}.

\begin{lemma}\label{sojorn_F3_cap} Let \( \bm{W} ( t ), \ t \in E \) be a continuous stochastic process, \( G :
  \R^d \to \R_+ \) a Borel function bounded on compacta and \( \mu \) a finite
  capacity on \( E \). For a Borel set \(A\subset E\), denote
  \begin{equation}
    \label{sojourn F3 def S}
      S(A)
      \coloneqq
      \left\{ 
        s \geq 0 :
        \bm{g} \mapsto \mu \{ t \in A : G ( \bm{g}(t) ) > s \}
        \text{ is discontinuous at } 
        \bm{g} = \bm{W} - \bm{x}
      \right\}.
\end{equation}
  Let \( F \) be defined by~\eqref{sojourn F3 def F}. If either of the following
  conditions holds:
  \begin{enumerate}[a)]
    \item \label{sojourn F3 condition a}
    \( 
    \pk{ \lambda ( S(E) ) = 0 } = 1
    \)
    for almost all \( \bm{x} \in \R^d \),
    \item \label{sojourn F3 condition b}
    \( \mu \) is null-additive, \( 
    \pk{ \lambda ( S(F) ) = 0 } = 1
    \)
    for almost all \( \bm{x} \in \R^d \) and condition~\eqref{F4_G_condition} is
    satisfied,
  \end{enumerate}
  then \( \Gamma : C ( E, \R^d ) \to \R \) defined by the Choquet integral
  \begin{equation*}
    \Gamma(\bm f)=\int_0^\infty \mu\{t\in E:G(\bm f(t))>s\}\,ds-L
  \end{equation*}
  satisfies condition~\ref{F3}.
  \smallskip

  A simple sufficient condition for the requirement involving \(S(E)\) is the
  following: if \(G\) is continuous, then \(\lambda(S(E))=0\) for every
  realization of \(\bm W-\bm x\), and the resulting Choquet functional is
  continuous at every path in \(C(E,\R^d)\). More generally, the same conclusion
  about \(S(A)\) holds on any Borel set \(A\subset E\) whenever
  \(G(\bm g_n(\cdot))\to G(\bm g(\cdot))\) uniformly on \(A\) for all
  \(\bm g_n\to\bm g\) in \(C(E,\R^d)\).
\end{lemma}

\begin{remark}
  The auxiliary set \(S(F)\) can sometimes be checked after scalarization,
  but a continuity hypothesis is essential.  Fix a path \(\bm g\) and assume
  that every sequence \(\bm g_n\to\bm g\) in \(C(E,\R^d)\) satisfies
  \[
    \sup_{t\in F}
    |G(\bm g_n(t))-G(\bm g(t))|\longrightarrow0.
  \]
  For the bounded function \(u=G(\bm g)|_F\), define
  \begin{equation*}
    \widetilde S(F,u)
    =
    \left\{
      s\ge0:
      v\mapsto\mu\{t\in F:v(t)>s\}
      \text{ is discontinuous at }u
      \text{ in the uniform norm}
    \right\}.
  \end{equation*}
  Then \(S(F)\subset\widetilde S(F,G(\bm g)|_F)\).  Without uniform
  continuity of the scalarized compositions along perturbations of \(\bm g\),
  this inclusion need not hold.

  For example, let
  \(\mu(A)=\pk{A\cap X\ne\varnothing}\), where \(X\) is a random closed
  subset of \(F\) and the relevant suprema are measurable.  Then
  \begin{equation*}
    \widetilde S(F,u)
    =
    \left\{s\ge0:
      \pk{\sup_{x\in X}u(x)=s}>0
    \right\},
  \end{equation*}
  the set of atoms of a real random variable, and is therefore at most
  countable.  If \(X=F\) almost surely, this set is the singleton
  \(\{\sup_{x\in F}u(x)\}\).
\end{remark}

\begin{proof}[Proof of Lemma~\ref{sojorn_F3_cap}]
  Fix \(\bm x\) outside the exceptional null sets in the assumptions and put
  \(\bm f=\bm W-\bm x\).

  \textbf{Case~\ref{sojourn F3 condition a}.} On the event
  \(\lambda(S(E))=0\), let \(\bm f_n\to\bm f\) in \(C(E,\R^d)\). Since \(G\) is
  bounded on compacta and non-negative, there is \(M<\infty\) such that
  \(0\leq G(\bm f_n(t)),G(\bm f(t))\leq M\) for all large \(n\) and all
  \(t\in E\). Thus
  \[
    \mu\{t\in E:G(\bm f_n(t))>s\}\leq \mu(E)\Ind{s<M},
  \]
  which is an \(L^1(0,\infty)\) majorant. For every
  \(s\notin S(E)\), the level-set capacity is continuous at \(\bm f\), so
  \[
    \mu\{t\in E:G(\bm f_n(t))>s\}
    \to
    \mu\{t\in E:G(\bm f(t))>s\}.
  \]
  Dominated convergence in \(s\) proves continuity of the Choquet integral at
  \(\bm f\). Hence \(\pk{\bm W-\bm x\in D(\Gamma)}=1\).

  \textbf{Case~\ref{sojourn F3 condition b}.} On the event
  \(\lambda(S(F))=0\) and \(\mu(F^c)=0\), null-additivity gives, for every
  \(\bm g\in C(E,\R^d)\) and every \(s\ge0\),
  \[
    \mu\{t\in E:G(\bm g(t))>s\}
    =\mu\{t\in F:G(\bm g(t))>s\}.
  \]
  Repeating the preceding dominated-convergence argument on \(F\) gives
  \(\pk{\bm W-\bm x\in D(\Gamma)}=1\).

  \textbf{Sufficient condition.} Assume that, for the relevant Borel set
  \(A\subset E\), one has
  \(G(\bm g_n(\cdot))\to G(\bm g(\cdot))\) uniformly on \(A\) whenever
  \(\bm g_n\to\bm g\) in \(C(E,\R^d)\). For fixed \(\bm g\), put
  \(u(t)=G(\bm g(t))\). If \(u_n\to u\) uniformly on \(A\), then, for every
  \(s>0\), every \(\varepsilon\in(0,s)\), and all large \(n\),
  \begin{equation*}
    \{t\in A:u(t)>s+\varepsilon\}
    \subset
    \{t\in A:u_n(t)>s\}
    \subset
    \{t\in A:u(t)>s-\varepsilon\}.
  \end{equation*}
  Therefore the level-set map is continuous at every positive continuity point
  of the non-increasing function
  \(s\mapsto\mu\{t\in A:u(t)>s\}\).  A monotone real function has at most
  countably many discontinuities, and the single endpoint \(s=0\) is immaterial;
  hence \(\lambda(S(A))=0\).

  There is also a direct continuity estimate.  For bounded non-negative
  functions \(u,v\) on \(A\), put \(\delta=\|u-v\|_\infty\).  Monotonicity of
  the Choquet integral and the layer-cake identity for addition of a constant
  give
  \[
    \int_Au\,d\mu
    \le \int_A(v+\delta)\,d\mu
    =\int_Av\,d\mu+\delta\mu(A),
  \]
  and the same estimate with \(u\) and \(v\) interchanged yields
  \begin{equation}
    \label{eq:choquet-supnorm-lipschitz}
    \left|\int_A u\,d\mu-\int_A v\,d\mu\right|
    \le \mu(A)\,\|u-v\|_\infty .
  \end{equation}
  Consequently, if \(G\) is continuous, then the Choquet functional is
  continuous at every \(\bm g\in C(E,\R^d)\): the compactness of \(\bm g(E)\)
  gives uniform convergence of the compositions, and
  \eqref{eq:choquet-supnorm-lipschitz} applies.
\end{proof}

\begin{lemma}\label{sojorn_F2}Let \( G : \R^d \to \R \) be monotone and \( \mu \) be a finite measure on \(
  E \) such that \( \mu(E)>0 \) and \( \supp \mu \) is connected. If \( L / \mu(E) \) is not a
  plateau level of \( G \), then the functional \( \Gamma \) defined
  by~\eqref{def:sojourn} satisfies condition~\ref{F2}. If \( G \) is strictly
  monotone, then the same result holds without the condition that \( \supp \mu
  \) is connected.
\end{lemma}
\begin{proof}
  Assume without loss of generality that \( G \) is monotone increasing. Fix a
  continuous function \( \bm{f}:E\to\R^d \) and put
  \begin{equation*}
    I_{\bm f}(\bm{x})
    =
    \int_E G(\bm f(t)-\bm{x})\,d\mu(t),
    \qquad
    D_{\bm f}
    =
    \{\bm{x}\in\R^d:I_{\bm f}(\bm{x})=L\}.
  \end{equation*}
  We first show that, under either set of assumptions in the statement,
  \(D_{\bm f}\) cannot contain two points of the form
  \(\bm{x}\) and \(\bm{x}+a\bm{1}\), where \(a>0\).

  Suppose, to the contrary, that \(\bm{x},\bm{x}'\in D_{\bm f}\) with
  \(\bm{x}'=\bm{x}+a\bm{1}\), \(a>0\). Then
  \begin{equation*}
    \int_E G(\bm f(t)-\bm{x})\,d\mu(t)
    =
    \int_E G(\bm f(t)-\bm{x}')\,d\mu(t)
    =L.
  \end{equation*}
  Since \(G(\bm f(t)-\bm{x})\geq G(\bm f(t)-\bm{x}')\), this gives
  \begin{equation}\label{eq:measure_diag_zero}
    \int_E
    \left(G(\bm f(t)-\bm{x})-G(\bm f(t)-\bm{x}')\right)
    \,d\mu(t)=0.
  \end{equation}
  If \(G\) is strictly monotone, then the integrand in
  \eqref{eq:measure_diag_zero} is strictly positive for every \(t\in E\). Since
  \(\mu(E)>0\), its integral is strictly positive, a contradiction.

  It remains to consider the non-strictly monotone case with connected support.
  From \eqref{eq:measure_diag_zero} there is a set \(A\subset E\) such that
  \(\mu(A^c)=0\) and
  \begin{equation}\label{eq:measure_diag_equal_a}
    G(\bm f(t)-\bm{x})=G(\bm f(t)-\bm{x}')
    \qquad \text{for all } t\in A.
  \end{equation}
  By monotonicity, \eqref{eq:measure_diag_equal_a} implies
  \begin{equation}\label{eq:measure_diag_equal_box}
    G(\bm f(t)-\bm{y})=G(\bm f(t)-\bm{x})
    \qquad
    \text{for all }t\in A \text{ and all }\bm{y}\in[\bm{x},\bm{x}'].
  \end{equation}
  Let \(S=\supp\mu\). Since \(\mu(E\setminus S)=0\), the set \(A\cap S\) is
  dense in \(S\). Set \(\bm m=(\bm{x}+\bm{x}')/2\) and
  \(\rho=\dist_\infty(\bm m,[\bm{x},\bm{x}']^c)=a/2\). By uniform continuity of
  \(\bm f\), choose \(\delta>0\) such that
  \(\|\bm f(t)-\bm f(s)\|_\infty<\rho\) whenever \(|t-s|<\delta\). If
  \(t,s\in A\cap S\) and \(|t-s|<\delta\), then
  \(\bm m-(\bm f(t)-\bm f(s))\in[\bm{x},\bm{x}']\), and hence
  \begin{align*}
    G(\bm f(t)-\bm{x})
    &=G(\bm f(t)-\bm m)  \\
    &=G\left(\bm f(s)-\big(\bm m-(\bm f(t)-\bm f(s))\big)\right) \\
    &=G(\bm f(s)-\bm{x}),
  \end{align*}
  where the first and third equalities use~\eqref{eq:measure_diag_equal_box}.
  A connected compact metric space is \(\eta\)-chain connected for every
  \(\eta>0\). Taking \(\eta<\delta/3\) and then perturbing the intermediate
  points of an \(\eta\)-chain in \(S\) into the dense set \(A\cap S\), we obtain
  a chain in \(A\cap S\) whose jumps are still smaller than \(\delta\). Thus
  \(t\mapsto G(\bm f(t)-\bm{x})\) is constant on \(A\cap S\); denote the constant
  by \(c\). Since \(A\cap S\) has full \(\mu\)-measure,
  \begin{equation*}
    L=I_{\bm f}(\bm{x})=c\,\mu(E),
    \qquad \text{so } c=L/\mu(E).
  \end{equation*}
  For any \(t_0\in A\cap S\), \eqref{eq:measure_diag_equal_box} gives
  \begin{equation*}
    \bm f(t_0)-[\bm{x},\bm{x}']
    \subset
    G^{-1}\big(L/\mu(E)\big).
  \end{equation*}
  The box \(\bm f(t_0)-[\bm{x},\bm{x}']\) has positive \(d\)-dimensional Lebesgue
  measure, because \(\bm{x}'=\bm{x}+a\bm{1}\) with \(a>0\). This says that
  \(L/\mu(E)\) is a plateau level of \(G\), contradicting the assumption.

  Consequently, for every continuous path \(\bm f\), the set \(D_{\bm f}\) has no
  two points lying on the same line parallel to \(\bm 1\). Indeed, if
  \(\mathcal H=\{\bm h\in\R^d:\langle \bm h,\bm 1\rangle=0\}\), then each section
  \begin{equation*}
    \{r\in\R:\bm h+r\bm 1\in D_{\bm f}\},
    \qquad \bm h\in\mathcal H,
  \end{equation*}
  contains at most one point.
  The map \(I_{\bm f}\) is Borel: monotone finite-valued functions are Borel and
  are bounded on the compact range of \(\bm f-\bm x\), and integration preserves
  Borel measurability in \(\bm x\). Hence \(D_{\bm f}\) is measurable.
  Fubini's theorem in the orthogonal decomposition
  \(\R^d=\mathcal H\oplus\operatorname{span}\{\bm 1\}\) therefore yields
  \(\lambda(D_{\bm f})=0\).

  Applying this pathwise to \(\bm f=\bm W-\bm d\) gives
  \begin{equation*}
    \int_{\R^d}\pk{\Gamma(\bm W-\bm d-\bm{x})=0}\,d\bm{x}
    =
    \E{\lambda(D_{\bm W-\bm d})}=0,
  \end{equation*}
  and hence \(\pk{\Gamma(\bm W-\bm d-\bm{x})=0}=0\) for Lebesgue-a.e.
  \(\bm{x}\in\R^d\).
\end{proof}

We use the capacity terminology of
Definition~\ref{def:capacity-terminology}.  In particular, the
support-intersection property implies \(\mu(E\setminus\supp\mu)=0\), whereas the
separate equality \(\mu(\supp\mu)=\mu(E)\) means that the support carries full
capacity.  We shall also use the standard fact that the Choquet integral of
bounded non-negative functions is superadditive when \(\mu\) is convex
(supermodular); see \cite{Denneberg1994}.

\begin{lemma}\label{sojorn_F2_cap}
Let \(L>0\), let \(G:\R^d\to\R_+\) be continuous and coordinatewise
non-decreasing, and let
\(\mu\) be a regular finite capacity satisfying the support-intersection
property on \(E\), with \(\mu(E)>0\). Define
\[
  \Gamma(\bm f)=\int_0^\infty \mu\{t\in E:G(\bm f(t))>s\}\,ds-L .
\]
\begin{enumerate}[a)]
  \item \label{sojorn F2 condition a}
  If
  \begin{equation}
    \label{eq:strict-above-zero}
    G(\bm y)>0,\quad \bm z\succ\bm y
    \quad\Longrightarrow\quad G(\bm z)>G(\bm y),
  \end{equation}
  then \(\Gamma\) satisfies condition~\ref{F2}.
  \item \label{sojorn F2 condition b}
  If \(\mu\) is convex, \(\supp\mu\) is connected,
  \(\mu(\supp\mu)=\mu(E)\), and \(L/\mu(E)\) is not a plateau level of \(G\),
  then \(\Gamma\) satisfies condition~\ref{F2}.
\end{enumerate}
\end{lemma}
\begin{proof}
  Assume without loss of generality that \(G\) is monotone increasing. For a
  continuous path \(\bm f:E\to\R^d\), write
  \begin{equation*}
    I_{\bm f}(\bm{x})
    =
    \int_E G(\bm f(t)-\bm{x})\,d\mu(t)
    =
    \int_0^\infty
      \mu\{t\in E:G(\bm f(t)-\bm{x})>s\}\,ds,
  \end{equation*}
  and set \(D_{\bm f}=\{\bm{x}:I_{\bm f}(\bm{x})=L\}\).  By continuity of
  \(G\), uniform convergence on compact ranges, and
  \eqref{eq:choquet-supnorm-lipschitz}, the map \(I_{\bm f}\) is continuous;
  in particular, \(D_{\bm f}\) is Borel.  As in the proof of
  Lemma~\ref{sojorn_F2}, it is enough to show that \(D_{\bm f}\) contains no two
  points \(\bm{x}\) and \(\bm{x}+a\bm 1\), \(a>0\).  The same
  diagonal-section Fubini argument then gives \(\lambda(D_{\bm f})=0\) for
  every continuous path \(\bm f\).  Once this pathwise statement has been
  proved, continuity of the Choquet functional makes
  \((\omega,\bm x)\mapsto
  \Ind\{I_{\bm W(\omega)-\bm d}(\bm x)=L\}\) measurable, and Tonelli gives
  \[
    \int_{\R^d}
      \pk{I_{\bm W-\bm d}(\bm x)=L}\,d\bm x
    =
    \E{\lambda(D_{\bm W-\bm d})}=0.
  \]
  This is exactly condition~\ref{F2}.

  Fix \(\bm{x}'=\bm{x}+a\bm 1\), \(a>0\), and suppose for contradiction that
  \(\bm{x},\bm{x}'\in D_{\bm f}\). Denote \(S=\supp\mu\). The
  support-intersection property gives \(\mu(E\setminus S)=0\).

  \smallskip
  \textbf{Case~\ref{sojorn F2 condition a}.}
  Put \(u_{\bm y}(t)=G(\bm f(t)-\bm y)\).  The continuous function
  \(u_{\bm{x}'}\) attains its maximum on the compact set \(S\); choose
  \(t_0\in S\) such that
  \[
    m=\max_{t\in S}u_{\bm{x}'}(t)=u_{\bm{x}'}(t_0).
  \]
  Since \(I_{\bm f}(\bm x')=L>0\) and sets outside \(S\) have zero capacity,
  one has \(m>0\). Pick \(\eta\in(0,a)\). Then
  \(\bm f(t_0)-\bm{x}-\eta\bm 1\succ\bm f(t_0)-\bm{x}'\), and
  \eqref{eq:strict-above-zero} gives
  \[
    r_0
    \coloneqq
    G(\bm f(t_0)-\bm{x}-\eta\bm 1)
    >m.
  \]
  By continuity of \(\bm f\), the set
  \[
    U=\{t\in E:\|\bm f(t)-\bm f(t_0)\|_\infty<\eta\}
  \]
  is an open neighbourhood of \(t_0\), and hence \(\mu(U)>0\). For
  \(t\in U\), monotonicity gives \(u_{\bm x}(t)\ge r_0\). On the other hand,
  \(u_{\bm x'}\le m\) on \(S\), and all subsets of \(E\setminus S\) have zero
  capacity. Therefore, for \(r\in(m,r_0)\),
  \[
    \mu\{u_{\bm x}>r\}\ge\mu(U)>0,
    \qquad
    \mu\{u_{\bm x'}>r\}=0.
  \]
  Integrating in \(r\) gives
  \(I_{\bm f}(\bm x)>I_{\bm f}(\bm x')\), a contradiction.

  \smallskip
  \textbf{Case~\ref{sojorn F2 condition b}.}
  Since \(I_{\bm f}\) is monotone decreasing in \(\bm{x}\), the equality
  \(I_{\bm f}(\bm{x})=I_{\bm f}(\bm{x}')=L\) implies
  \begin{equation}
    \label{eq:cap_box_integral_constant}
    I_{\bm f}(\bm y)=L
    \qquad\text{for every }\bm y\in[\bm{x},\bm{x}'].
  \end{equation}
  Let \(\bm y,\bm z\in(\bm{x},\bm{x}')\) with \(\bm z\succ\bm y\), and set
  \[
    h_{\bm y,\bm z}(t)
    =G(\bm f(t)-\bm y)-G(\bm f(t)-\bm z)\ge0.
  \]
  By superadditivity of the Choquet integral for a convex capacity,
  \begin{align*}
    L
    &=I_{\bm f}(\bm y)
      =\int_E\big(G(\bm f(t)-\bm z)+h_{\bm y,\bm z}(t)\big)\,d\mu(t)\\
    &\ge I_{\bm f}(\bm z)+\int_Eh_{\bm y,\bm z}(t)\,d\mu(t)
      =L+\int_Eh_{\bm y,\bm z}(t)\,d\mu(t).
  \end{align*}
  Thus \(\int_Eh_{\bm y,\bm z}\,d\mu=0\). If
  \(h_{\bm y,\bm z}(t_0)>0\) at some \(t_0\in S\), continuity gives a
  relatively open neighbourhood \(U\) of \(t_0\) on which
  \(h_{\bm y,\bm z}\) is bounded below by a positive constant. Then
  \(\mu(U)>0\), so its Choquet integral is positive, a contradiction. Hence
  \begin{equation}
    \label{eq:cap_equal_on_support}
    G(\bm f(t)-\bm y)=G(\bm f(t)-\bm z),
    \qquad t\in S,
  \end{equation}
  for comparable \(\bm y,\bm z\) in the open box. If they are not comparable,
  choose \(\bm q\) in the box below both and apply the preceding equality twice;
  thus \eqref{eq:cap_equal_on_support} holds for all pairs in the open box.

  Let \(\bm m=(\bm{x}+\bm{x}')/2\) and
  \(\rho=\dist_\infty(\bm m,(\bm{x},\bm{x}')^c)>0\). If \(t,s\in S\) and
  \(\|\bm f(t)-\bm f(s)\|_\infty<\rho\), then
  \(\bm m-(\bm f(t)-\bm f(s))\in(\bm{x},\bm{x}')\), so
  \eqref{eq:cap_equal_on_support} gives
  \[
    G(\bm f(t)-\bm m)=G(\bm f(s)-\bm m).
  \]
  Therefore \(t\mapsto G(\bm f(t)-\bm m)\) is locally constant on the connected
  set \(S\), hence constant there; call the constant \(c\).  We next justify
  that points outside \(S\) do not affect the integral in this convex-capacity
  case.  For every Borel \(A\subset E\), supermodularity applied to \(A\) and
  \(S\), together with \(\mu(S)=\mu(E)\), gives
  \[
    \mu(A)+\mu(S)
    \le \mu(A\cup S)+\mu(A\cap S)
    \le \mu(E)+\mu(A\cap S)
    =\mu(S)+\mu(A\cap S).
  \]
  Hence \(\mu(A)=\mu(A\cap S)\).  The Choquet integral of a function equal
  to \(c\) on \(S\) is therefore \(c\mu(S)=c\mu(E)\).  From
  \eqref{eq:cap_box_integral_constant},
  \[
    L=I_{\bm f}(\bm m)=c\mu(E),
    \qquad c=L/\mu(E).
  \]
  For any \(t_0\in S\),
  \[
    \bm f(t_0)-(\bm{x},\bm{x}')
    \subset G^{-1}\bigl(L/\mu(E)\bigr).
  \]
  The left-hand side is an open box of positive Lebesgue measure, contradicting
  the non-plateau assumption.

  In both cases diagonal pairs have been ruled out, and the diagonal-section
  argument completes the proof.
\end{proof}

\begin{corollary}[No atoms under partial coordinate shifts]
\label{cor:choquet-partial-shifts}
Let \(L>0\), let \(G:\R^d\to\R_+\) be continuous and coordinatewise
non-decreasing, and let
\(\mu\) be a regular finite capacity satisfying the support-intersection
property on \(E\), with \(\mu(E)>0\).  Let
\(D\subseteq\{1,\ldots,d\}\) be non-empty, let
\(\bm c\in\R^{D^c}\), and put
\[
  G_{D,\bm c}(\bm y_D)=G(\bm y_D,\bm c_{D^c}).
\]
Let \(\bm U\) be a continuous \(\R^{|D|}\)-valued stochastic process on
\(E\), and suppose that \(G_{D,\bm c}\) and \(\mu\) satisfy either
condition~\ref{sojorn F2 condition a} or
condition~\ref{sojorn F2 condition b} of
Lemma~\ref{sojorn_F2_cap}.  Then
\[
  \pk{
    \int_E G_{D,\bm c}(\bm U(t)-\bm x_D)\,d\mu(t)=L
  }=0
\]
for Lebesgue-a.e. \(\bm x_D\in\R^{|D|}\).  If this expression is embedded in a
full shift vector \(\bm x\in\R^d\) but is independent of
\(\bm x_{D^c}\), the same no-atom statement holds for Lebesgue-a.e.
\(\bm x\in\R^d\).
\end{corollary}

\begin{proof}
Apply Lemma~\ref{sojorn_F2_cap} in dimension \(|D|\), with scalarization
\(G_{D,\bm c}\) and path \(\bm U\).  The last assertion follows from Tonelli's
theorem for the product decomposition
\(\R^d=\R^{|D|}\times\R^{|D^c|}\).
\end{proof}

\begin{remark}
  The connected-support hypothesis is a convenient path-independent sufficient
  condition, but it is not necessary.  For an ordinary measure and a fixed
  path, the condition
  \begin{equation*}
    \mu \Big\{ t : G ( \bm{f} ( t ) - \bm{x} ) > G ( \bm{f} ( t ) - \bm{x}' ) \Big\} > 0
  \end{equation*}
  is exactly what makes the two integrals strictly different for the given pair
  of shifts.  For a genuinely non-additive capacity, positivity of this set
  alone need not force strict inequality of the two Choquet integrals, which is
  why Lemma~\ref{sojorn_F2_cap} uses the stronger support and convexity
  hypotheses in part~\ref{sojorn F2 condition b}.  The following additive
  counterexample shows why some strictness condition is needed when the support
  is disconnected.  Let \(E=[-1,1]\), take the distinct points
  \(t_0=-1\) and \(t_1=1\), and put
  \(\mu=\delta_{t_0}+\delta_{t_1}\).  Define the continuous non-decreasing
  function
  \[
    G(y)=
    \begin{cases}
      1, & y\le-1,\\
      (y+3)/2, & -1<y<1,\\
      2, & y\ge1.
    \end{cases}
  \]
  Choose a continuous function \(f\) with \(f(t_0)=2\) and \(f(t_1)=-2\),
  and take \(x=0\) and \(x'=1/2\).  Then
  \begin{equation*}
    \int_E G(f(t)-x)\,d\mu(t)
    =G(2-x)+G(-2-x)
    =3
    =G(2-x')+G(-2-x')
    =\int_E G(f(t)-x')\,d\mu(t).
  \end{equation*}
  Here \(L=3\) and \(L/\mu(E)=3/2\) is not a plateau level of \(G\); the
  failure is caused by the disconnected support.
\end{remark} 

\begin{lemma}
  Let \( N \in \N \) and \( \Gamma_i : C ( E, \R^d ) \to \R \), \( i = 1, \dots,
  N \) be functionals satisfying~\ref{F2}. Then \( \Gamma_{\min} ( \bm{f} )
  \coloneqq \min_{i = 1, \dots, N} \Gamma_i ( \bm{f} ) \) and \( \Gamma_{\max} (
  \bm{f} ) \coloneqq \max_{i = 1, \dots, N} \Gamma_i ( \bm{f} ) \) also
  satisfy~\ref{F2}. If \( \Gamma_i \) satisfy~\ref{F3}, then \( \Gamma_{\min} \)
  and \( \Gamma_{\max} \) also satisfy~\ref{F3}.
\end{lemma}
\begin{proof}
  For both choices \(\Gamma_*=\Gamma_{\min}\) and
  \(\Gamma_*=\Gamma_{\max}\), the event \(\{\Gamma_*=0\}\) is contained in
  \(\bigcup_{i=1}^N\{\Gamma_i=0\}\). Hence
  \begin{equation*}
    \pk{ \Gamma_* ( \bm{W} - \bm{d} - \bm{x} ) = 0 }
    \leq
    \sum_{i = 1}^{N} \pk{ \Gamma_i ( \bm{W} - \bm{d} - \bm{x} ) = 0}
    = 0
  \end{equation*}
  for almost all \( \bm{x} \in \R^d \). Moreover, with
  \( \Gamma_* = \Gamma_{\min} \) or \( \Gamma_* = \Gamma_{\max} \), we have
  \begin{equation*}
    \pk{\bm{W} - \bm{d} - \bm{x} \in D ( \Gamma_* )^c}
    \leq \pk{\bm{W} - \bm{d} - \bm{x} \in \bigcup_{i = 1}^N D ( \Gamma_i )^c }
    \leq \sum_{i = 1}^N \pk{ \bm{W} - \bm{d} - \bm{x} \in D ( \Gamma_i )^c }
    = 0
  \end{equation*}
  for almost all \( \bm{x} \in \R^d \).
\end{proof}

\begin{corollary}
    For the process \(\bm W-\bm d\) appearing in conditions~\ref{F2}--\ref{F3},
    if \( G_i : \R \to \R \), \( i = 1, \dots, d \), are
    monotone and \( \mu_i \), \( i = 1, \dots, d \), are finite measures on
    \(E\) with \(\mu_i(E)>0\) and connected supports, then
    \begin{equation*}
      \Gamma ( \bm{f} ) = \min_{i = 1, \dots, d} 
      \left\{ \int_E G_i ( f_i ( t ) ) \, d \mu_i ( t ) - S_i \right\} 
    \end{equation*}
    satisfies~\ref{F3} for all \( \bm{S} \in \R^d \) and~\ref{F2} for all \(
    \bm{S} \) such that \( S_i / \mu_i ( E ) \) is not a plateau level of \( G_i
    \) for every \(i\).
\end{corollary} 

\begin{proof}
  For each \(i\), apply Lemmas~\ref{sojorn_F3} and~\ref{sojorn_F2} to the
  one-dimensional process \(W_i-d_i\) and the functional
  \[
    \Gamma_i(\bm f)=\int_E G_i(f_i(t))\,d\mu_i(t)-S_i.
  \]
  Lemma~\ref{sojorn_F3} gives~\ref{F3}, while the connected support and
  non-plateau assumptions give~\ref{F2} by Lemma~\ref{sojorn_F2}. Fubini in the
  unused coordinates turns the one-dimensional ``almost every shift'' statements
  into statements for Lebesgue-a.e. \(\bm x\in\R^d\). The displayed functional is
  \(\min_i\Gamma_i\), so the preceding lemma for finite minima preserves
  conditions~\ref{F2} and~\ref{F3}.
\end{proof}

The next corollary allows us to apply our results to the non-Gaussian processes
that are obtained from Gaussian processes by applying continuous operators.

\begin{corollary}\label{composition}
    Let \( E_1 \) and \( E_2 \) be compact sets, \(\mu\) a finite measure on
    \(E_2\) with \(\mu(E_2)>0\), and
    \( \mathscr{ A }: C( E_1, \R^d ) \to C( E_2, \R^d ) \) a continuous operator
    satisfying
    \begin{enumerate}
      \item [(1)] For all \( \bm{f} \in C( E_1, \R^d ) \),
      \[
      \sup_{ t \in E_2 } \min_{ i = 1, \dots, d } \mathscr{ A } ( \bm{f} )_i ( t ) > 0 
      \implies \sup_{ s \in E_1 } \min_{ j = 1, \dots, d } f_j( s ) > 0 .
      \]
      \item [(2)] For all \( \bm{x} \in \R^d \),
      \[
        \mathscr{ A } ( \bm{f} + \bm{x} )(t) = \mathscr{ A } ( \bm{f} )(t) + \bm{x} 
        \quad \forall t \in E_2 .
      \]
    \end{enumerate}
    Let \(G:\R^d\to\R\) be monotone and bounded on compacta, and assume that
    \(G,\mu,L\) satisfy the hypotheses of Lemma~\ref{sojorn_F2}. Assume
    additionally that \(L>0\) and that
    \( G( \bm{x} ) \leq L/\mu( E_2 ) \) whenever
    \( \min_{ i = 1, \dots, d } x_i \leq 0\). Then the functional
    \[
      \Gamma( \bm{f} ) = \int_{ E_2 } G( \mathscr{ A } ( \bm{f} )(t) ) \, d \mu( t ) - L
      \]
    satisfies conditions~\ref{F1}--\,\ref{F3}.
  \end{corollary}
  \begin{proof}
    To check~\ref{F1}, suppose \(\Gamma(\bm f)>0\). Then
    \(\int_{E_2}G(\mathscr A(\bm f)(t))\,d\mu(t)>L\). Hence
    \(G(\mathscr A(\bm f)(t))>L/\mu(E_2)\) for some \(t\in E_2\); otherwise the
    integral would be at most \(L\). By the displayed assumption on \(G\), this
    implies
    \(\sup_{t\in E_2}\min_i\mathscr A(\bm f)_i(t)>0\), and then property~(1)
    gives \(\sup_{s\in E_1}\min_i f_i(s)>0\).

    Put \(\bm Z=\bm W-\bm d\). For~\ref{F2}, property~(2) gives
    \[
      \Gamma(\bm Z-\bm x)
      =\int_{E_2}G(\mathscr A(\bm Z)(t)-\bm x)\,d\mu(t)-L.
    \]
    Lemma~\ref{sojorn_F2}, applied pathwise to the continuous process
    \(\bm Y=\mathscr A(\bm Z)\), yields the no-atom condition. For~\ref{F3},
    Lemma~\ref{sojorn_F3} applies to the same process \(\bm Y\), since \(G\) is
    monotone.
  \end{proof}

\begin{lemma}\label{lem:moving_window_functional}
    Let \( E_0 \) and \( K \) be non-empty compact sets in \( \R^n \), and put
    \(E=E_0+K\). Then the functional \(\Gamma:C(E,\R^d)\to\R\) defined by
    \[
      \Gamma ( \bm{f} ) 
      = \sup_{ t \in E_0 } 
      \inf_{ s \in K } 
      \min_{ i = 1, \dots, d } f_i ( t + s )
    \]
    satisfies conditions~\ref{F1}--\,\ref{F3}.  If \(0\in K\), then it also
    has the stronger start-set witness
    \[
      \Gamma(\bm f)>0
      \Longrightarrow
      \sup_{t\in E_0}\min_i f_i(t)>0.
    \]
  \end{lemma}
  \begin{proof}
    If \(\Gamma(\bm f)>0\), then for some \(t\in E_0\) and every \(s\in K\) we
    have \(\min_i f_i(t+s)>0\). Since \(K\neq\emptyset\) and \(t+s\in E\), this
    gives \(\sup_{r\in E}\min_i f_i(r)>0\), proving~\ref{F1}.  If \(0\in K\),
    taking \(s=0\) gives the stronger witness on \(E_0\).

    The map \(\Gamma\) is 1-Lipschitz with respect to the sup-norm: replacing
    \(\bm f\) by \(\bm g\) changes each \(\min_i f_i(t+s)\), hence also the
    infimum and supremum, by at most \(\|\bm f-\bm g\|_\infty\). Therefore every
    path is a continuity point of \(\Gamma\), and~\ref{F3} holds.

    For~\ref{F2}, set
    \[
    H( \bm{x} )
    = \pk{ \Gamma(\bm W-\bm d-\bm x)>0 } .
    \]
    The function \(H\) is coordinatewise non-increasing, hence it is continuous
    at Lebesgue-a.e. \(\bm x\in\R^d\). Also
    \(\Gamma(\bm f-a\bm 1)=\Gamma(\bm f)-a\). Thus, for \(a>0\),
    \[
      \{\Gamma(\bm W-\bm d-\bm x)=0\}
      \subset
      \{\Gamma(\bm W-\bm d-(\bm x-a\bm1))>0\}
      \setminus
      \{\Gamma(\bm W-\bm d-(\bm x+a\bm1))>0\}.
    \]
    At every continuity point of \(H\), the probability of the right-hand side
    tends to zero as \(a\downarrow0\). Hence
    \[
      \pk{\Gamma(\bm W-\bm d-\bm x)=0}=0
    \]
    for Lebesgue-a.e. \(\bm x\), which is condition~\ref{F2}.
  \end{proof}

\begin{lemma}[Parisian no-atom property under partially frozen coordinates]
\label{lem:parisian-partially-frozen}
Let \(E_0,F\subset\R\) be non-empty compact sets with \(0\in F\), and let
\(\bm Z\) be a continuous \(\R^{|I|}\)-valued random process on \(E_0+F\).
Let \(\bm q\) be continuous on the same set and let
\(\bm c\in(0,\infty)^{\Jsl}\).  For
\((\bm x_I,\bm x_{\Jtan})\in\R^{|I|+|\Jtan|}\), put
\[
  \Phi(\bm x_I,\bm x_{\Jtan})
  =
  \Pi_{E_0;F}
  \begin{pmatrix}
    \bm Z-\bm q-\bm x_I\\
    -\bm x_{\Jtan}\\
    \bm c
  \end{pmatrix},
\]
with empty subvectors omitted.  Then
\[
  \pk{\Phi(\bm x_I,\bm x_{\Jtan})=0}=0
\]
for Lebesgue-a.e. \((\bm x_I,\bm x_{\Jtan})\).  Consequently the same statement
holds for Lebesgue-a.e. full \(\bm x\in\R^d\) when the unused coordinates
\(\bm x_{\Jsl}\) are integrated out.
\end{lemma}

\begin{proof}
Let
\[
  c_*=
  \begin{cases}
    \min_{j\in \Jsl}c_j,&\Jsl\ne\emptyset,\\
    +\infty,&\Jsl=\emptyset,
  \end{cases}
\]
and define
\[
  a(\bm x_{\Jtan})
  =
  \begin{cases}
    \min\{c_*,\min_{j\in\Jtan}(-x_j)\},&\Jtan\ne\emptyset,\\
    c_*,&\Jtan=\emptyset.
  \end{cases}
\]
Since minimum with a constant commutes with both the infimum in \(s\) and the
supremum in \(t\),
\[
  \Phi(\bm x_I,\bm x_{\Jtan})
  =
  \min\left\{
    \Pi^{I}_{E_0;F}(\bm Z-\bm q-\bm x_I),
    a(\bm x_{\Jtan})
  \right\},
\]
where \(\Pi^{I}\) is the moving-window functional using only the active
coordinates.  If \(a(\bm x_{\Jtan})<0\), the right-hand side cannot be zero.
The set \(\{\bm x_{\Jtan}:a(\bm x_{\Jtan})=0\}\) is contained in a finite
union of coordinate hyperplanes and is therefore Lebesgue-null.  If
\(a(\bm x_{\Jtan})>0\), equality to
zero is equivalent to
\[
  \Pi^{I}_{E_0;F}(\bm Z-\bm q-\bm x_I)=0.
\]
Lemma~\ref{lem:moving_window_functional}, applied to the active coordinates,
shows that the probability of this event is zero for Lebesgue-a.e.
\(\bm x_I\).  Fubini's theorem proves the assertion in
\((\bm x_I,\bm x_{\Jtan})\), and a second application of Fubini handles the unused
coordinates.
\end{proof}

\newpage
\section{Technical estimates}
\label{sec:technical_estimates}

Throughout this appendix, \( g(t) = \min_{\bm x \ge \bm b} \bm x^\top
\Sigma^{-1}(t) \, \bm x \) denotes the inverse generalized variance introduced
in Section~\ref{sec:non_stationary_case}.

\begin{lemma}
  \label{lemma:gen_var_expansion}
  \( g ( t ) = g ( 0 ) + 2 \, t^\beta \, \bm{w}^\top A \, \bm{w} + o ( t^\beta )
  \) as \( t \to 0 \).
\end{lemma}

\begin{proof}
Put \(Q_t=\Sigma(t)^{-1}\), \(Q_0=\Sigma^{-1}\), and let \(\bb(t)\) be the
unique minimizer of \(\bm x^\top Q_t\bm x\) over \(\bm x\ge\bm b\).  Since
\(Q_t\to Q_0\) and the objectives are uniformly coercive near \(t=0\), every
convergent subsequence of \(\bb(t)\) minimizes the limiting problem.  The
minimizer of that problem is unique, hence
\[
  \bb(t)\longrightarrow\bb.
\]
Both \(\bb(t)\) and \(\bb\) are feasible for every one of the quadratic
programmes.  Consequently,
\begin{align*}
  g(t)-g(0)&\le \bb^\top(Q_t-Q_0)\bb,\\
  g(t)-g(0)&\ge \bb(t)^\top(Q_t-Q_0)\bb(t).
\end{align*}
Because \(\bb(t)\to\bb\) and \(Q_t-Q_0=O(t^\beta)\), the two bounds differ by
\(o(t^\beta)\).  Thus
\begin{equation}
  \label{eq:generalized-variance-envelope}
  g(t)-g(0)=\bb^\top(Q_t-Q_0)\bb+o(t^\beta).
\end{equation}
Assumption~\ref{D2}, used with \(s=t\), gives
\[
  \Sigma-\Sigma(t)=(A+A^\top)t^\beta+o(t^\beta).
\]
The inverse-matrix expansion therefore yields
\[
  Q_t-Q_0
  =\Sigma^{-1}(\Sigma-\Sigma(t))\Sigma^{-1}+o(t^\beta)
  =t^\beta\Sigma^{-1}(A+A^\top)\Sigma^{-1}+o(t^\beta).
\]
Substitution in \eqref{eq:generalized-variance-envelope}, followed by
\(\bm w=\Sigma^{-1}\bb\), gives
\[
  g(t)-g(0)
  =t^\beta\bm w^\top(A+A^\top)\bm w+o(t^\beta)
  =2t^\beta\bm w^\top A\bm w+o(t^\beta),
\]
which proves the claim.
\end{proof} 

\begin{lemma}
  \label{lemma:cond_mean}
  The conditional mean vector
  \begin{equation*}
    \bm{d}_\tau ( t ) 
    = ( I - R ( \tau + t, \tau ) \, \Sigma^{-1} (\tau) ) \, \bb
  \end{equation*}
  admits the following asymptotic formula
  \begin{equation*}
    \bm{d}_\tau (t)
    =
    \Big[ (\tau+t)^\beta - \tau^\beta \Big] \, A \, \bm{w}
    + t^\alpha \, V \bm{w}
    + o ( ( \tau + t )^\beta + \tau^\beta + t^\alpha )
    \quad \text{as} \quad
    \tau, t \to 0^+.
  \end{equation*}
\end{lemma}

\begin{proof}
  We have
  \begin{align*}
    \Sigma - R ( \tau + t, \tau )
    & =
    A ( \tau + t )^\beta + A^\top \tau^\beta + V t^\alpha
    + o ( (\tau + t)^\beta + \tau^\beta + t^\alpha )
    \\[7pt]
    \Sigma - \Sigma ( \tau )
    & =
    A \tau^\beta + A^\top \tau^\beta
    + o ( \tau^\beta ).
  \end{align*}
  Therefore,
  \begin{equation*}
    \Sigma(\tau) - R ( \tau + t, \tau ) 
    =
    A \Big[ ( \tau + t )^\beta - \tau^\beta \Big]
    + V t^\alpha
    + o ( (\tau + t)^\beta + \tau^\beta + t^\alpha ).
  \end{equation*}
  Therefore, with \( \Delta = \Sigma ( \tau ) - R ( \tau + t, \tau ) \) we have
  \begin{align*}
    I - R ( \tau + t, \tau ) \, \Sigma^{-1} ( \tau )
    & = I - ( \Sigma ( \tau ) - \Delta ) \, \Sigma^{-1} ( \tau )
    = \Delta \, \Sigma^{-1} ( \tau )
    \\[7pt]
    & =
    \Delta \Sigma^{-1} + o ( \| \Delta \| )
    = 
    \Big[ ( \tau + t )^\beta - \tau^\beta \Big] \, A \Sigma^{-1}
    + t^\alpha V \Sigma^{-1}
    + o ( ( \tau + t )^\beta + \tau^\beta + t^\alpha ).
  \end{align*}
  Hence,
  \begin{equation*}
    \bm{d}_\tau (t)
    =
    \Big[ (\tau+t)^\beta - \tau^\beta \Big] \, A \, \bm{w}
    + t^\alpha \, V \bm{w}
    + o ( ( \tau + t )^\beta + \tau^\beta + t^\alpha ).
    \qedhere
  \end{equation*}
\end{proof} 

\begin{lemma}
  \label{lemma:cond_cov}
  The conditional covariance matrix
  \begin{equation*}
    K_\tau ( t, s )
    =
    R ( \tau + t, \tau + s ) 
    - R ( \tau + t, \tau ) \, \Sigma^{-1} (\tau ) \, R ( \tau, \tau + s )
  \end{equation*}
  admits the following asymptotic formula:
  \begin{equation*}
    K_\tau ( t, s )
    =
    S_{\alpha,V}(t)+S_{\alpha,V}(-s)-S_{\alpha,V}(t-s)
    + o ( \tau^\beta + (\tau + t)^\beta + (\tau +s)^\beta + t^\alpha + s^\alpha + |t-s|^\alpha )
  \end{equation*}
  as \( \tau, t, s \to 0^+ \). In particular, for \(t\ge s\ge0\) the leading
  term is
  \[
    Vt^\alpha+V^\top s^\alpha-V(t-s)^\alpha .
  \]
\end{lemma}

\begin{proof}
It is enough to consider \(t\ge s\ge0\); the case \(s>t\) follows from
\(K_\tau(t,s)=K_\tau(s,t)^\top\).  Set
\[
  \Delta_1=\Sigma(\tau)-R(\tau+t,\tau+s),\qquad
  \Delta_2=\Sigma(\tau)-R(\tau+t,\tau),\qquad
  \Delta_3=\Sigma(\tau)-R(\tau,\tau+s).
\]
Subtracting the diagonal expansion
\(\Sigma-\Sigma(\tau)=(A+A^\top)\tau^\beta+o(\tau^\beta)\) from the three
instances of Assumption~\ref{D2} gives
\begin{align*}
  \Delta_1
  &=A\big((\tau+t)^\beta-\tau^\beta\big)
    +A^\top\big((\tau+s)^\beta-\tau^\beta\big)
    +V(t-s)^\alpha+o(r_{\tau,t,s}),\\
  \Delta_2
  &=A\big((\tau+t)^\beta-\tau^\beta\big)
    +Vt^\alpha+o(r_{\tau,t,s}),\\
  \Delta_3
  &=A^\top\big((\tau+s)^\beta-\tau^\beta\big)
    +V^\top s^\alpha+o(r_{\tau,t,s}),
\end{align*}
where
\[
  r_{\tau,t,s}=\tau^\beta+(\tau+t)^\beta+(\tau+s)^\beta
  +t^\alpha+s^\alpha+|t-s|^\alpha.
\]
Here the third line follows by applying \ref{D2} to
\(R(\tau+s,\tau)\) and then transposing.  Algebraically,
\begin{align*}
  K_\tau(t,s)
  &=-\Delta_1+\Delta_2+\Delta_3
    -\Delta_2\Sigma(\tau)^{-1}\Delta_3.
\end{align*}
Since \(\Delta_2,\Delta_3=O(r_{\tau,t,s})\) and \(\Sigma(\tau)^{-1}\) stays
bounded, the last product is \(o(r_{\tau,t,s})\).  The terms containing \(A\)
and \(A^\top\) cancel exactly, leaving
\[
  K_\tau(t,s)=Vt^\alpha+V^\top s^\alpha-V(t-s)^\alpha
  +o(r_{\tau,t,s}).
\]
This is
\(S_{\alpha,V}(t)+S_{\alpha,V}(-s)-S_{\alpha,V}(t-s)\) for \(t\ge s\ge0\),
and the transpose relation gives the formula for all \(s,t\ge0\).
\end{proof} 

\begin{lemma}[Limiting process]
  \label{lemma: limiting process}
  The conditions~\ref{A1'}, \ref{A2} and \ref{A3} of Lemma~\ref{vectorPickands_functionals}
  are satisfied with
  \begin{enumerate}[topsep=5pt, itemsep=5pt]
    \item \( R_{u, \tau} ( t, s ) = R ( u^{-2 / \nu} ( \tau + t ), u^{-2 / \nu}
    ( \tau + s ) ), \) with \( \nu = \min \{ \alpha, \beta \} \)
    \item \( Q_u = [ 0, \Lambda u^{2 / \alpha - 2 / \beta} ] \) if \( \alpha <
    \beta \) and \( Q_u = \{ 0 \} \) if \( \alpha \geq \beta \),
    \item \( E = [0, S] \),
    \item \( \bm{d} ( t ) = t^\alpha \, V \bm{w} \, \1_{\alpha \leq \beta} +
    t^\beta \, A \, \bm{w} \, \1_{\alpha \geq \beta}, \)
    \item \( K ( t, s ) =
    \bigl(S_{\alpha,V}(t)+S_{\alpha,V}(-s)-S_{\alpha,V}(t-s)\bigr)
    \, \1_{\alpha \leq \beta} \),
  \end{enumerate}
  for \(s,t\in[0,S]\); in particular, for \( t > s > 0 \) this is
  \( ( V t^\alpha + V^\top s^\alpha - V (t-s)^\alpha )
    \, \1_{\alpha \leq \beta} \). 
\end{lemma}

\begin{proof}
Write
\[
  q_{u,\tau}=u^{-2/\nu}\tau,
  \qquad h_u(t)=u^{-2/\nu}t,
  \qquad \nu=\min\{\alpha,\beta\}.
\]
Then \(\Sigma_{u,\tau}=\Sigma(q_{u,\tau})\).  We repeatedly use the elementary
bound
\begin{equation}
  \label{eq:power-increment-bound}
  0\le (x+y)^\beta-x^\beta
  \le C_\beta\bigl(y^\beta+\1_{\{\beta>1\}}x^{\beta-1}y\bigr),
  \qquad x,y\ge0.
\end{equation}

Suppose first that \(\alpha<\beta\).  Then \(\nu=\alpha\) and, uniformly for
\(\tau\in Q_u\),
\[
  0\le q_{u,\tau}\le \Lambda u^{-2/\beta},
  \qquad 0\le h_u(t)\le Su^{-2/\alpha}.
\]
The diagonal part of Assumption~\ref{D2} therefore gives
\[
  \sup_{\tau\in Q_u}u\|\Sigma(q_{u,\tau})-\Sigma\|
  =O(u^{-1})\longrightarrow0,
\]
which is Assumption~\ref{A1'}.  By
\eqref{eq:power-increment-bound}, uniformly for \(t\in[0,S]\),
\begin{align*}
  u\bigl((q_{u,\tau}+h_u(t))^\beta-q_{u,\tau}^\beta+h_u(t)^\alpha\bigr)
  &\longrightarrow0,\\
  u^2\bigl((q_{u,\tau}+h_u(t))^\beta-q_{u,\tau}^\beta\bigr)
  &\longrightarrow0,\\
  u^2h_u(t)^\alpha&=t^\alpha.
\end{align*}
Moreover, the quantities multiplying the little-oh remainders in
Lemmas~\ref{lemma:cond_mean} and~\ref{lemma:cond_cov}, after multiplication by
\(u^2\), remain bounded, while their arguments tend to zero uniformly.  Those
lemmas now give Assumptions~\ref{A2.1}--\ref{A2.3} with
\[
  \bm d(t)=t^\alpha V\bm w,
  \qquad
  K(t,s)=S_{\alpha,V}(t)+S_{\alpha,V}(-s)-S_{\alpha,V}(t-s).
\]

If \(\alpha\ge\beta\), then \(Q_u=\{0\}\), \(\nu=\beta\), and
\(q_{u,0}=0\).  Thus Assumption~\ref{A1'} is immediate.  With
\(h_u(t)=u^{-2/\beta}t\), Lemmas~\ref{lemma:cond_mean} and
\ref{lemma:cond_cov} yield, uniformly on \([0,S]\),
\begin{align*}
  u^2\bm d_0(h_u(t))
  &\longrightarrow t^\beta A\bm w
    +t^\alpha V\bm w\,\1_{\{\alpha=\beta\}},\\
  u^2K_0(h_u(t),h_u(s))
  &\longrightarrow
    \bigl(S_{\alpha,V}(t)+S_{\alpha,V}(-s)-S_{\alpha,V}(t-s)\bigr)
    \1_{\{\alpha=\beta\}}.
\end{align*}
Also
\(u\|\Sigma-R(h_u(t),0)\|\to0\), so
Assumption~\ref{A2.1} holds.  These limits are exactly the stated choices of
\(\bm d\) and \(K\).

Finally, Assumption~\ref{D3} gives in every regime, since \(\gamma\ge\nu\),
\[
  u^2\E{\left|\bm X(q_{u,\tau}+h_u(t))
                  -\bm X(q_{u,\tau}+h_u(s))\right|^2}
  \lesssim u^{2-2\gamma/\nu}|t-s|^\gamma
  \lesssim_S |t-s|^\nu.
\]
This is Assumption~\ref{A3} and completes the verification.
\end{proof} 

The following lemma extending the Bonferroni inequality is a consequence of
assumption~\ref{B1}.

\begin{lemma}
  \label{lemma Bonferroni adapted for sojourns}
  Let \( \Gamma \) be a functional satisfying Assumption~\ref{B1}, and let
  \(E_1,\dots,E_k\in\mathcal K\), \(k\geq2\), be compact sets. For
  \(\bm f\in C(\cup_{i=1}^kE_i,\R^d)\) and \(L\in\mathcal L\), put
  \[
    A_i=\{\Gamma_{E_i}(\bm f)>L\},\qquad
    C_i=\left\{\sup_{t\in E_i}\min_{j=1,\dots,d}f_j(t)>0\right\}.
  \]
  Then
  \[
    \bigcup_{i=1}^k A_i
    \subset
    \left\{\Gamma_{\cup_{i=1}^kE_i}(\bm f)>L\right\}
    \subset
    \bigcup_{i=1}^k A_i \cup
    \bigcup_{1\leq i<j\leq k}(C_i\cap C_j).
  \]
  Consequently, if \(\bm f\) is random, then
  \begin{align*}
    \pk{\Gamma_{\cup_{i=1}^k E_i}(\bm f)>L}
    &\leq
    \sum_{i=1}^k\pk{\Gamma_{E_i}(\bm f)>L}
    +\sum_{1\leq i<j\leq k}\pk{C_i\cap C_j},
    \\
    \pk{\Gamma_{\cup_{i=1}^k E_i}(\bm f)>L}
    &\geq
    \sum_{i=1}^k\pk{\Gamma_{E_i}(\bm f)>L}
    -\sum_{1\leq i<j\leq k}\pk{C_i\cap C_j}.
  \end{align*}
\end{lemma}
\begin{proof}
  The first inclusion follows by repeated application of the first implication
  in Assumption~\ref{B1}. For the second inclusion we prove the contrapositive
  by induction on \(k\). The case \(k=2\) follows from the second implication
  in Assumption~\ref{B1}, applied once with \((A,B)=(E_1,E_2)\) and once with
  \((A,B)=(E_2,E_1)\). For the induction step, set
  \(F_{k-1}=\cup_{i=1}^{k-1}E_i\). If
  \(\Gamma_{F_{k-1}\cup E_k}(\bm f)>L\) and none of the events \(A_i\) occurs,
  Assumption~\ref{B1} applied to \((F_{k-1},E_k)\) gives either
  \(\Gamma_{F_{k-1}}(\bm f)>L\) or \(C_k\). In the first case the induction
  hypothesis yields \(C_i\cap C_j\) for some \(1\leq i<j\leq k-1\). In the
  second case, applying Assumption~\ref{B1} to \((E_k,F_{k-1})\) and using that
  \(A_k\) does not occur gives
  \(\sup_{t\in F_{k-1}}\min_j f_j(t)>0\), hence \(C_i\) for some
  \(i<k\); therefore \(C_i\cap C_k\) occurs. This proves the deterministic
  inclusions. The two probability inequalities are then the ordinary
  Bonferroni inequalities, together with
  \(\Gamma_{E_i}(\bm f)>L\Rightarrow C_i\), which is the last implication in
  Assumption~\ref{B1}.
\end{proof}

In the following three lemmas, write
\( 
H(S) \coloneqq H_{\Gamma,L,\alpha,V,\bm 0,\Sigma,\bm b}([0,S])
\) 
for a Pickands-type preconstant with \( I \neq \varnothing \). Such constants
are represented by the integral in Lemma~\ref{vectorPickands_functionals}, with
\(\bm{X}_{u,\tau}(t)=\bm X(u^{-2/\alpha}t)\), \(Q_u=\{0\}\), and a stationary
Gaussian representer \(\bm X\) satisfying the local covariance expansion.  This
representation does not use the constant-reduction property~\ref{B5}; when
\ref{B5} is available it can also be rewritten in the shorter form of
Corollary~\ref{corollary1}. In the proofs below \(\bm X\) denotes any such
representer.

\begin{lemma}
  \label{lemma:finiteness-of-H}
  The family \(S^{-1}H(S)\), \(S\ge1\), is bounded from above.
\end{lemma}

\begin{proof}
Set
\[
  a_u=u^{-|I|}\varphi_\Sigma(u\bb),
  \qquad
  p_u(S)=\pk{\Gamma_{[0,Su^{-2/\alpha}]}
                  (\uu(\bm X-u\bm b))>L_u}.
\]
For every fixed \(S>0\), affine scaling and
Lemma~\ref{vectorPickands_functionals} give
\begin{equation}
  \label{eq:local-H-asymptotic-finiteness}
  p_u(S)\sim H(S)a_u.
\end{equation}
Partition the scaled interval \([0,S]\) into at most \(\lceil S\rceil\) unit
subintervals.  The upper inequality in
Lemma~\ref{lemma Bonferroni adapted for sojourns} bounds \(p_u(S)\) by the sum
of the corresponding one-block functional probabilities plus the sum of
pairwise simultaneous high-excursion probabilities.  By stationarity and the
fixed-block version of Lemma~\ref{vectorPickands_functionals}, each one-block
term is \(O(a_u)\).  Lemma~\ref{lemma:double_sum_bound_stationary}, applied
with unit blocks, gives \(O(a_u)\) for adjacent pairs and
\(O(e^{-c m^\alpha}a_u)\) for pairs separated by \(m\) unit blocks.  Therefore
\[
  p_u(S)\le C\lceil S\rceil a_u
  +Ca_u\sum_{k=1}^{\lceil S\rceil}
       \sum_{m\ge0}e^{-cm^\alpha}
  \le C' S a_u,
  \qquad S\ge1,
\]
where a possible final truncated subinterval is enlarged to a unit interval.
Divide by \(a_u\), let \(u\to\infty\) in
\eqref{eq:local-H-asymptotic-finiteness}, and obtain \(H(S)\le C'S\), uniformly
for \(S\ge1\).
\end{proof}

\begin{lemma}
  \label{lemmma:positivity-of-H}
  If \(H(S_0)>0\) for some \(S_0>0\), then
  \(
    \liminf_{S \to \infty} S^{-1} H ( S ) > 0.
  \)
\end{lemma}

\begin{proof}
Put
\[
  a_u=u^{-|I|}\varphi_\Sigma(u\bb),
  \qquad r_u=u^{2/\alpha}a_u,
  \qquad
  p_u=\pk{\Gamma_{[0,1]}(\uu(\bm X-u\bm b))>L_u}.
\]
Fix \(R\ge S_0\), and retain, inside \([0,1]\), every other block of length
\(Ru^{-2/\alpha}\):
\[
  B_{j,u}=u^{-2/\alpha}[2jR,(2j+1)R].
\]
There are \(N_u\sim u^{2/\alpha}/(2R)\) such blocks.  The union of the events
\(\{\Gamma_{B_{j,u}}>L_u\}\) is contained in the event defining \(p_u\), by
Assumption~\ref{B1}.  Ordinary Bonferroni, the local asymptotic on a block, and
Lemma~\ref{lemma:double_sum_bound_stationary} therefore give
\begin{align*}
  \liminf_{u\to\infty}\frac{p_u}{r_u}
  &\ge \frac{H(R)}{2R}
    -C\sum_{m\ge1}e^{-c((2m-1)R)^\alpha}\\
  &\ge \frac{H(R)}{2R}-Ce^{-c'R^\alpha}.
\end{align*}
Here the sum covers pairs whose physical separation tends to zero; pairs at a
fixed positive physical separation are exponentially negligible by
Assumption~\ref{R1} and the standard two-copy Gaussian tail bound.
Since \(H(R)\ge H(S_0)>0\), one may choose \(R\) so large that the last
quantity, denoted by \(\eta\), is strictly positive.

Now partition \([0,1]\) into consecutive blocks of length
\(Su^{-2/\alpha}\), for an arbitrary fixed \(S\ge1\).  The upper Bonferroni
bound, the one-block asymptotic, and the adjacent/separated estimates in
Lemma~\ref{lemma:double_sum_bound_stationary} imply
\begin{equation}
  \label{eq:common-horizon-upper-positivity}
  \limsup_{u\to\infty}\frac{p_u}{r_u}
  \le \frac{H(S)}{S}+\epsilon(S),
  \qquad
  \epsilon(S)\le C\bigl(S^{-1/2}+e^{-cS^{\alpha/2}}\bigr).
\end{equation}
The terminal remainder contributes only \(O(a_u)=o(r_u)\).  Combining the
positive lower bound with
\eqref{eq:common-horizon-upper-positivity} yields
\[
  \frac{H(S)}{S}\ge \eta-\epsilon(S).
\]
Letting \(S\to\infty\) proves the asserted positive lower limit.
\end{proof} 

\begin{lemma}
    \label{lemma: lim H(S) / S in (0, infty) if alpha < beta}
    The limit \( \lim_{S \to \infty} S^{-1} H(S) \) exists and is finite; it is
    positive whenever the non-triviality condition \(H(S_0)>0\) holds for some \(S_0>0\).
\end{lemma}

\begin{proof}
Use the common-horizon notation
\[
  a_u=u^{-|I|}\varphi_\Sigma(u\bb),
  \qquad r_u=u^{2/\alpha}a_u,
  \qquad
  p_u=\pk{\Gamma_{[0,1]}(\uu(\bm X-u\bm b))>L_u}.
\]
For any fixed \(S\ge1\), partition \([0,1]\) into blocks of length
\(Su^{-2/\alpha}\).  The two Bonferroni inequalities, the fixed-block
asymptotic, and Lemma~\ref{lemma:double_sum_bound_stationary} give
\begin{equation}
  \label{eq:common-horizon-two-sided}
  \frac{H(S)}{S}-\epsilon(S)
  \le \liminf_{u\to\infty}\frac{p_u}{r_u}
  \le \limsup_{u\to\infty}\frac{p_u}{r_u}
  \le \frac{H(S)}{S}+\epsilon(S),
\end{equation}
where
\[
  \epsilon(S)\le C\bigl(S^{-1/2}+e^{-cS^{\alpha/2}}\bigr)
  \longrightarrow0.
\]
Indeed, adjacent block pairs give the \(S^{-1/2}\) term, all remaining local
pairs give an exponentially small term, pairs at a fixed positive physical
separation are negligible by Assumption~\ref{R1}, and the terminal remainder
is \(o(r_u)\).

Apply \eqref{eq:common-horizon-two-sided} with \(S=S_1\) and \(S=S_2\).  It
follows that
\[
  \frac{H(S_1)}{S_1}
  \le \frac{H(S_2)}{S_2}+\epsilon(S_1)+\epsilon(S_2),
\]
and the reverse inequality follows after interchanging \(S_1\) and \(S_2\).
Thus \(H(S)/S\) is Cauchy as \(S\to\infty\), so the limit exists.  Its
finiteness follows from Lemma~\ref{lemma:finiteness-of-H}; under the stated
non-triviality assumption, its positivity follows from
Lemma~\ref{lemmma:positivity-of-H}.
\end{proof} 

In the following lemma, \(H(\Lambda)\) denotes either the Piterbarg
constant truncated at \([0,\Lambda]\),
\begin{equation*}
  H(\Lambda)
  =H_{\Gamma,L,\alpha,V_{I,I},\bm d_I,\Sigma,\bm b}([0,\Lambda]),
  \qquad \alpha=\beta,
\end{equation*}
or the deterministic truncated constant
\begin{equation*}
  H(\Lambda)
  =
  \int_{\R^d}
      \exp \left(
        \bm{x}_I^\top \bm{w}_I
        - \frac{1}{2} \bm{x}_J^\top ( \Sigma^{-1} )_{J J} \bm{x}_J
      \right)
      \Ind{
        \Gamma_{[0, \Lambda]} \left(
          -\bm{d}_I - \bm{x}_I,
          -\bm{x}_{\Jtan},
          ( \bb - \bm{b} )_{\Jsl}
        \right)
        > L
      }\,d\bm{x},
  \qquad \alpha>\beta,
\end{equation*}
where \(\bm d(t)=t^\beta A\bm w\). In both cases the relevant
Assumptions~\ref{B3}--\ref{B4} and non-triviality assumptions are those of
Theorem~\ref{thm:main_theorem}.

\begin{lemma}[Convergence of Piterbarg and deterministic constants]
    \label{lemma: lim H(S) in (0, infty) if alpha > beta}
    The limit \( \lim_{\Lambda \to \infty} H(\Lambda) \) exists and is finite;
    it is positive under the stated non-triviality condition.
\end{lemma}

\begin{proof}
  By Assumption~\ref{B1}, \(H(\Lambda)\) is non-decreasing in \(\Lambda\). The
  non-triviality assumption gives \(H(\Lambda_0)>0\) for some \(\Lambda_0>0\),
  so the limit, if finite, is positive.

  We next prove boundedness. For fixed \(\Lambda\), Lemmas~\ref{lemma: limiting process}
  and~\ref{vectorPickands_functionals}, applied on \([0,\Lambda]\), give
  \begin{equation*}
    \pk{\Gamma_{[0, \Lambda u^{-2/\beta}]}(\uu(\bm X-u\bm b))>L_u}
    \sim H(\Lambda)u^{-|I|}\varphi_\Sigma(u\bb).
  \end{equation*}
  By Assumption~\ref{B1}, the event on the left implies
  \(\{\exists t\in[0,\Lambda u^{-2/\beta}]:\bm X(t)>u\bm b\}\). Choose a fixed
  \(\Lambda_0\) large enough for Lemma~\ref{log layer bound lemma}. The local
  vector Pickands bound on \([0,\Lambda_0u^{-2/\beta}]\), together with the
  log-layer estimate on \([\Lambda_0u^{-2/\beta},T]\), yields
  \begin{equation*}
    \pk{\exists t\in[0,T]:\bm X(t)>u\bm b}
    \lesssim u^{-|I|}\varphi_\Sigma(u\bb).
  \end{equation*}
  Hence \(H(\Lambda)\) is bounded uniformly in \(\Lambda\). Since it is
  non-decreasing, it has a finite limit, and the preceding positivity argument
  completes the proof.
\end{proof}

Define for \( \tau, \lambda, S \in \R_+ \) the double events' probabilities by
\begin{equation*}
  P_{\bm{b}} ( \tau, \lambda, S )
  =
  \pk{
  \begin{aligned}
    & \exists\, t \in u^{-2 / \alpha} [ \tau, \tau + S ] : & \bm{X} ( t ) > u \bm{b}
    \\
    & \exists\, s \in u^{-2 / \alpha} [ \lambda, \lambda + S ] : & \bm{X} ( s ) > u \bm{b}
  \end{aligned} 
  }.
\end{equation*}

\begin{lemma}[Double sum bound for stationary case]
  \label{lemma:double_sum_bound_stationary}
  There exist \( \delta > 0 \) and \(c>0\) such that, for all
  \(0 < \lambda - \tau + S < \delta u^{2 / \alpha}\),
  \( \lambda - \tau \ge S \), and \( S \geq 1 \),
  \begin{equation*}
    \frac{P_{\bm{b}} ( \tau, \lambda, S )}{u^{-|I|} \varphi_{\Sigma} ( u \bb )}
    \lesssim
    \begin{cases}
      S^{1/2}+S e^{-cS^{\alpha/2}}, & \lambda-\tau=S,\\[3pt]
      S \exp\left(-c(\lambda-\tau-S)^\alpha\right),
        & \lambda-\tau>S,
    \end{cases}
  \end{equation*}
\end{lemma}

\begin{proof}
Write \(a_u=u^{-|I|}\varphi_\Sigma(u\bb)\) and
\(d=\lambda-\tau\).  The event for the full vector is contained in the
corresponding event for the active subvector \(\bm X_I\), and the standard
quadratic-programming tail asymptotic gives
\[
  \pk{\bm X_I(0)>u\bm b_I}\asymp a_u.
\]
Assumption~\ref{R1}, now in its positive-definite form, and the local expansion
in~\ref{R2} are exactly the one-parameter assumptions used in
\cite[Lemma~5.1]{VVGP}.  Applied to two unit intervals, that lemma gives,
after changing constants,
\begin{equation}
  \label{eq:stationary-unit-double-bound}
  \pk{\text{an exceedance on each of two unit scaled blocks at gap }g}
  \le C a_u e^{-cg^\alpha},
\end{equation}
uniformly while the total physical lag is below a sufficiently small fixed
number.  Bounded gaps are covered by the same estimate after enlarging \(C\),
using the fixed-block Pickands upper bound.

Assume first that \(d>S\) and put \(g=d-S\).  Split both scaled intervals into
unit subintervals.  Summing
\eqref{eq:stationary-unit-double-bound} over the resulting pairs and grouping
pairs according to their additional separation gives the elementary
convolution estimate
\[
  \sum_{k,l}e^{-c(d+l-k-1)_+^\alpha}
  \le C S e^{-c'g^\alpha}.
\]
Consequently
\[
  P_{\bm b}(\tau,\lambda,S)
  \le C S e^{-c'(\lambda-\tau-S)^\alpha}a_u.
\]
The restriction \(d+S<\delta u^{2/\alpha}\) ensures that all physical lags
involved lie in the local range in which the cited estimate is uniform.

If \(d=S\), split the second interval at \(\lambda+S^{1/2}\).  The probability
of an exceedance in its first part is
\(O(S^{1/2}a_u)\), by a union over unit blocks.  The remaining part is
separated from the first interval by a gap \(S^{1/2}\), so the already proved
separated estimate is
\(O(S e^{-cS^{\alpha/2}}a_u)\).  Adding the two terms proves the adjacent
bound.
\end{proof} 

\begin{lemma}[Double sum bound for non-stationary case]
  \label{lemma:double_sum_bound}
  Let \( \bm{X} \) be a centered continuous Gaussian process that satisfies
  assumptions~\ref{D1}--\,\ref{D3} with \( \alpha < \beta \) and \( \Lambda > 0
  \). Then for \( 0 \leq \tau, \lambda \leq \Lambda u^{2 / \alpha - 2 / \beta}\)
  satisfying \( \lambda - \tau \ge S \) with \(u\) and \(S\) large enough we
  have
  \begin{equation*}
    \pk{
      \begin{aligned}
        & \exists\, t \in u^{-2 / \alpha}  [\tau, \tau+S] : 
        & \bm{X} ( t ) > u \bm{b}
        \\
        & \exists\, s \in u^{-2 / \alpha} [\lambda, \lambda+S] : 
        & \bm{X} ( s ) > u \bm{b}
      \end{aligned} 
    }
    \lesssim_\Lambda
    u^{-|I|} \,
    \varphi_{\Sigma} ( u \bb )
    \times
    \begin{cases}
      S^{1 / 2} + S e^{-c S^{\alpha / 2}}, & \text{ if } \lambda - \tau = S, \\
      S \, e^{-c (\lambda - \tau - S)^\alpha}, & \text{ if } \lambda - \tau > S.
    \end{cases}
  \end{equation*}
  with some \( c > 0 \) independent of \( u \) and \( S \).
\end{lemma}

\begin{proof}
Put \(a_u=u^{-|I|}\varphi_\Sigma(u\bb)\).  We first record the uniform
unit-block estimate.  Restricting to the active coordinates, Assumption~\ref{D2}
is the one-dimensional locally additive covariance expansion used in
\cite[Lemma~12]{VVGF}; Assumption~\ref{D3} supplies its uniform increment
bound.  Lemmas~\ref{lemma:gen_var_expansion} and
\ref{lemma: limiting process} give, uniformly for
\(0\le\tau,\lambda\le\Lambda u^{2/\alpha-2/\beta}\), the required boundary
density comparison and conditional local limit.  Hence the cited lemma yields
\begin{equation}
  \label{eq:nonstationary-unit-double-bound}
  \pk{\text{an exceedance on each of two unit scaled blocks at gap }g}
  \le C_\Lambda a_u e^{-cg^\alpha}.
\end{equation}
The possible non-stationary variance penalty only decreases the right-hand
side and has therefore been discarded.

For \(\lambda-\tau>S\), split both length-\(S\) intervals into unit
subintervals and sum
\eqref{eq:nonstationary-unit-double-bound}.  The same convolution estimate as
in the stationary proof gives
\[
  \pk{\text{both length-\(S\) blocks contain an exceedance}}
  \le C_\Lambda S e^{-c(\lambda-\tau-S)^\alpha}a_u.
\]
For \(\lambda-\tau=S\), split the second interval into
\([\lambda,\lambda+S^{1/2}]\) and
\([\lambda+S^{1/2},\lambda+S]\).  The one-block upper estimate bounds the
first contribution by \(C_\Lambda S^{1/2}a_u\); the second pair of intervals
has gap \(S^{1/2}\), so the separated estimate bounds it by
\(C_\Lambda S e^{-cS^{\alpha/2}}a_u\).  This proves both cases.
\end{proof}

\begin{lemma}[Log-layer]
  \label{log layer bound lemma}
  Let \( \bm{X} \) be a Gaussian process that satisfies
  assumptions~\ref{D1}--\,\ref{D3}. Then there exist positive constants \(
  \Lambda_0 > 0 \) and \( u_0 > 0 \) such that for all \( u \geq u_0 \) and \(
  \Lambda \geq \Lambda_0 \)
  \[
  \pk{ \exists\, t \in [\Lambda u^{-\frac{2}{\beta}}, T] : \bm{X}(t) > u \bm{b}} 
  \lesssim 
  \exp\left( -c \Lambda^\beta \right) 
  u^{(2/\alpha-2/\beta)_+-|I|} \,
  \varphi_{\Sigma} ( u \bb )
  \]
  holds with some \( c > 0 \) independent of \( u \) and \( \Lambda \).
\end{lemma} 

\begin{proof}
Let \(\nu=\min\{\alpha,\beta\}\) and
\(a_u=u^{-|I|}\varphi_\Sigma(u\bb)\).  By
Lemma~\ref{lemma:gen_var_expansion} and the uniqueness in Assumption~\ref{D1},
there are \(c_0,\eta,\delta>0\) such that
\[
  g(t)\ge g(0)+c_0t^\beta\quad(0\le t\le\delta),
  \qquad
  g(t)\ge g(0)+\eta\quad(\delta\le t\le T).
\]
A standard compact-set Gaussian quadratic-programming/Piterbarg bound makes
the contribution of \([\delta,T]\) exponentially smaller than the claimed
right-hand side.

Set
\[
  r_u=u^{-2/\beta}(\log u)^{2/\beta}.
\]
We first consider the near layer
\([\Lambda u^{-2/\beta},r_u]\), when it is non-empty, and cover it by
\(I_{k,u}=u^{-2/\nu}[k,k+1]\).  Conditioning at the left endpoint of a block,
using Assumptions~\ref{D2}--\ref{D3}, and applying
Lemma~\ref{PavelsIntegral} to the conditional field gives the uniform local
upper estimate
\begin{equation}
  \label{eq:log-layer-local-block}
  \pk{\exists t\in I_{k,u}:\bm X(t)>u\bm b}
  \le C a_u
    \exp\{-c k^\beta u^{2-2\beta/\nu}\}.
\end{equation}
For completeness, on this layer
\(u\|\Sigma(t)-\Sigma\|=O(u^{-1}(\log u)^2)=o(1)\), so the active Gaussian
density prefactor is uniformly comparable with \(a_u\).  The generalized
variance expansion supplies the exponential factor in
\eqref{eq:log-layer-local-block}.  The conditional integral is uniformly
bounded when \(\nu<\beta\); when \(\nu=\beta\), its possible growth is at most
\(\exp\{C(1+k^{(\beta-1)_+})\}\), by the power-increment bound and
Lemma~\ref{PavelsIntegral}, and is absorbed by \(e^{-ck^\beta}\) after reducing
\(c\).

With
\(K_u=\lceil\Lambda u^{2/\nu-2/\beta}\rceil\), summing
\eqref{eq:log-layer-local-block} and extending the sum to infinity gives
\[
  \sum_{k\ge K_u}e^{-c k^\beta u^{2-2\beta/\nu}}
  \le C e^{-c'\Lambda^\beta}
  \begin{cases}
    1,&\nu=\beta,\\
    u^{2/\nu-2/\beta},&\nu<\beta.
  \end{cases}
\]
Thus the near layer has the required order.

It remains to treat
\([\max\{\Lambda u^{-2/\beta},r_u\},\delta]\).  The vector Piterbarg
inequality on this fixed neighbourhood, together with the lower bound for
\(g\), gives for some fixed \(M\)
\[
  \pk{\exists t\text{ in this interval}:\bm X(t)>u\bm b}
  \le C u^M a_u
  \exp\{-c u^2\max(\Lambda u^{-2/\beta},r_u)^\beta\}.
\]
Since the exponent is at least
\(c\max\{\Lambda^\beta,(\log u)^2\}\), the polynomial factor is absorbed after
reducing \(c\).  Combining this estimate, the near-layer sum, and the
exponentially negligible interval \([\delta,T]\), and observing that
\(2/\nu-2/\beta=(2/\alpha-2/\beta)_+\), proves the lemma.
\end{proof}

\begin{lemma}
  \label{PavelsIntegral}
  For every \(\varepsilon>0\), \(\gamma\in(0,2]\), and
  \(\bm\Lambda\in\R_{\ge0}^k\), there exist constants
  \(\delta,c,C>0\) such that the following holds.  Let
  \[
    \left\{
      (\bm\chi_{\bm x}(\bm t))_{\bm t\in[\bm0,\bm\Lambda]}:
      \bm x\in\R^d
    \right\}
  \]
  be a family of continuous \(\R^d\)-valued Gaussian random fields for which
  the map
  \[
    \bm x\longmapsto
    \pk{\exists\,\bm t\in[\bm0,\bm\Lambda]:
      \bm\chi_{\bm x}(\bm t)>\bm x}
  \]
  is measurable.  Suppose that there exist \(G,\sigma\ge0\) and
  \(\bm w=(w_1,\ldots,w_d)^\top\), with
  \(\varepsilon\le w_i\le\varepsilon^{-1}\) for every \(i\), such that for
  every \(\bm x\in\R^d\) and every
  \(\bm t,\bm s\in[\bm0,\bm\Lambda]\),
  \begin{enumerate}[itemsep=0.3cm, topsep=0.3cm, label={\alph*)}]
    \item
    \[
      \sup_{F\subset\{1,\ldots,d\}}
      \sup_{\bm t\in[\bm0,\bm\Lambda]}
      \bm w_F^\top\E{\bm\chi_{\bm x,F}(\bm t)}
      \le G+\delta\sum_{j=1}^d|x_j|;
    \]
    \item
    \[
      \sup_{F\subset\{1,\ldots,d\}}
      \sup_{\bm t\in[\bm0,\bm\Lambda]}
      \Var{\bm w_F^\top\bm\chi_{\bm x,F}(\bm t)}
      \le\sigma^2;
    \]
    \item
    \[
      \sup_{F\subset\{1,\ldots,d\}}
      \Var{
        \bm w_F^\top\bm\chi_{\bm x,F}(\bm t)
        -\bm w_F^\top\bm\chi_{\bm x,F}(\bm s)
      }
      \le\varepsilon^{-1}|\bm t-\bm s|^\gamma.
    \]
  \end{enumerate}
  Then
  \[
    \int_{\R^d}e^{\bm w^\top\bm x}
    \pk{\exists\,\bm t\in[\bm0,\bm\Lambda]:
      \bm\chi_{\bm x}(\bm t)>\bm x}
    \,d\bm x
    \le C e^{c(G+\sigma^2)}.
  \]
\end{lemma}

\begin{proof}
Put
\[
  A(\bm x)=\{\exists\bm t\in[\bm0,\bm\Lambda]:
  \bm\chi_{\bm x}(\bm t)>\bm x\},
\]
and decompose \(\R^d\) into the orthants indexed by
\(F=\{i:x_i>0\}\).  The orthant \(F=\varnothing\) contributes at most
\(\prod_iw_i^{-1}\), so fix \(F\ne\varnothing\) and write
\[
  p=\bm w_F^\top\bm x_F,
  \qquad n=|\bm x_{F^c}|_1.
\]
On this orthant, \(A(\bm x)\) implies
\[
  \sup_{\bm t}\bm w_F^\top\bm\chi_{\bm x,F}(\bm t)>p.
\]
The centered scalar field on the left has variance at most \(\sigma^2\) and
increment metric bounded by
\(\varepsilon^{-1/2}|\bm t-\bm s|^{\gamma/2}\).  Dudley's entropy bound on the
fixed rectangle and Borell--TIS therefore give constants \(C_0,c_0,C_1\),
independent of \(\bm x\), such that
\begin{equation}
  \label{eq:pavel-integral-borell-bound-revised}
  \pk{A(\bm x)}
  \le C_0\exp\left\{-c_0
  \frac{\bigl(p-G-\delta(|\bm x_F|_1+n)-C_1(1+\sigma)\bigr)_+^2}
       {1+\sigma^2}\right\}.
\end{equation}
This form is also valid when \(\sigma=0\).

Because \(w_i\in[\varepsilon,\varepsilon^{-1}]\),
\(|\bm x_F|_1\le\varepsilon^{-1}p\) and
\(\bm w_{F^c}^\top\bm x_{F^c}\le-\varepsilon n\).  Choose \(\delta>0\),
depending only on the fixed parameters, so small that
\(\delta\varepsilon^{-1}\le1/4\) and \(\varepsilon/(4\delta)>2\), and put
\(A_0=G+C_1(1+\sigma)+1\).  Split the orthant into
\[
  D_1=\{\delta n\le p/4+A_0\},
  \qquad D_2=D_1^c.
\]
On \(D_1\), the positive part in
\eqref{eq:pavel-integral-borell-bound-revised} is bounded below by
\(p/2-2A_0\).  Under the change of variables \(y_i=w_ix_i\), \(i\in F\),
the section \(\sum_{i\in F}y_i=p\) has volume bounded by a constant times
\(p^{|F|-1}\).  The negative coordinates have the integrable weight
\(e^{-\varepsilon n}\).  Completing the square in \(p\) therefore gives
\[
  \int_{D_1}e^{\bm w^\top\bm x}\pk{A(\bm x)}\,d\bm x
  \le C_F e^{c_F(G+\sigma^2)}.
\]

On \(D_2\) use only \(\pk{A(\bm x)}\le1\).  Radial simplex coordinates for
the negative variables give a polynomial factor times \(e^{-\varepsilon n}\).
Since \(n>(p/4+A_0)/\delta\), the negative-coordinate integral is bounded by
\(C\exp\{-\varepsilon(p/4+A_0)/\delta\}\).  After multiplication by the
positive-coordinate weight \(e^p\), this becomes
\(C\exp\{-[\varepsilon/(4\delta)-1]p-cA_0\}\), which is integrable against
\(p^{|F|-1}\,dp\) by the choice of \(\delta\).  Hence
\[
  \int_{D_2}e^{\bm w^\top\bm x}\pk{A(\bm x)}\,d\bm x
  \le C_F\le C_F e^{c_F(G+\sigma^2)}.
\]
Summing over the finitely many orthants proves the lemma.
\end{proof}

\section*{Declaration of generative AI and AI-assisted technologies in the manuscript preparation process}

During the preparation of this work the authors used ChatGPT~5.5 and
Claude~Opus~4.8 in order to edit the text, improve its style, and revise it
for potential mistakes or sloppy writing.  After using these tools, the
authors reviewed and edited the content as needed and take full
responsibility for the content of the published article.

\end{document}